\documentclass[12pt]{iopart}
\usepackage{iopams}  
  \expandafter\let\csname equation*\endcsname\relax
  \expandafter\let\csname endequation*\endcsname\relax
\usepackage{amsmath}
\usepackage{amsthm}
\usepackage{color}
\usepackage{url}
\usepackage{algorithm}
\usepackage{algorithmic}
\usepackage{float}
\usepackage{hyperref} 
\usepackage{graphicx}
\usepackage{epstopdf}
\usepackage{epsfig,subfigure}
\usepackage{epsf}

\newtheorem{theorem}{Theorem}[section]

\newtheorem{lemma}{Lemma}

\newtheorem{remark}{Remark}

\newcommand{\bfone}{\mathbf{1}}
\newcommand{\bfzero}{\mathbf{0}}
\newcommand{\sigmam}{\sigma_{\mathrm{L}}}

\newcommand{\bfc}{\mathbf{c}}
\newcommand{\bfd}{\mathbf{d}}
\newcommand{\bfdo}{\mathbf{d}_{\mathrm{obs}}}

\newcommand{\bfk}{\mathbf{k}}
\newcommand{\bfm}{\mathbf{m}}
\newcommand{\bfmo}{\mathbf{m}_{\mathrm{0}}}
\newcommand{\bfn}{\mathbf{n}}

\newcommand{\bfr}{\mathbf{r}}

\newcommand{\bfs}{\mathbf{s}}

\newcommand{\bfu}{\mathbf{u}}
\newcommand{\bfy}{\mathbf{y}}
\newcommand{\bfz}{\mathbf{z}}
\newcommand{\bfzi}{\mathbf{z}_i}

\newcommand{\Cd}{C_{\mathbf d}}
\newcommand{\Cy}{C_{\mathbf y}}
\newcommand{\CL}{C_L}
\newcommand{\Cm}{C_{\mathbf m}}
\newcommand{\Wd}{W_{\mathbf d}}
\newcommand{\WL}{W_{L}}
\newcommand{\Wm}{W_{L}}

\newcommand{\bfzeta}{\mbox{\boldmath{$\zeta$}}}

\newcommand{\argmin}[1]{\textnormal{arg} \min_{#1}}

\newcommand{\figdirA}

\begin{document}
\title[Underdetermined parameter estimation by the $\chi^2$ principle] {Regularization Parameter Estimation for Underdetermined problems by the $\chi^2$ principle with application to  $2D$  focusing gravity inversion}
\author{Saeed Vatankhah$^1$,  Rosemary A Renaut$^2$ and Vahid E Ardestani,$^1$}
\date{\today}
\address{$^1$Institute of Geophysics, University of  Tehran,Tehran, Iran, $^2$ School of Mathematical and Statistical Sciences, Arizona State University, Tempe, USA }
\eads{\mailto{svatan@ut.ac.ir}, \mailto{renaut@asu.edu}, \mailto{ebrahim@ut.ac.ir}}    
\begin{abstract}
The $\chi^2$-principle generalizes the Morozov discrepancy principle to the augmented residual of the Tikhonov regularized least squares problem. 
For weighting of the data fidelity by a known Gaussian noise distribution on the measured data and, when  the stabilizing, or regularization, term is considered to be weighted by unknown inverse covariance information on the model parameters,   the minimum of the Tikhonov   functional  becomes a random variable that follows a $\chi^2$-distribution with $m+p-n$ degrees of freedom for the model matrix $G$ of size $m \times n$ and regularizer $L$ of size $p\times n$. Here it is proved that the result  holds for the underdetermined case, $m<n$ provided  that $m+p\ge n$  and that the null spaces of the operators do not intersect.    
A Newton root-finding algorithm  is used to find the   regularization parameter $\alpha$ which yields the optimal inverse covariance weighting in the case of a white noise assumption on the mapped model data. It is implemented for small-scale problems using the generalized singular value decomposition, or singular value decomposition when $L=I$.  Numerical results  verify the algorithm for the case of  regularizers approximating zero to second order derivative approximations,  contrasted    with the methods of generalized cross validation and unbiased predictive risk estimation. The inversion of underdetermined $2D$ focusing gravity data  produces models with non-smooth properties, for which typical implementations in this field use the iterative   minimum support stabilizer and both regularizer and regularizing parameter are updated   each iteration.   For  a simulated data set with noise, the regularization parameter estimation methods for underdetermined data sets are used in this iterative framework, also contrasted with the   L-curve and the Morozov Discrepancy principle.  These experiments   demonstrate the efficiency and robustness of the $\chi^2$-principle in this context, moreover showing that the L-curve and Morozov Discrepancy Principle are outperformed in general by  the three other techniques. Furthermore, the minimum support stabilizer is of general use for the $\chi^2$-principle when implemented without the desirable knowledge of a mean value of the model.

\end{abstract}
\submitto{Inverse Problems}

\maketitle
\section{Introduction}
We discuss the solution of numerically ill-posed and underdetermined systems of equations, $\bfd = G\bfm$. Here $G\in\mathcal{R}^{m \times n}$, with  $m<n$, is the matrix resulting from the discretization of a forward operator which maps from the parameter or model space to the data space, given respectively by the discretely sampled vectors $\bfm\in\mathcal{R}^n$, and $\bfd\in\mathcal{R}^m$. We assume that  the measurements of the  data $\bfd$ are error-contaminated, $\bfd_{\mathrm{obs}}=\bfd+\bfn$ for noise vector $\bfn$.  Such problems often arise from the discretization of a Fredholm integral equation of the first kind, with a kernel possessing an exponentially decaying spectrum that is responsible for the ill-posedness of the problem. Extensive literature on the solution of such problems  is available in standard literature, e.g. \cite{ABT,Engl, Hansen:98,tarantola,Vogel:02,Zhd}. 

A well-known approach for finding an acceptable solution to the ill-posed problem is to augment the data fidelity term,   $\|\Wd(G\bfm-\bfdo)\|_2^2$, here measured in a weighted  $L_2$ norm\footnote{Here we use the standard  definition for the weighted norm of the vector $\bfy$, $\|\bfy\|^2_W:=\bfy^TW\bfy$.},   by a stabilizing regularization term for the model parameters, $\|L(\bfm-\bfm_0)\|_2^2$, yielding the Tikhonov objective function 
\begin{align}\label{objective}
P^{\alpha}(\bfm):= \|\Wd(G \bfm-\bfdo)\|_2^2 + \alpha^2 \|L(\bfm-\bfm_0)\|_2^2.
\end{align}
Here $\alpha$ is the regularization parameter which trades-off between the two terms, $\Wd$ is a data weighting matrix, $\bfm_0$ is a given reference vector of \textit{a priori} information for the model $\bfm$, and the choice of $L\in\mathcal{R}^{p\times n}$ impacts the basis for the solution $\bfm$. The Tikhonov regularized solution, dependent on $\alpha$, is given by 
\begin{align}\label{Tiksoln}
\bfm_{\mathrm{Tik}}(\alpha)= \argmin{\bfm}\{P^{\alpha}(\bfm)\}.
\end{align} 
If  $\Cd=(\Wd^T\Wd)^{-1}$   is  the data covariance matrix,  and  we assume 
 white noise   for the mapped model parameters $L\bfm$ so that the model covariance matrix is $\CL=\sigmam^2I=\alpha^{-2}I=(\WL^T\WL)^{-1}$, then \eqref{objective}  is
 \begin{align}\label{objective2}
P^{\sigmam}(\bfm)= \|G \bfm-\bfdo\|_{\Cd^{-1}}^2 +  \|L(\bfm-\bfm_0)\|_{\CL^{-1}}^2.
\end{align}
Note  we will use in general the notation $\bfy \sim \mathbb{N}(\hat{\bfy}, \Cy)$ to indicate that $\bfy$ is normally distributed with mean $\hat{\bfy}$ and symmetric positive definite (SPD)   covariance matrix $\Cy$.  Using $\CL=\alpha^{-2}I$  in  \eqref{objective2} permits an assumption of white noise in the estimation for $L\bfm$ and thus statistical interpretation of the regularization parameter $\alpha^2$ as  the inverse of the white noise  variance. 

The determination of an optimal $\alpha$   is a topic of much previous research and includes methods such as the L-curve (LC) \cite{Hansen:92}, generalized cross validation (GCV) \cite{GoHeWa}, the unbiased predictive risk estimator (UPRE) \cite{Marq,ThKaTi}, the residual periodogram (RP) \cite{hakikj:06,rust}, and the Morozov discrepancy principle  (MDP) \cite{morozov}, all of which are well described in the literature, see e.g.  \cite{Hansen:98,Vogel:02}  for comparisons of  the criteria and further references. The motivation and assumptions for these methods varies; while the  UPRE requires that  statistical information on the noise in the measurement data be provided,  for the MDP it is sufficient to have an estimate of the overall error level in the data. More recently, a new approach based on the $\chi^2$ property of the functional \eqref{objective2} under  statistical assumptions applied through $\CL$,  was proposed by \cite{mead:08}, for the overdetermined case  $m>n$ with $p=n$.  The extension to the case with $p\le n$ and a discussion of effective numerical algorithms is given in \cite{mere:09, RHM10}, with also consideration of the case when $\hat{\bfm}=\bfmo$ is not available. Extensions for nonlinear problems \cite{MH13}, inclusion of inequality  constraints \cite{MR10} and for multi parameter assumptions \cite{mead:13} have also been considered.

The fundamental premise of the $\chi^2$ principle for estimating $\sigmam$ is that provided the noise distribution on the measured data is available, through knowledge of $\Cd$, such that the weighting on the model and measured parameters is as in \eqref{objective2}, and that the mean value $\hat{\bfm}$ is known, then $P^{\sigmam}(\bfm_{\mathrm{Tik}}(\sigmam))$ is a random variable following a $\chi^2$ distribution with $m+p-n$ degrees of freedom,   $P^{\sigmam}(\bfm_{\mathrm{Tik}}(\sigmam))\sim \chi^2(m+p-n, c)$, with centrality parameter $c=0$. Thus the expected value satisfies $\hat{P}^{\sigmam}(\bfm_{\mathrm{Tik}}(\sigmam)) =m+p-n$.  As a result   $P^{\sigmam}(\bfm_{\mathrm{Tik}}(\sigmam))$  lies within an interval centered around its expected value, which facilitates the development of the Newton root-finding algorithm for the optimal $\sigmam$ given in \cite{mere:09}. This algorithm has the advantage as compared to other techniques of being very fast for finding the unique $\sigmam$, provided the root exists, requiring generally no more than $10$ evaluations of $P^{\sigmam}(\bfm_{\mathrm{Tik}}(\sigmam))$ to converge to a reasonable estimate.  The algorithm in \cite{mere:09} was presented for small scale problems in which one can use the singular value decomposition (SVD) \cite{GoLo:96}  for matrix $G$ when $L=I$ or the generalized singular value decomposition (GSVD) \cite{PaigeSau1} of the matrix pair $[\Wd G; L]$.  For the large scale case an approach using the Golub-Kahan iterative bidiagonalization based on the LSQR algorithm \cite{PaigeSau2, PaigeSau3} was presented in \cite{RHM10} along with the extension of the algorithm for the non-central destribution of $P^{\sigmam}(\bfm_{\mathrm{Tik}}(\sigmam))$, namely when $\bfmo$ is unknown but may be estimated from a set of measurements.  In this paper the   $\chi^2$ principle  is first extended to the estimation of $\sigmam$ for underdetermined problems,  specifically for the central $\chi^2$ distribution   with known $\bfmo$, with the proof of the result in Section~\ref{theory} and examples in Section~\ref{otherresults}.  
                                                         
In many cases the smoothing that arises when  the basis mapping operator $L$     approximates low order derivatives  is unsuitable for handling material properties that vary over relatively short distances, such as in the inversion of gravity data produced by localized sources. A stabilizer that does not penalize sharp boundaries is instead preferable. This can be achieved by setting the regularization term in a different  norm, as for example using  the total variation, \cite{ROF},   already significantly studied and applied for geophysical inversion e.g. \cite{Ring:99,SGR}. Similarly,the minimum support (MS) and minimum gradient support (MGS) stabilizers provide solutions with non-smooth properties and were introduced for geophysical inversion in \cite{PoZh:99} and \cite{Zhd}, respectively.  We note now the relationship of the iterated MS stabilizer with non stationary iterated Tikhonov regularization, \cite{HG98} in which the solution is iterated to convergence with a fixed operator $L$, but updated residual and iteration dependent $\alpha^{(k)}$ which is forced to zero geometrically. The technique was  extended for example for image deblurring  in \cite{DH13} with $\alpha^{(k)}$ found using a version of the MDP  and dependent on a good preconditioning approximation for the square model matrix. More generally,  the iterative  MS stabilizers, with both $L$ and $\alpha$  iteration dependent, are closely related to the  iteratively reweighted norm (IRN) approximation for the Total Variation norm introduced and analyzed in \cite{WoRo:07}.
 In this paper the MS stabilizer is used to reconstruct non-smooth models  for  the  geophysical problem of  gravity data inversion with the updating regularization parameter found using the most often applied techniques of MDP and L-curve, contrasted with the UPRE, GCV and the $\chi^2$ principle. These results also show that the $\chi^2$ principle can be  applied without  knowledge of $\bfmo$ through the MS iterative process.  Initialization with $\bfmo =0$ is contrasted with an initial stabilizing choice determined by the spectrum of the operator.

The outline of this paper is  as follows. In Section~\ref{theory} the theoretical development of the $\chi^2$ principle for the underdetermined problem is presented. We note that the proof is stronger than that used in the original literature when $m>n$ and hence improves the general result. The algorithm uses the GSVD (SVD) at each iteration and leads to a Newton-based algorithm for estimating the regularization parameter. A review of other standard techniques for parameter estimation is presented in Section~\ref{other} and numerical examples contrasting these with the $\chi^2$ approach also given, Section~\ref{otherresults}. The MS stabilizer is described in Section~\ref{MSstab} and numerical experiments contrasting the impact of the choice of the regularization parameter within the MS algorithm for the problem of $2D$ gravity inversion in Section~\ref{results}. Conclusions and future work are discussed in Section~\ref{sec:conc}.
\section{Theoretical Development}\label{theory}
Although the proof of the result on the degrees of freedom for the underdetermined case $m<n$  effectively follows the ideas introduced \cite{mere:09,RHM10}, the modification presented here provides a stronger result which can also strengthen the result for the overdetermined case, $m\ge n$.
\subsection{$\chi^2$ distribution for the underdetermined case}
We first assume that it is possible to solve the  normal equations 
\begin{align}\label{normal}
(G^T\Wd^T\Wd G  + L^T\WL^T\WL L) \bfy = G^T\Wd^T\Wd \bfr, \,\, \bfr :=\bfdo-G\bfmo, \,\, \bfy=\bfm-\bfmo
\end{align}
for the shifted system associated with \eqref{objective2}. The \textbf{invertibility} condition for \eqref{normal} requires that  $\tilde{L}:=\WL L$ and $\tilde{G}:=\Wd G$  have null spaces which do not intersect 
\begin{align}\label{invert}
\fl \mathcal{N}(\WL L) \cap \mathcal{N}(\Wd G) =0.
\end{align}
Moreover, we also assume $m+p\ge n$ which is realistic when  $L$ approximates a  derivative operator of order $l$, then $p=n-l$, and typically $l$ is small, $n-l\ge 0$.

 Following \cite{mere:09} we first find the functional $P^{\WL}(\bfm_{\mathrm{Tik}}(\WL))$ where $\bfm_{\mathrm{Tik}}(\WL) = \bfy(\WL)+\bfmo$ and $\bfy(\WL)$ solves \eqref{normal},  for  general   $\WL$.  There are many  definitions for the GSVD in the literature,  differing  with respect to the ordering of the singular decomposition terms, but all effectively equivalent to the original GSVD introduced in \cite{PaigeSau1}.   For ease of presentation we introduce $\mathbf{1}_k$ and  $\mathbf{0}_k$ to be the vectors of length $k$ with $1$, respectively $0$, in all rows, and define $q=n-m\ge0$.
We use the GSVD as stated in \cite{ABT}. 
\begin{lemma}[GSVD]\label{gsvdlemma}
Suppose $H:=[\tilde{G}; \tilde{L}]$, where   $\tilde{G}$  has size $m \times n$, $m<n$, $\tilde{L}$ has  size $p \times n$, $p\le n$ with both $\tilde{G}$ and $\tilde{L}$ of full row rank $m$, and $p$, respectively, and by \eqref{invert} that $H$ has full column rank $n$.
The generalized singular value decomposition  for $H$ is
\begin{align}\label{gsvd}
[\tilde{G}; \tilde{L}] &= [U\tilde{\Upsilon} X^T; V\tilde{M}X^T] \\
\tilde{\Upsilon} = \left[0_{m\times q}   \ \ \Upsilon  \right],  \,\, \Upsilon &=\mathrm{diag}(\nu_{q+1}, \dots, \nu_{p},\bfone_{n-p})\in \mathcal{R}^{m\times m}, \, \nu_i=1, i=p+1:n,\\
\tilde{M}=\left[M \ \ 0_{p \times (n-p)}  \right], \,\, M&=\mathrm{diag}(\bfone_q,\mu_{q+1}, \dots, \mu_{p})\in \mathcal{R}^{p\times p}, \, \mu_i=1, i=1:q, \\
0 < \nu_{q+1} \le \dots \le \nu_{p} < 1, & \quad  1 >  \mu_{q+1} \ge \dots \ge \mu_{p} > 0, \quad \nu_i^2+\mu_i^2=1.
\end{align} \
Matrices $U\in\mathcal{R}^{m \times m}$ and  $V\in\mathcal{R}^{p \times p}$ are orthogonal, $U^TU=I_m$,  $V^TV=I_p$, and $X \in \mathcal{R}^{n \times n}$ is invertible; $X^{-1}$ exists.
\end{lemma}
\begin{remark}\label{gsvrange}The indexing  in matrices $M$ and $\Upsilon$ uses the column index and we use the 
  definitions $\mu_i=0$, $i=p+1:n$ and $\upsilon_i=0$, $i=1:q$.  The generalized singular values are given by $\gamma_i=\upsilon_i/\mu_i$, $i=1:n$. Of these $n-p$ are infinite, $m+p-n$ are finite and non-zero, and $q$ are zero. 
\end{remark}

We first introduce $\tilde{\bfr}:=\Wd\bfr$ and note the relations 
\begin{align*}
\tilde{\Upsilon}^T\tilde{\Upsilon}&+\tilde{M}^T\tilde{M}  =I_n, \qquad
\tilde{G}^T\tilde{G}+\tilde{L}^T\tilde{L} = XX^T,  \qquad   \tilde{\Upsilon}\tilde{\Upsilon}^T =\Upsilon\Upsilon^T,\qquad \mathrm{and}\\
\bfy&=(X^T)^{-1} \tilde{\Upsilon}^TU^T\tilde{\bfr},\qquad \tilde{G}\bfy=U\tilde{\Upsilon}\tilde{\Upsilon}^TU^T\tilde{\bfr},\qquad \tilde{L}\bfy=V\tilde{M}\tilde{\Upsilon}^TU^T\tilde{\bfr}.
\end{align*}
Thus with $\bfs=U^T\tilde{\bfr}$, with indexing from $q+1:n$ for $\bfs$ of length $m$, $s_i=\bfu_{i-q}^T\tilde{\bfr}$,
\begin{align}\nonumber
P^{\WL}(\bfm_{\mathrm{Tik}}(\WL))&=\tilde{\bfr}^TU(I_m- \tilde{\Upsilon}\tilde{\Upsilon}^T)U^T\tilde{\bfr} = \sum_{i=q+1}^{p} \mu_i^2 s_i^2 
=\|\bfk\|_2^2,\\\bfk&=QU^T\Wd\bfr, \quad Q:=\mathrm{diag}(\mu_{q+1},\dots,\mu_{p},\bfzero_{n-p}).\label{veck}
\end{align}
To obtain our desired result on $P^{\WL}(\bfm_{\mathrm{Tik}}(\WL))$ as a random variable we investigate the statistical distribution of the components for $\bfk$, following \cite[Theorem 3.1]{mere:09} and \cite[Theorem 1]{RHM10} for  the cases of a central, and non-central distribution, respectively, but with  modified assumptions that lead to a stronger result.  
\begin{theorem}[central and non-central  $\chi^2$ distribution of : $P^{\WL}(\bfm_{\mathrm{Tik}}(\WL)$]\label{newtheorem}
Suppose  $\bfn \sim \mathbb{N}(0, \Cd)$, $\bfzeta:=(\bfm-\bfmo) \sim \mathbb{N}(\hat{\bfm}, \Cm  )$,   $L\bfzeta:=L(\bfm-\bfmo) \sim \mathbb{N}(\hat{\bfm}, \CL )$, 
 the invertibility condition  \eqref{invert}, and that $m+p-n>0$ is sufficiently large that limiting distributions for the $\chi^2$ result  hold. Then for 
\begin{enumerate}
\item  $\bfmo=\hat{\bfm}$:  $P^{\WL}(\bfm_{\mathrm{Tik}}(\WL))\sim \chi^2(I_{m+p-n}, 0)$.
\item  $\bfmo\ne\hat{\bfm}$:  $P^{\WL}(\bfm_{\mathrm{Tik}}(\WL))\sim \chi^2({m+p-n}, c)$, $c=\|QU^T\Wd G(\hat{\bfm}-\bfmo)\|_2^2:=\|\bfc\|_2^2$. 
\end{enumerate} 
Equivalently the minimum value of the functional $P^{\WL}(\bfm_{\mathrm{Tik}}(\WL))$  is a random variable which follows a $\chi^2$ distribution with $m+p-n$ degrees of freedom and centrality parameter $c=\|QU^T\Wd G(\hat{\bfm}-\bfmo)\|_2^2$.  
\end{theorem}
\begin{proof}
By \eqref{veck} it is sufficient to examine the components $k_i$, $i=q+1, \dots, p$ to demonstrate $\|\bfk\|^2$ is a sum of normally distributed components with mean $\bfc$ and then employ the limiting argument to yield the $\chi^2$ distribution, as in \cite{mere:09,RHM10}.  First observe that $\bfd \sim \mathbb{N}(G\hat{\bfm}, \Cd+G\Cm G^T)$, thus  $\bfr=\bfdo-G\bfmo=\bfd+\bfn-G\bfmo=G(\bfm-\bfmo)+\bfn \sim \mathbb{N}(G(\hat{\bfm}-\bfmo), \Cd+G\Cm G^T)$,  $\Wd \bfr \sim \mathbb{N}(\Wd G(\hat{\bfm}-\bfmo), \Wd (\Cd +G\Cm G^T)\Wd^T)$, and $\bfk \sim \mathbb{N}(QU^T\Wd G(\hat{\bfm}-\bfmo), QU^T\Wd (\Cd +G\Cm G^T)\Wd^TUQ^T)$. The result for the central parameter $c$ is thus immediate. For the covariance we have
\begin{align}\label{vark}
C_{\mathbf{k}}&=QU^T\Wd (\Cd +G\Cm G^T)\Wd^TUQ^T = QQ^T + Q\tilde{\Upsilon}X^T \Cm X \tilde{\Upsilon}^TQ^T.
\end{align}
By assumption,  $L$ has full row rank and $\Cm$ is SPD, thus for $\CL:=L\Cm L^T$ we can define $\WL:=\sqrt{\CL^{-1}}$. Therefore
\begin{align}\nonumber
I_p=\WL \CL \WL^T 
= \WL L\Cm L^T \WL^T&= \tilde{L} \Cm \tilde{L}^T = V\tilde{M}X^T \Cm X \tilde{M}^T V^T \quad \mathrm{implies}\\
\tilde{M} (X^T \Cm X) \tilde{M}^T&=I_p.\label{import}
\end{align}
Introduce pseudoinverses for $\tilde{M}$ and $\tilde{M}^T$,  denoted by superscript $\dagger$,  and a dimensionally-consistent block decomposition for $(X^T\Cm X)$, in which $C_{11}$ is of size $p\times p$, 
\begin{align}\label{decompC}
X^T\Cm X &= \left(\begin{array}{cc} C_{11}&C_{12} \\C_{21}&C_{22}  \end{array}\right), \quad 
\tilde{M}^\dagger = I_n \left(\begin{array}{c} M^{-1} \\ 0 \end{array} \right) I_p, \quad 
(\tilde{M}^T)^{\dagger}  = I_p \left( \begin{array}{cc} M^{-1} & 0 \end{array} \right )I_n.
\end{align}
Then applying to \eqref{import}
\begin{align*}
\tilde{M}^\dagger (\tilde{M}^T)^{\dagger} &= \left(\begin{array}{cc} M^{-2} & 0 \\0 & 0 \end{array}\right) 
= \tilde{M}^\dagger \tilde{M} (X^T \Cm X) \tilde{M}^T  (\tilde{M}^T)^{\dagger} \quad \mathrm{yielding} \\
\left(\begin{array}{cc} M^{-2} & 0 \\0 & 0 \end{array}\right) & = \left(\begin{array}{cc} I_{p} & 0\\0 & 0 \end{array}\right) (X^T \Cm X)  \left(\begin{array}{cc} I_{p} & 0\\0 & 0 \end{array}\right) 
= \left(\begin{array}{cc} C_{11} & 0\\0 & 0 \end{array}\right).
\end{align*}
Moreover,
\begin{align}\label{QU}
Q\tilde{\Upsilon} &= \left(\begin{array}{cc} M_{11}&0\\0& 0_{n-p} \end{array}\right) \left(\begin{array}{ccc} 0&\Upsilon_{11}&0\\0 &0&I_{n-p} \end{array}\right) 
= \left(\begin{array}{ccc} 0&\Upsilon_{11}M_{11}&0\\0 &0&0_{n-p}\end{array}\right)
\end{align}
where  $M_{11}:=\mathrm{diag}(\mu_{q+1}, \dots, \mu_{p})$, and $\Upsilon_{11} :=\mathrm{diag}(\nu_{q+1}, \dots, \nu_p)$. Then with a   block decomposition of $(X^T \Cm X)$, in which as compared to \eqref{decompC} now $[C_{13}, C_{12}]:=C_{12}$ and $[C_{31};C_{21}]=C_{21}$, we have for \eqref{vark}
\begin{align*}
C_{\bfk}&=QQ^T  +Q\tilde{\Upsilon}X^T \Cm X \tilde{\Upsilon}^T Q^T  \\
&=\left(\begin{array}{cc} M^2_{11}  & 0 \\ 0 & 0_{n-p}\end{array}\right) + 
\left(\begin{array}{ccc} 0&\Upsilon_{11}M_{11}&0\\0 &0&0_{n-p}\end{array}\right)
\left(\begin{array}{ccc} I_{n-m}& 0 & C_{13}\\ 0  & M^{-2}_{11}&C_{12}\\C_{31} &C_{21}&C_{22}\end{array} \right)  
\left(\begin{array}{cc} 0 & 0 \\  \Upsilon_{11}M_{11} & 0\\0 & 0_{n-p}
\end{array}\right) \\
&=\left(\begin{array}{cc} M^2_{11}  & 0 \\ 0 & 0_{n-p}\end{array}\right)  + 
\left( \begin{array}{cc} \Upsilon^2_{11} & 0 \\ 0 & 0_{n-p} \end{array}\right) = \left( \begin{array}{cc} I_{m+p-n} & 0 \\ 0 & 0_{n-p}\end{array}\right),
\end{align*} 
as required to obtain the properties of the distribution for $\|\bfk\|^2$. 
\end{proof}
\begin{remark}
Note that the result is exactly the same as given in the previous results for the overdetermined situation $m\ge n$ but now for $m<n$ with $m+p\ge n$ and without  the prior assumption on the properties for the pseudoinverse on $\CL$. Namely we   directly use the pseudoinverse  $\tilde{M}^{\dagger}$ and its transpose hence, after adapting the proof for the case with $m\ge n$, this  tightens  the results previously presented in \cite{mere:09, RHM10}.
\end{remark}
\begin{remark}
We note as in \cite[Theorem 2]{RHM10} that the theory can be extended for the case in which the filtering of the GSVD replaces uses $f_i=0$ for $\upsilon_i<\tau$ for some tolerance $\tau$, eg suppose $\upsilon_i<\tau$ for $i \le p-r$ then we have the filtered  functional 
\begin{align}
\|\bfk(\sigmam)\|_2^2&=\sum_{i=q+1}^{p-r} s_i^2 + \sum_{i=p-r+1}^{p} \frac{s_i^2}{\gamma_i^2\sigmam^2+1} := \sum_{i=q+1}^{p-r} s_i^2+\|\bfk_{\mathrm{FILT}}\|_2^2.\label{kfilter}
\end{align}
Thus we obtain $\|\bfk(\sigmam)_{\mathrm{FILT}}\|_2^2 \sim \chi^2(r,c_{\mathrm{FILT}})$, where we use $c_{\mathrm{FILT}}=\|\bfc_{\mathrm{FILT}}\|^2_2=\|\tilde{I}\bfc\|^2_2$, in which $\tilde{I}=\mathrm{diag}(\bfzero_{m-r+p-n}, \bfone_{r},\bfzero_{n-p})$ picks out the filtered components only. 
\end{remark}
\begin{remark}As already noted in the statement of Lemma~\ref{gsvdlemma} the GSVD is not uniquely defined with respect to ordering of the columns of the matrices. On the other hand it is not essential that the ordering be given as stated to use the iteration defined by \eqref{newt}. In particular, it is sufficient to identify the ordering of the spectra in matrices $\Upsilon$ and $M$, and then to assure that elements of $\bfs$ are calculated in the same order, as determined by the consistent ordering of $U$. This also applies to the statement for \eqref{kfilter} and for the use of the approach with the SVD. In particular the specific form for \eqref{kfilter} assumes that the $\gamma_i$ are ordered from small to large, in opposition to the standard ordering for the SVD. 
\end{remark}
\subsection{Algorithmic Determination of $\sigmam$}
As in \cite{RHM10} Theorem~\ref{newtheorem}  suggests finding $\Wm$ such that $\|\bfk(\Wm)\|^2$ as closely as possible follows the  $\chi^2(m+p-n,c(\Wm))$ distribution. Let $\psi(\Wm)=z_{\theta/2} \sqrt{2(m+p-n+2c(\Wm))}$ where $z_{\theta/2}$ is the relevant $z$-value for the $\chi^2$ distribution with $m+p-n$ degrees of freedom. $\theta$ defines the $(1-\theta)$ confidence interval  
\begin{align}\label{rangeP}
(m+p-n+c(\Wm)) -\psi(\Wm) \le\|\bfk(\Wm)\|_2^2\le (m+p-n+c(\Wm)) +\psi(\Wm).
\end{align}
A root finding algorithm for $c=0$ and  $\Wm=\sigmam^{-2}I$  was presented in \cite{mere:09}, and extended for $c>0$ in \cite{RHM10}. The general and difficult multi-parameter case was discussed in \cite{mead:13}, with extensions for nonlinear problems in \cite{MH13}. We collect all the parameter estimation formulae   in \ref{App}.

\subsection{Related Parameter Estimation Techniques}\label{other}
In order to assess the impact of Theorem~\ref{newtheorem} in contrast to other accepted techniques for  regularization parameter estimation we very briefly review key aspects of the related algorithms which are then contrasted in Section~\ref{otherresults}.  Details can be found in the literature, but for completeness the   necessary formulae   when implemented for the GSVD (SVD) are given in  \ref{App}, here using as consistent with \eqref{objective} $\alpha:=\sigmam^{-1}$. 

 The {\bf Morozov Discrepancy Principle (MDP)}, \cite{morozov},    is a widely used technique for gravity and magnetic field data inversion. $\alpha$ is chosen under the assumption that the norm of the weighted residual, $\|\tilde{G}\bfy(\alpha)-\tilde{\bfr}\|_2^2\sim\chi^2(\delta,0)$, where $\delta$ denotes the  number of  degrees of freedom.  For a problem of full column rank $\delta=m-n$, \cite[p. 67, Chapter 3]{ABT}. But, as also noted in
 in \cite{MH13}, this is  only valid  when $m>n$ and, as   frequently adopted in practice,  a scaled version $\delta=\rho m$, $0<\rho\le 1$, can be  used. 
The choice of $\alpha$ by   {\bf Generalized Cross Validation (GCV)} is under the premise   that  if an arbitrary measurement is removed from the data set,  then the corresponding regularized solution should be able to predict the missing observation. The GCV formulation yields a minimization which can fail when the associated objective is nearly flat, creating difficulties to compute the minimum numerically, \cite{Hansen:92}. The {\bf L-curve}, which finds $\alpha$ through the 
trade-off between the norms of the regularization   $L(\bfm -\bfm_0)$  and the weighted  residuals,  \cite{Hansen:92,Regtools}, may not be robust for problems that do not generate well-defined corners, making it difficult to find the point of  maximum curvature of the plot as a function of $\alpha$.  Indeed, when $m<n$ the curve is generally smoother and it is harder to find $\alpha_{\mathrm{opt}}$, \cite{LiOd:99,VaArRe}.  As for  GCV, the 
{\bf Unbiased Predictive Risk Estimator (UPRE)} minimizes a functional, chosen to 
 to minimize the expected value of the predictive risk \cite{Vogel:02}, and requires that information on the noise distribution in the data is provided. Apparently, there is no one approach that is likely successful in all situations. Still, 
the GSVD (SVD) can be used in each case to simplify the objectives and functionals e.g. \cite{ABT,Regtools, mere:09, Vogel:02}, hence making their repeat evaluation  relatively cheap for small scale problems, and thus of relevance for comparison in the underdetermined situation with the proposed $\chi^2$ method \eqref{newt}.

\subsection{Numerical Evaluation for Underdetermined Problems}\label{otherresults}
We first assess the efficacy of using the noted regularization parameter estimation techniques for the solution of underdetermined problems, by presentation of some illustrative results using two examples from the standard literature, namely problems \texttt{gravity} and \texttt{tomo} from the Regularization Toolbox, \cite{Regtools}.  Problem \texttt{gravity} models a 1-D gravity surveying problem  for a point source located at depth $z$ and convolution kernel $K(s,t)=1/z(z^2+(s-t)^2)^{-1.5}$. The conditioning of the problem is worse  with increasing $z$. We chose $z=.75$ as compared to the default $z=.25$ and consider the example for  data measured for the kernel integrated against the source function $f(t)=\sin(\pi t) + 0.5 \sin(2 \pi t)$. Problem \texttt{tomo}  is a two dimensional tomography problem in which each right hand side datum represents a line integral along a randomly selected straight ray penetrating a rectangular domain. Following \cite{Regtools} we embed the structure from problem \texttt{blur} as the source with the domain. These two test problems suggest two different situations for under sampled data. For \texttt{tomo},  it is clear that an undersampled problem is one in  which insufficient rays are collected; the number of available projections through the domain are limited. To generate the data we take the full problem for a given $n$, leading to right hand side samples $d_i$, $i=1:n$ and  to under sample we take those same data and use the first $m$ data points, $d_i$, $i=1:m$, $m<n$. 
For \texttt{gravity}, we again take a full set of data  for the problem of size $n$, but because of the underlying integral equation relationship for the convolution, under sampling represents sampling at a constant rate from the $d_i$, i.e. we take the right hand side data $d(1:\Delta i:n)$ for a chosen integer sample step, $\Delta i$. Because the L-curve and MDP are well-known,   we only present results contrasting UPRE, GCV and $\chi^2$.

\subsubsection{Problem \texttt{gravity}}
We take  full problem size $n=3200$ and use sampling rates $\Delta i = 1$, $2$, $4$, $8$ and $16$, leading to problems of sizes $m\times n$, $m=3200$, $1600$, $800$, $400$, $200$,  so that we can contrast the solutions of the $m<n$ case with those of the  full case  $m=n$. The mean and standard deviation of the relative error over $25$ copies of the data are taken for noise levels $\eta=0.1$ and $\eta=.01$.  Noisy data are obtained as $\bfd^c=\bfd+\eta \max(\bfd) \Theta^c$, $c=1:25$, with $\Theta^c$ sampled from standard normal distribution using Matlab function \texttt{randn}. In downsampling, $\bfd^c$ are found for the full problem, and downsampling is applied to each $\bfd^c$, hence preserving the noise across problem size.   The UPRE and GCV algorithms use $200$ points to find the minimum and the $\chi^2$ is solved with tolerance determined by $\theta=0.90$ in \eqref{rangeP}. Noise levels $\eta=.1$ and $.01$  correspond to white noise variance approximately $.01$, and $.0001$, respectively. Matrices are weighted by the assumption of white noise rather than colored noise.  The results of the mean and standard deviation of the relative error for the $25$ samples are detailed in Tables~\ref{tabgrav1}-\ref{tabgrav2} for the two noise levels, all data sample rates, and for derivative orders in the regularization of order $0$, $1$ and $2$. Some randomly selected illustrative results, at down sampling rates $1$, $2$ and $10$  for each noise level are shown in Figures~\ref{grav1fig}-\ref{grav2fig}.
\begin{figure}[!htb]
\begin{center}
\subfigure[$p=0$ $3200$ ]{\includegraphics[width=.3\textwidth]{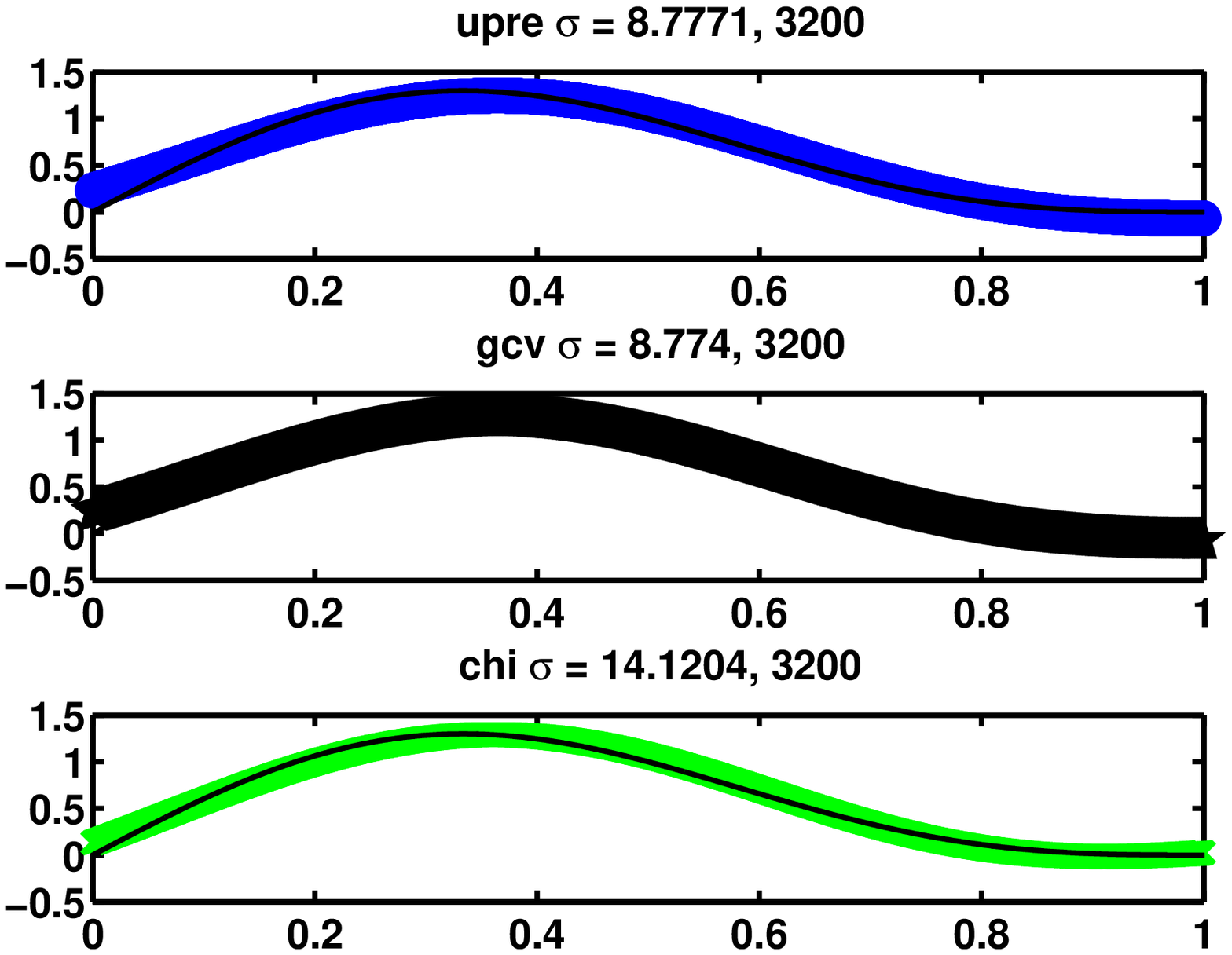}}
\subfigure[$p=0$  $1600$ ]{\includegraphics[width=.3\textwidth]{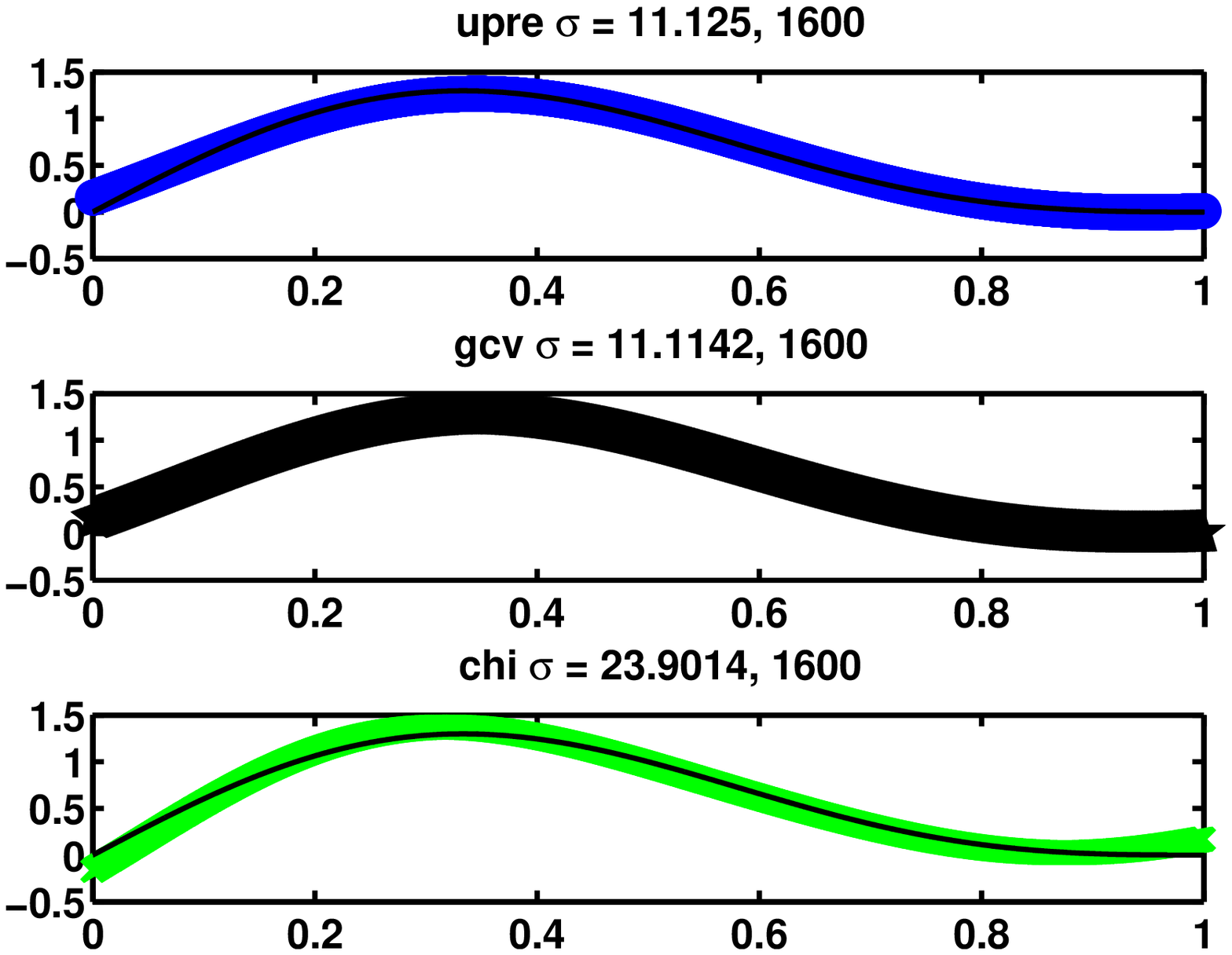}}
\subfigure[$p=0$ $200$ ]{\includegraphics[width=.3\textwidth]{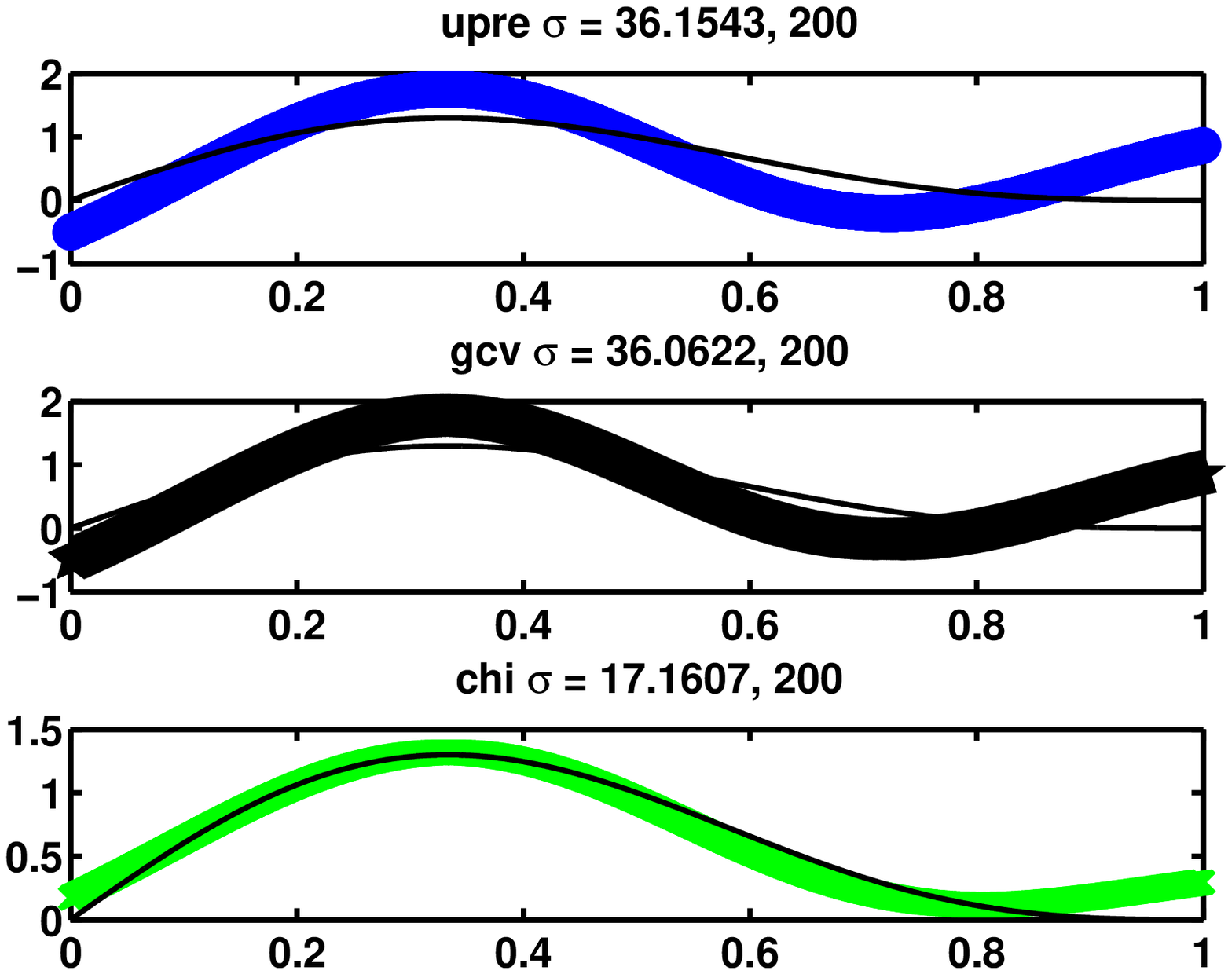}}\\
\subfigure[$p=1$  $3200$ ]{\includegraphics[width=.3\textwidth]{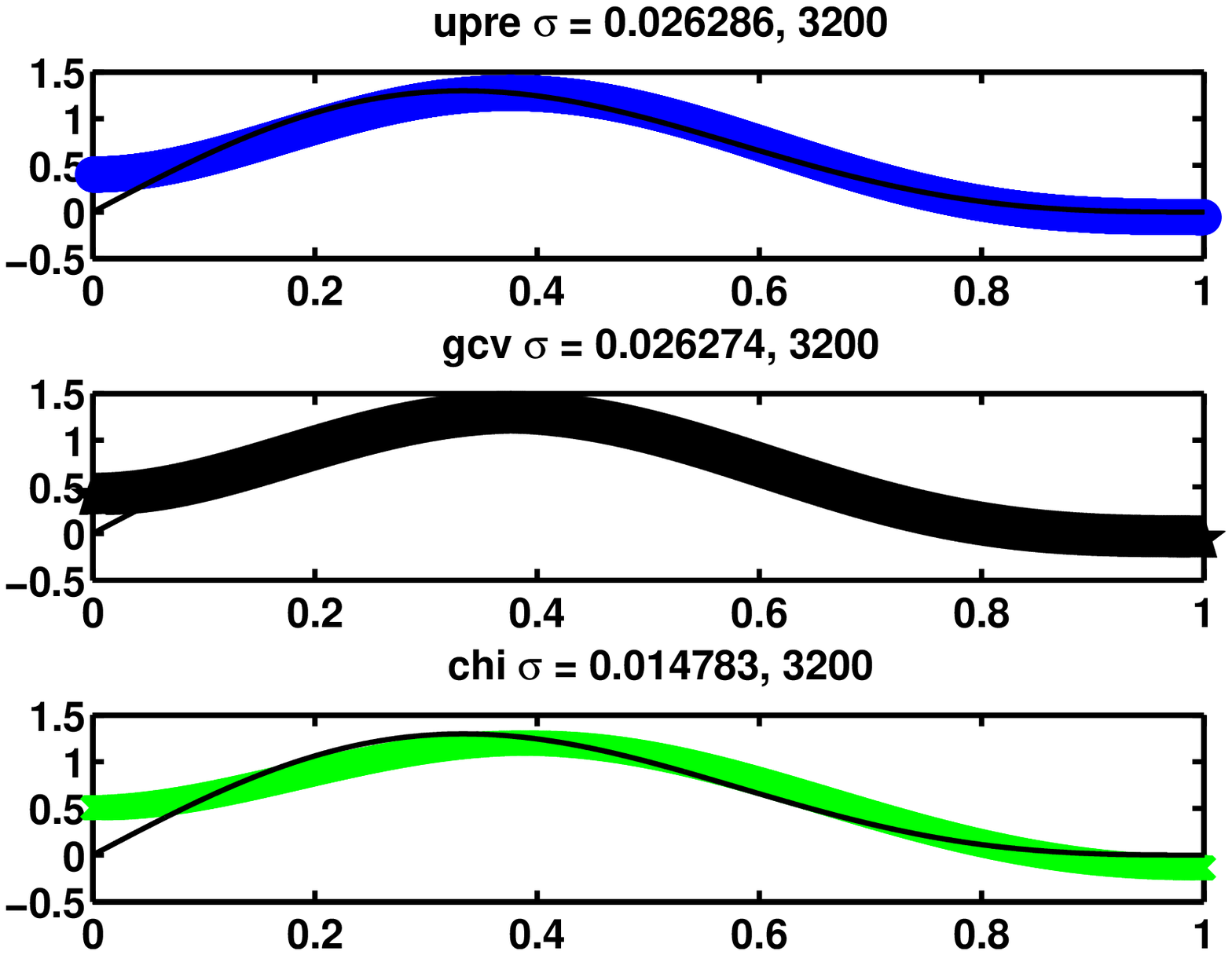}}
\subfigure[$p=1$  $1600$ ]{\includegraphics[width=.3\textwidth]{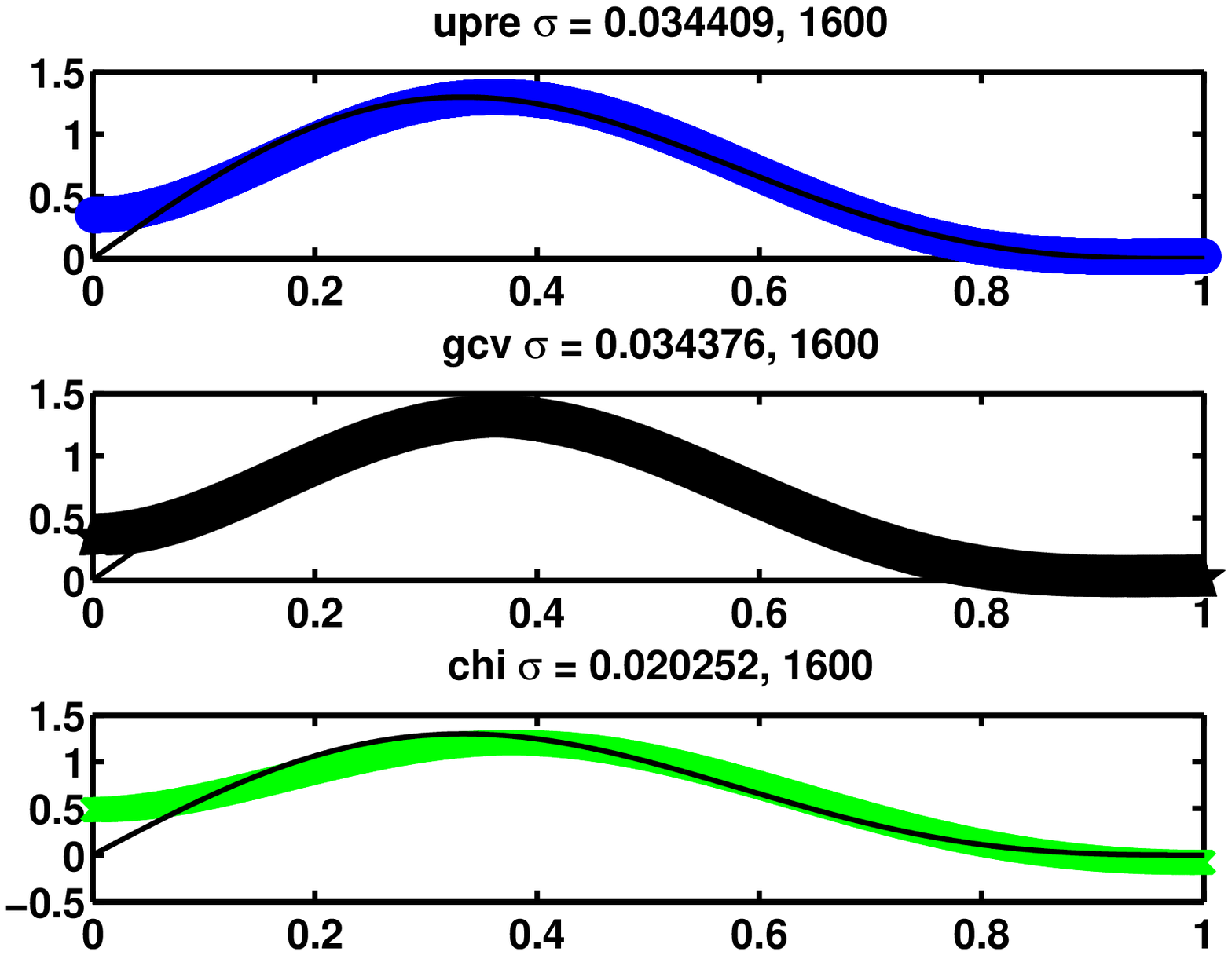}}
\subfigure[$p=1$ $200$ ]{\includegraphics[width=.3\textwidth]{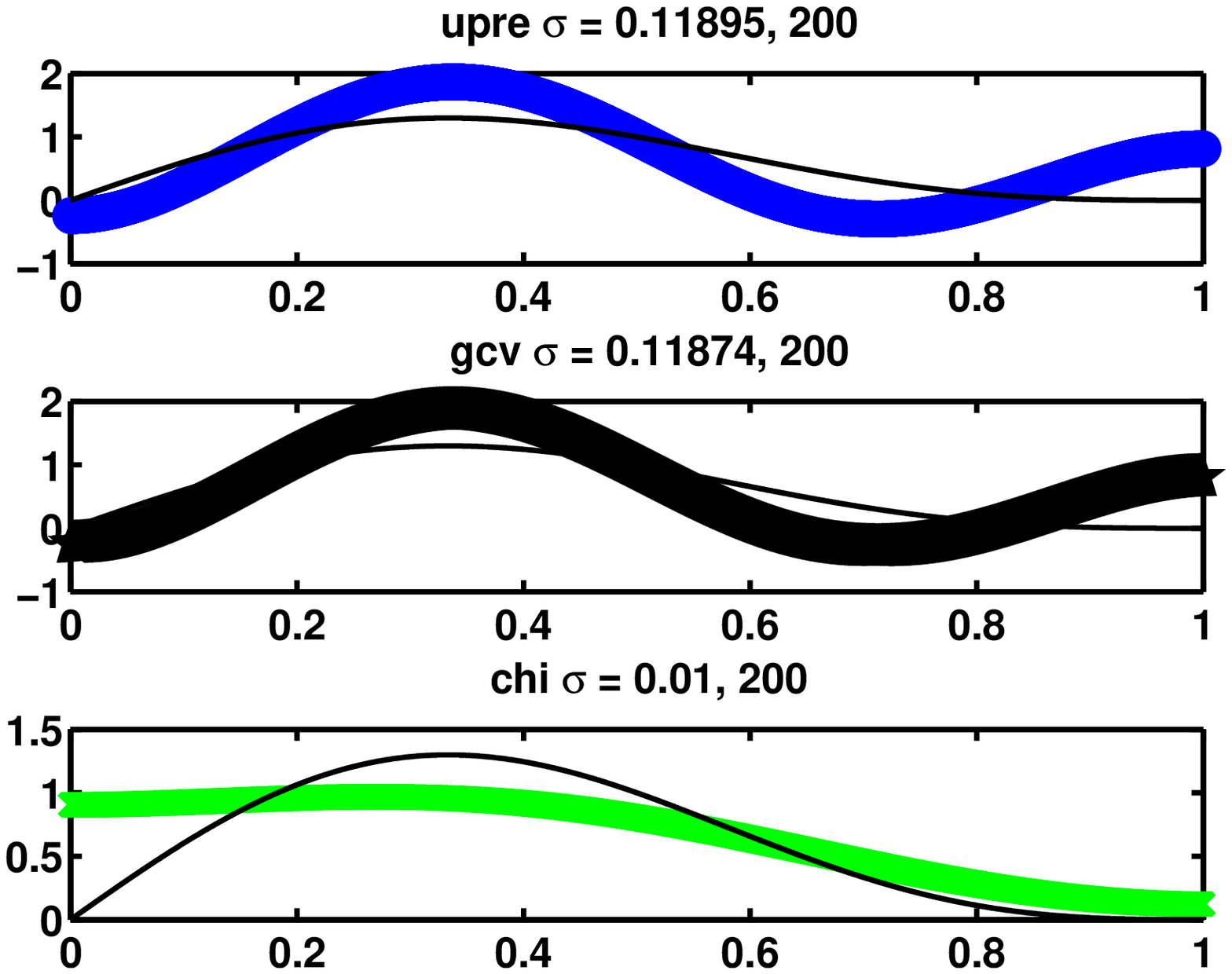}}\\
\subfigure[$p=2$ $3200$ ]{\includegraphics[width=.3\textwidth]{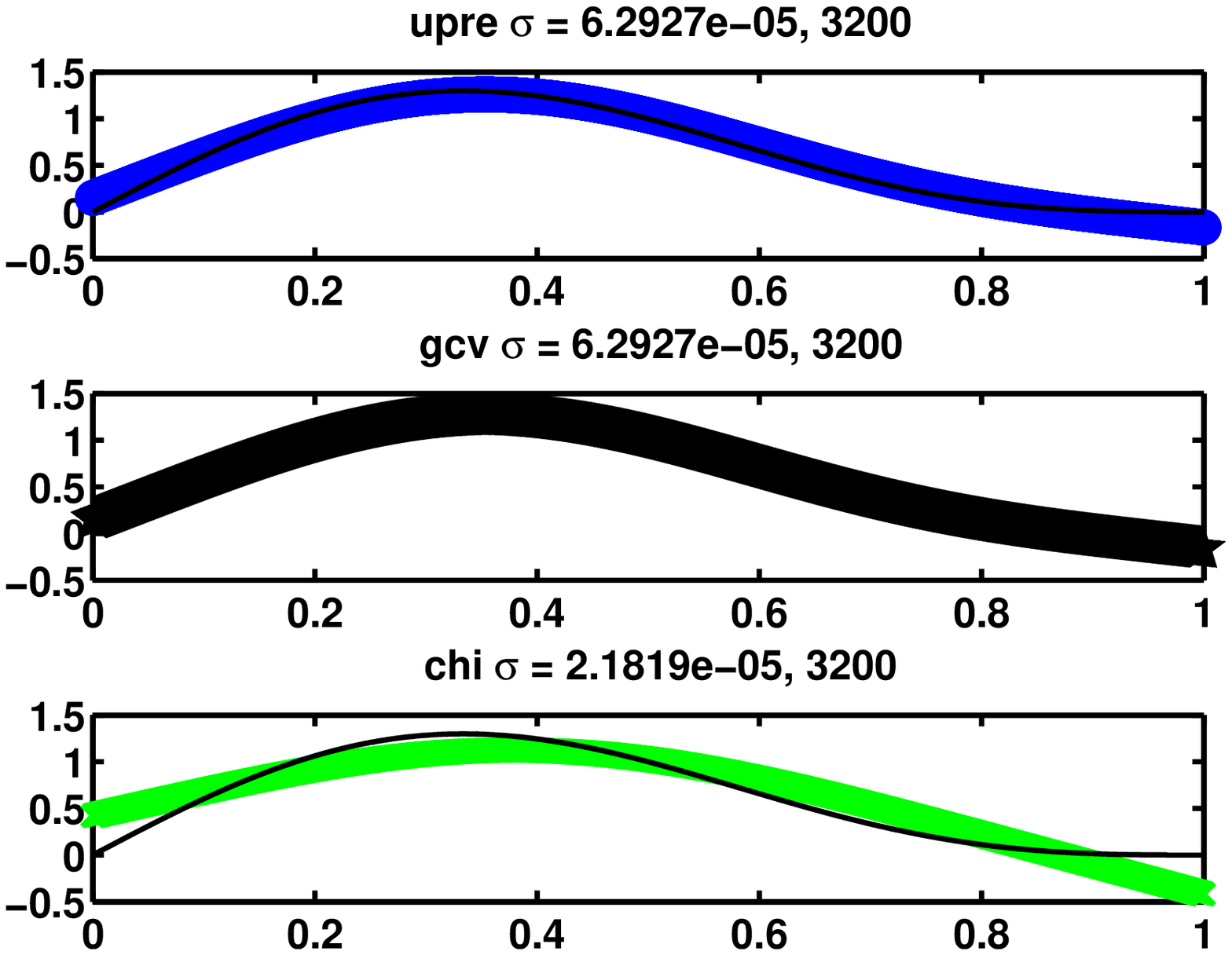}}
\subfigure[$p=2$ $1600$]{\includegraphics[width=.3\textwidth]{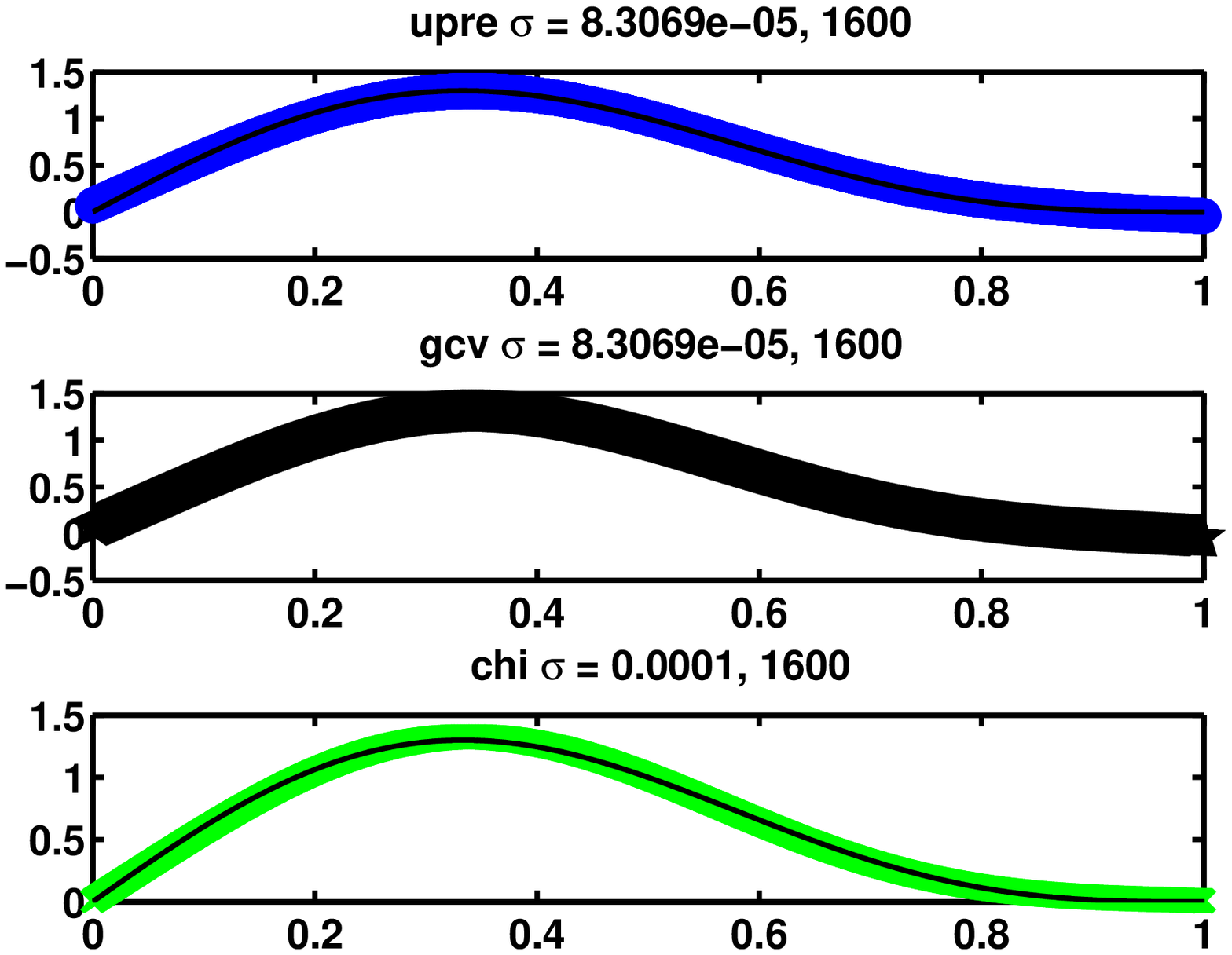}}
\subfigure[$p=2$ $200$ ]{\includegraphics[width=.3\textwidth]{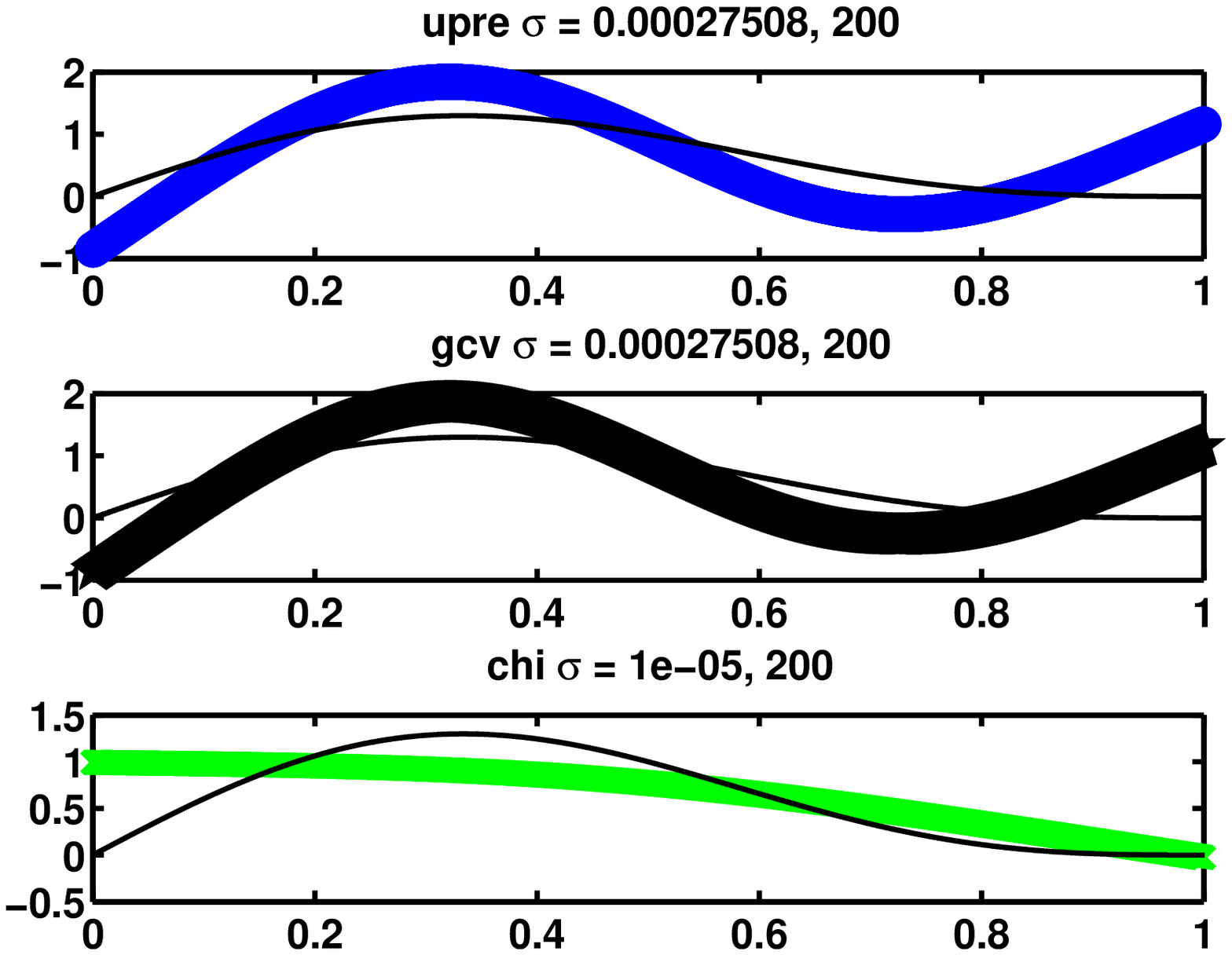}}
\end{center}\caption{Illustrative Results for noise level $.1$ for randomly selected sample right hand side in each case, but the same right hand side for each method. The exact solutions are given by the thin lines in each plot. \label{grav1fig}}
\end{figure}
\begin{figure}[!htb]
\begin{center}
\subfigure[$p=0$ $3200$ ]{\includegraphics[width=.3\textwidth]{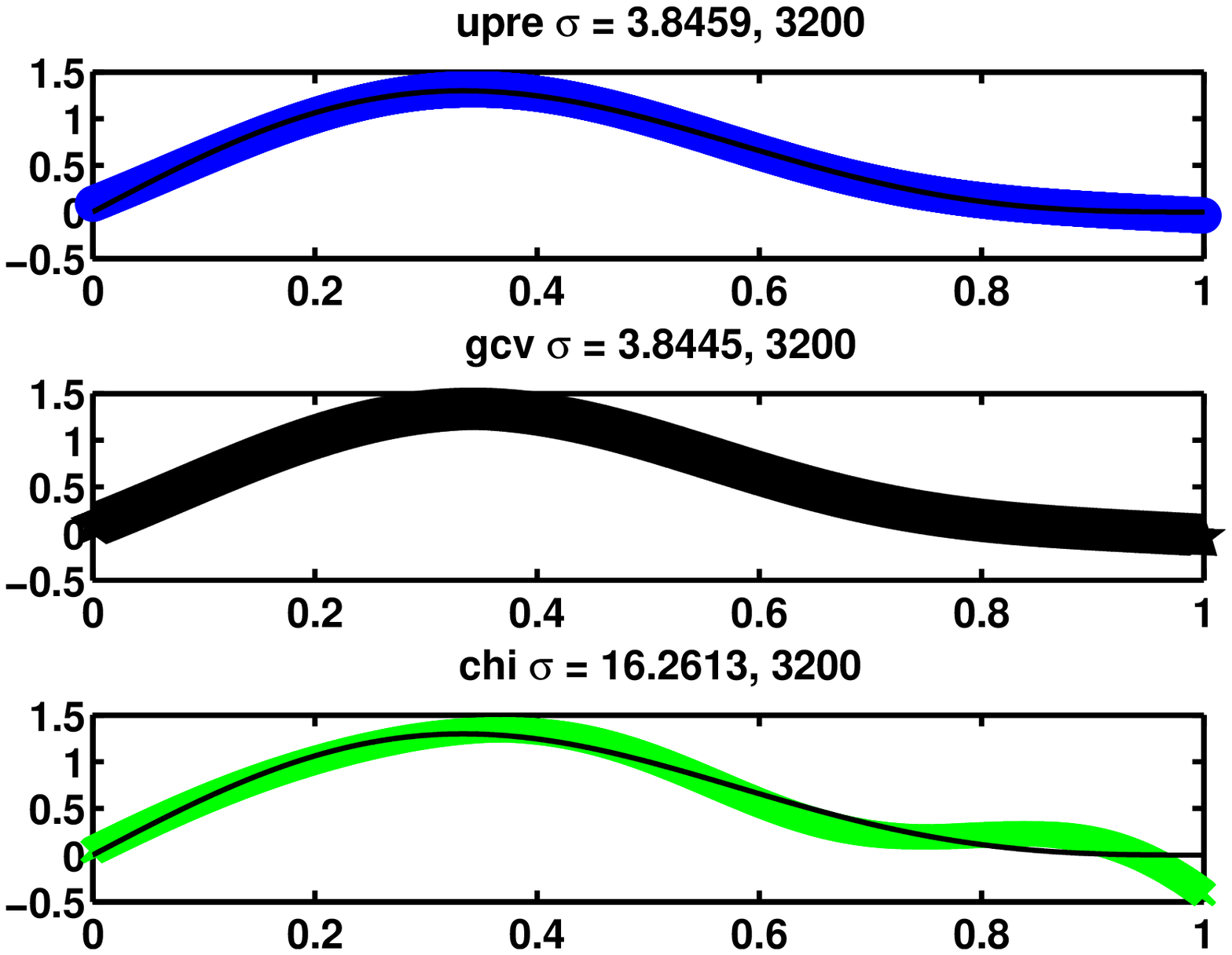}}
\subfigure[$p=0$  $1600$ ]{\includegraphics[width=.3\textwidth]{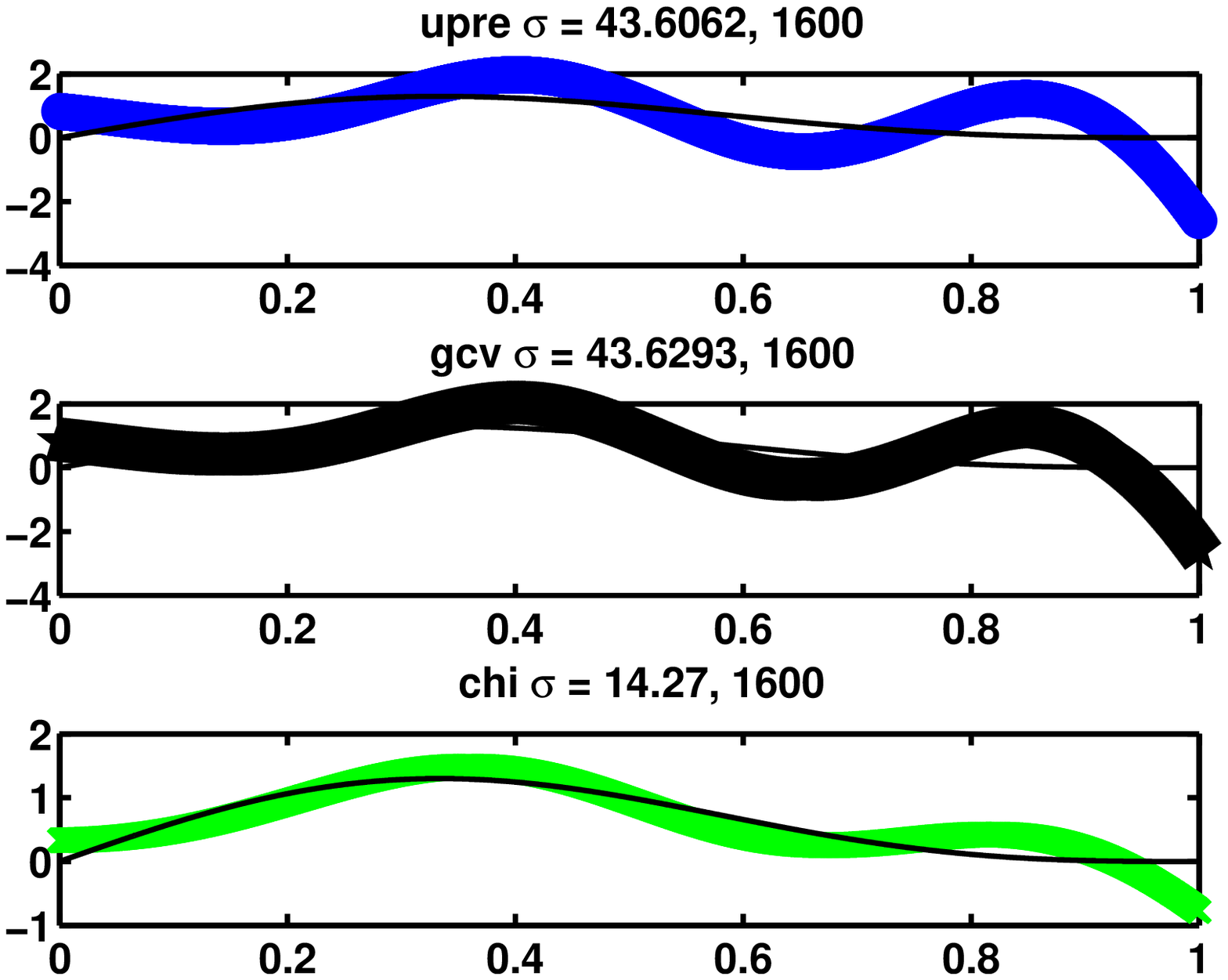}}
\subfigure[$p=0$ $200$ ]{\includegraphics[width=.3\textwidth]{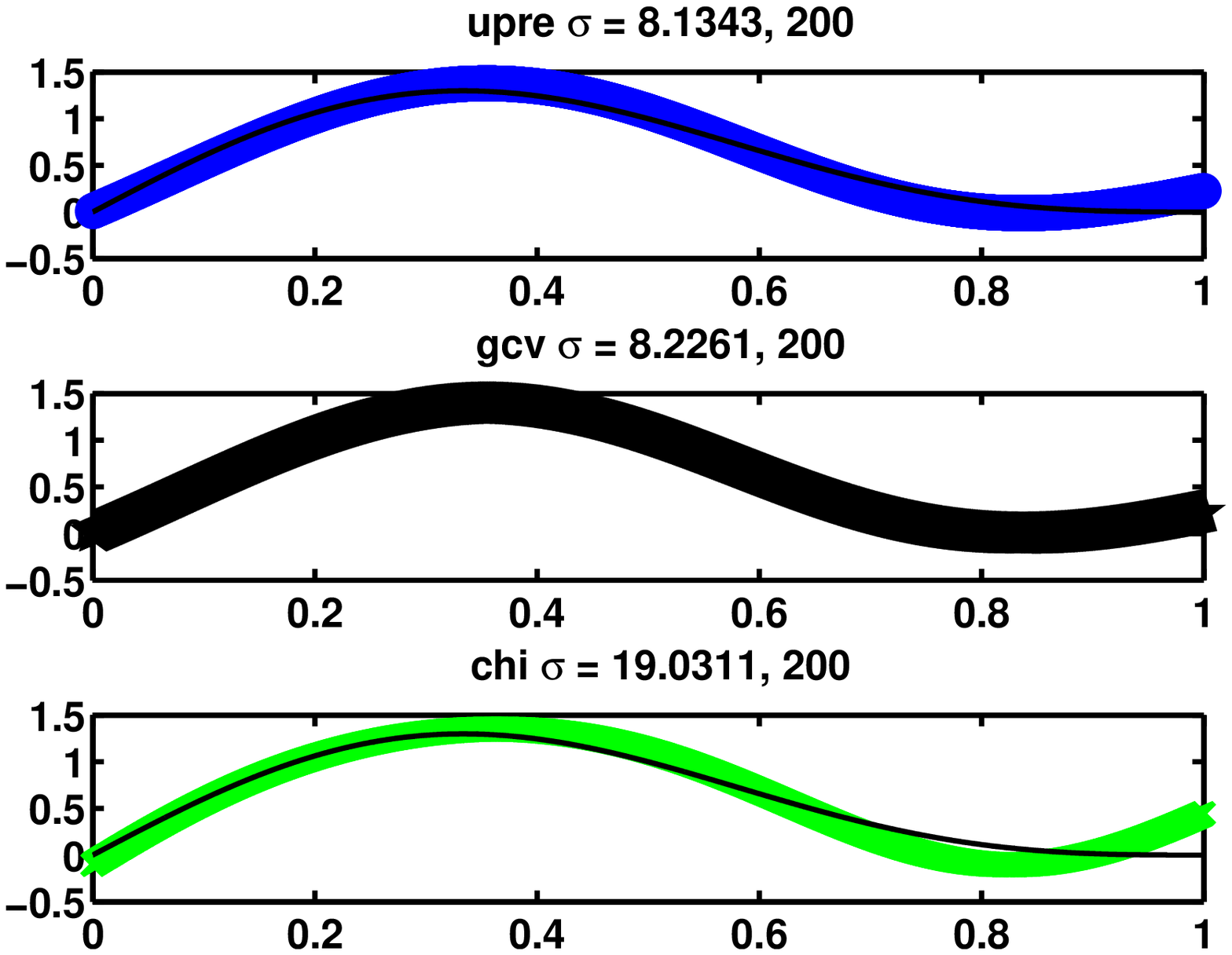}}\\
\subfigure[$p=1$  $3200$ ]{\includegraphics[width=.3\textwidth]{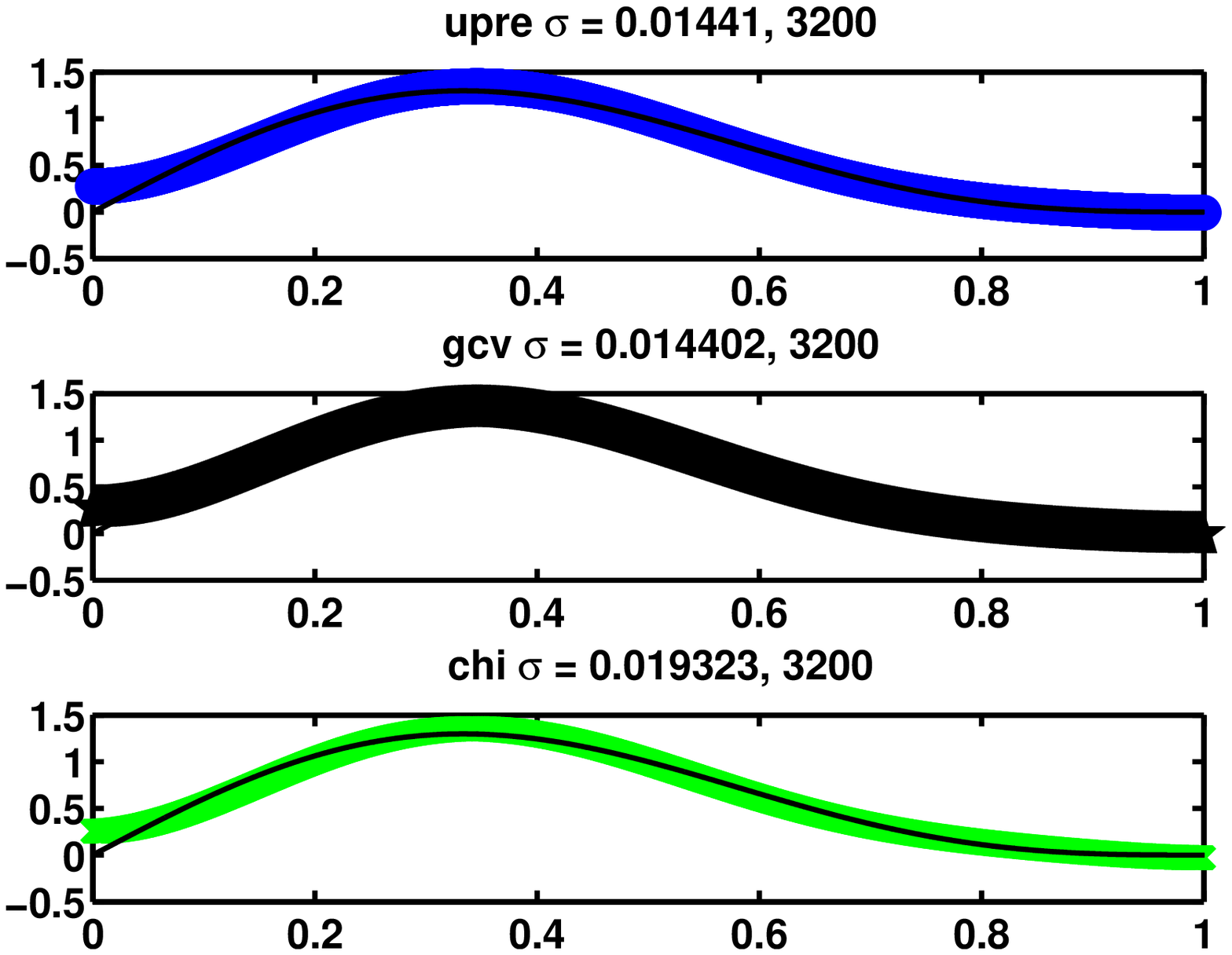}}
\subfigure[$p=1$  $1600$ ]{\includegraphics[width=.3\textwidth]{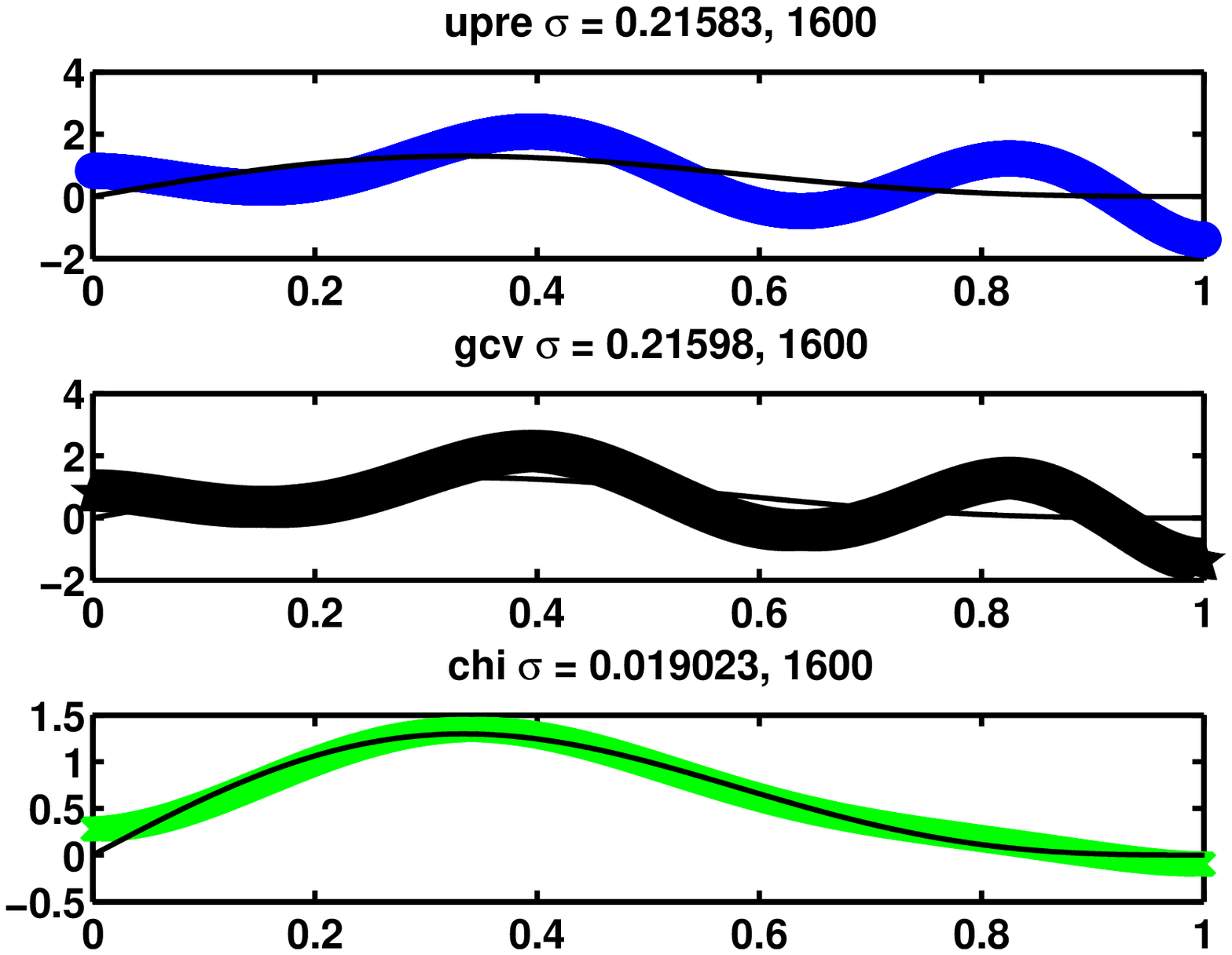}}
\subfigure[$p=1$ $200$ ]{\includegraphics[width=.3\textwidth]{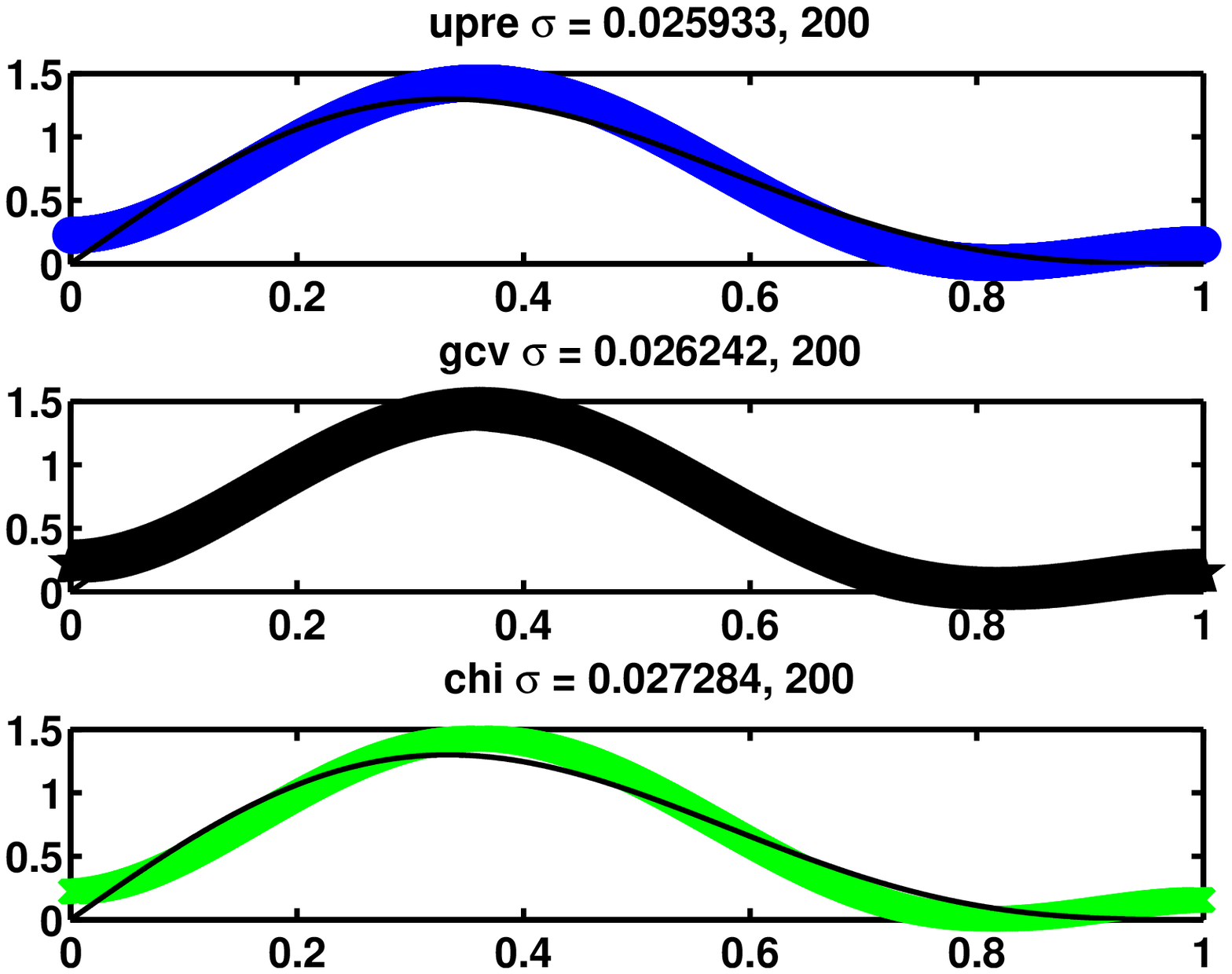}}\\
\subfigure[$p=2$ $3200$ ]{\includegraphics[width=.3\textwidth]{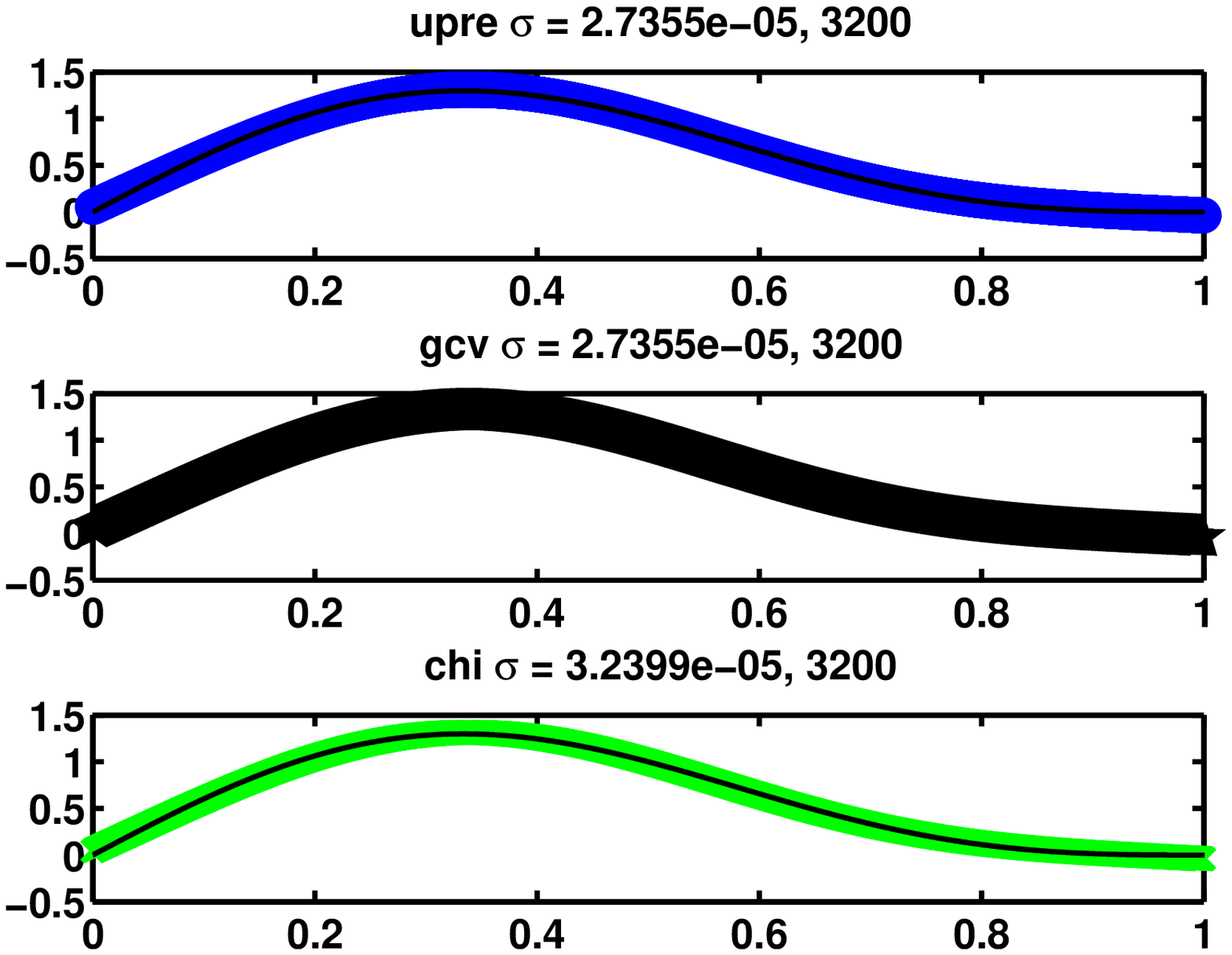}}
\subfigure[$p=2$ $1600$]{\includegraphics[width=.3\textwidth]{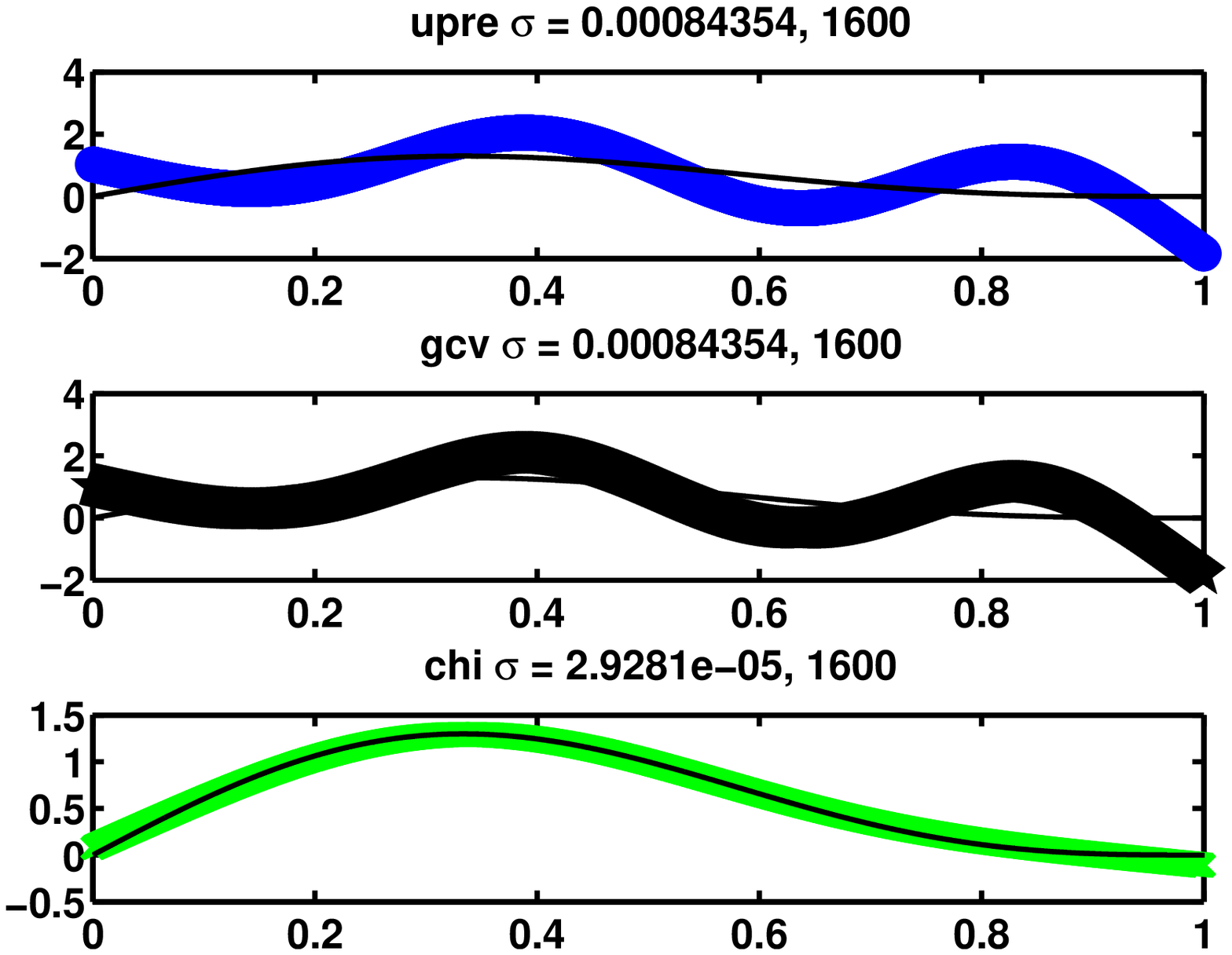}}
\subfigure[$p=2$ $200$ ]{\includegraphics[width=.3\textwidth]{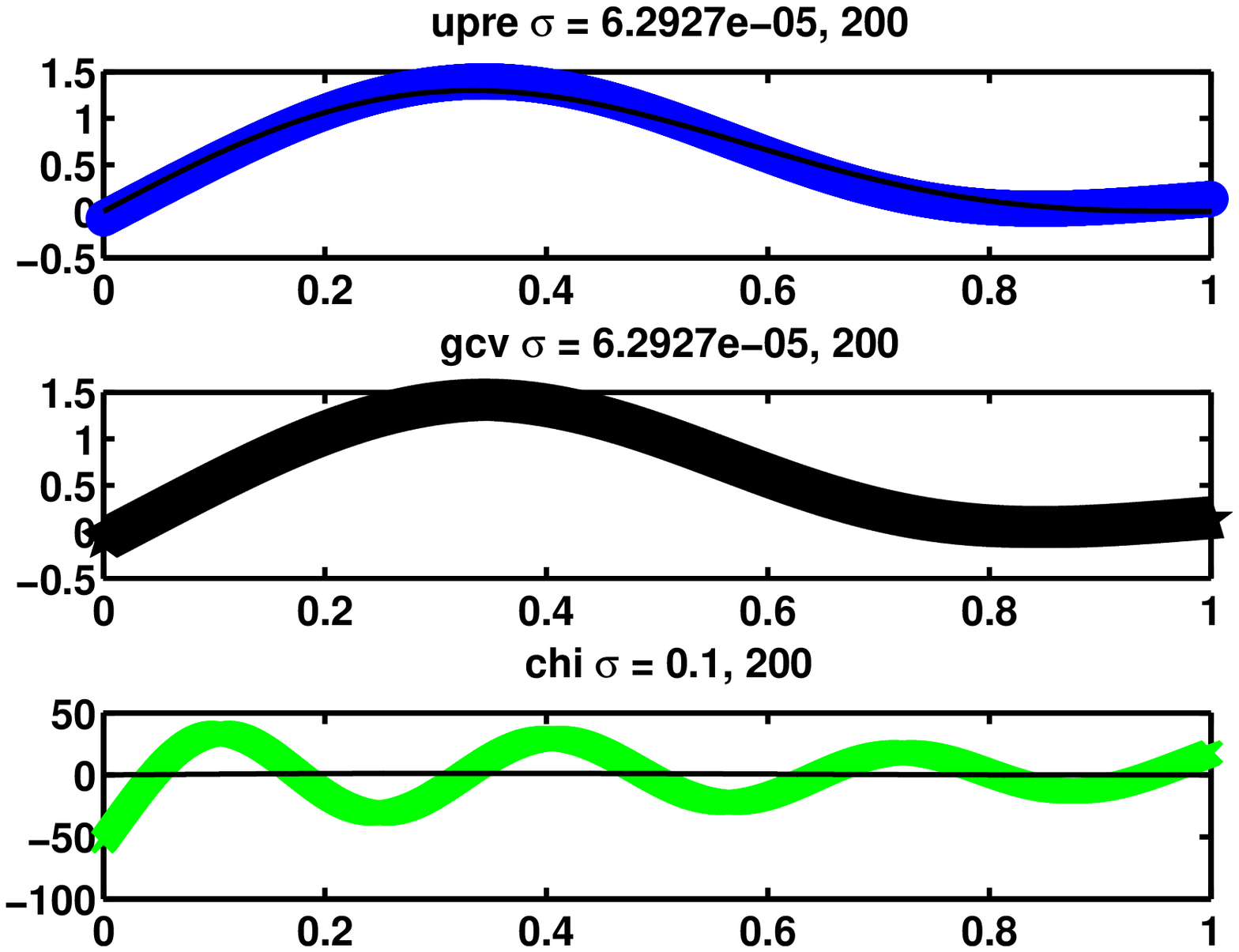}}
\end{center}\caption{Illustrative Results for noise level $.01$ for randomly selected sample right hand side in each case, but the same right hand side for each method. The exact solutions are given by the thin lines in each plot. \label{grav2fig}}
\end{figure}

The quantitative results presented in Tables~\ref{tabgrav1}-\ref{tabgrav2}, with the \textit{best} results in each case in bold,  demonstrate the remarkable consistency of the UPRE and GCV results. Application of the  $\chi^2$-principle is not as successful when $L=I$ for which the lack of a useful prior estimate for $\hat{\bfm}$, theoretically required to apply the central version of Theorem~\ref{newtheorem},  has a far greater impact.  On the other hand, for derivatives of order $1$ and $2$ this information is less necessary and competitive results are obtained, particularly for the lower noise level. Results in Figures~\ref{grav1fig}-\ref{grav2fig} demonstrate that all algorithms can   succeed, even with significant under sampling, $m=200$, but also may fail even for $m=1600$. When calculating for individual cases, rather than multiple cases at a time as with the results here, it is possible to adjust the number of points used in the minimization for the UPRE or GCV functionals. For the $\chi^2$ method it is possible to adjust the tolerance on root finding, or  apply filtering of the singular values,  with commensurate adjustment of the degrees of freedom, dependent on analysis of the root finding curve. It is clear that these are worst case results for the $\chi^2$-principle because of the lack of use of prior information. 

\begin{table}[!ht]\begin{center}\begin{tabular}{|*{6}{c|}}\hline
$m$&$3200$&$1600$&$800$&$400$&$200$\\ \hline
Method&  \multicolumn{5}{c|}{Derivative Order $0$} \\ \hline
UPRE&$\textbf{.175(.088)}$&$\textbf{.218(.158)}$&$\textbf{.213(.082)}$&$\textbf{.239(.098)}$&$.331(.204)$\\ \hline
GCV&\textbf{$\textbf{.175(.088)}$}&$\textbf{.218(.158)}$&$\textbf{.213(.082)}$&$\textbf{.239(.098)}$&$.332(.205)$\\ \hline
$\chi^2$&$.223(.179)$&$.273(.234)$&$.331(.180)$&$.327(.186)$&$\textbf{.290(.161)}$\\ \hline
&  \multicolumn{5}{c|}{Derivative Order $1$} \\ \hline
UPRE&$.202(.084)$&$\textbf{.248(.151)}$&$\textbf{.238(.077)}$&$\textbf{.260(.088)}$&$.336(.201)$\\ \hline
GCV&$.202(.084)$&$\textbf{.248(.151)}$&$\textbf{.238(.077)}$&$\textbf{.260(.088)}$&$.337(.202)$\\ \hline
$\chi^2$&$\textbf{.190(.052)}$&$.260(.171)$&$.272(.093)$&$.286(.116)$&$\textbf{.305(.065)}$\\ \hline
&  \multicolumn{5}{c|}{Derivative Order $2$} \\ \hline
UPRE&$\textbf{.195(.111)}$&$\textbf{.246(.160)}$&$\textbf{.257(.087)}$&$.280(.094)$&$\textbf{.361(.188)}$\\ \hline
GCV&$\textbf{.195(.111)}$&$\textbf{.246(.160)}$&$\textbf{.257(.087)}$&$\textbf{.279(.093)}$&$\textbf{.361(.188)}$\\ \hline
$\chi^2$&$.226(.087)$&$.258(.084)$&$.430(.230)$&$.338(.161)$&$.397(.175)$\\ \hline
\hline
\end{tabular}\caption{The mean and standard deviation of the relative error over $25$ copies of the data with noise level $.1$. In each case $n=3200$ and downsampling is obtained by sampling at a sampling rate $1$, $2$, $4$, $8$ and $16$. Best results for each case in boldface.}\label{tabgrav1}\end{center}\end{table}

\begin{table}[!ht]\begin{center}\begin{tabular}{|*{6}{c|}}\hline
$m$&$3200$&$1600$&$800$&$400$&$200$\\ \hline
Method&  \multicolumn{5}{c|}{Derivative Order $0$} \\ \hline
UPRE&$\textbf{.149(.205)}$&$\textbf{.075(.122)}$&$\textbf{.199(.301)}$&$\textbf{.120(.103)}$&$\textbf{.139(.081)}$\\\hline
GCV&$\textbf{.149(.205)}$&$.075(.123)$&$\textbf{.199(.301)}$&$.120(.104)$&$\textbf{.139(.081)}$\\\hline
$\chi^2$&$.255(.165)$&$.166(.130)$&$.300(.272)$&$.232(.120)$&$.267(.176)$\\\hline 
&  \multicolumn{5}{c|}{Derivative Order $1$} \\ \hline
UPRE&$.164(.197)$&$.108(.123)$&$.187(.258)$&$.164(.161)$&$\textbf{.155(.067)}$\\\hline
GCV&$.164(.197)$&$.108(.123)$&$.187(.258)$&$.164(.161)$&$\textbf{.155(.067)}$\\\hline
$\chi^2$&$\textbf{.151(.202)}$&$\textbf{.088(.030)}$&$\textbf{.137(.140)}$&$\textbf{.119(.058)}$&$.178(.197)$\\\hline 
&  \multicolumn{5}{c|}{Derivative Order $2$} \\ \hline
UPRE&$.125(.203)$&$.063(.122)$&$.104(.199)$&$.102(.110)$&$\textbf{.101(.063)}$\\\hline
GCV&$.125(.203)$&$.063(.122)$&$.104(.199)$&$\textbf{.095(.103)}$&$\textbf{.101(.063)}$\\\hline
$\chi^2$&$\textbf{.051(.034)}$&$\textbf{.045(.030)}$&$\textbf{.061(.040)}$&$.148(.209)$&$.187(.228)$\\\hline 
\end{tabular}\caption{The mean and standard deviation of the relative error over $25$ copies of the data with noise level $.01$. In each case $n=3200$ and downsampling is obtained by sampling at a sampling rate $1$, $2$, $4$, $8$ and $16$. Best results for each case in boldface. }\label{tabgrav2}\end{center}\end{table}

\subsubsection{Problem \texttt{tomo}}
Figure~\ref{tomo1fig} illustrates results for data contaminated by random noise with variance $.0004$ and $.0001$, $\eta=.02$ and $\eta=.01$, respectively, with solutions obtained with regularization using a first order derivative operator, and sampled using $100\%$, $75\%$ and $50\%$ of the data. At these noise levels, the quality of the solutions when obtained with $m=n$ are also not ideal, but do demonstrate that with reduction of sampling it is still possible to apply parameter estimation techniques to find effective solutions, i.e.  all methods succeed in finding useful regularization parameters, demonstrating again that these techniques can be used for under sampled data sets. 

Overall the results for \texttt{gravity} and \texttt{tomo} demonstrate that algorithms for regularization parameter estimation can be  successfully applied for problems with fewer samples than desirable.

\begin{figure}[!htb]
\begin{center}
\subfigure[$3600(.301)$]{\includegraphics[width=.15\textwidth]{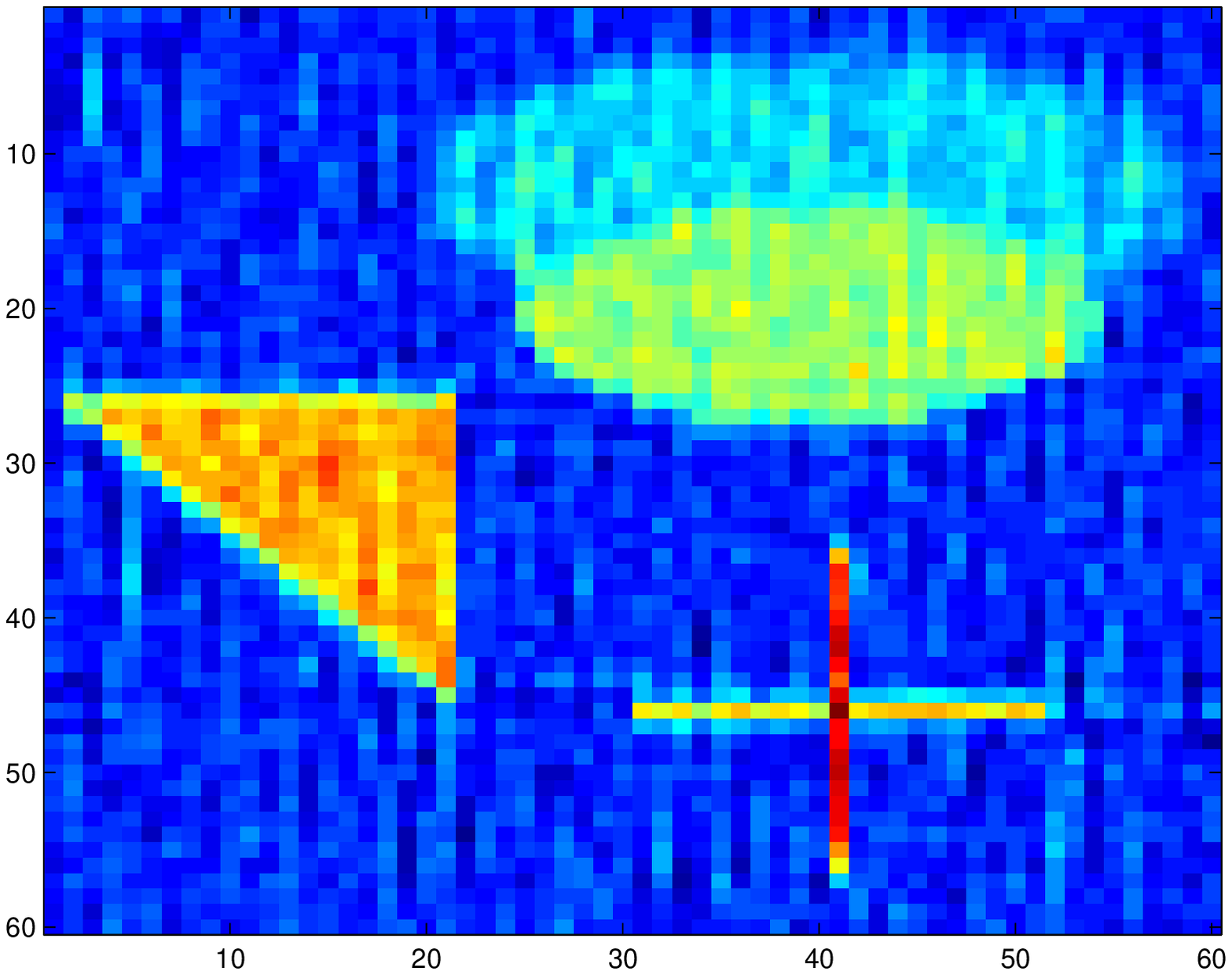}}
\subfigure[$2700(.329)$]{\includegraphics[width=.15\textwidth]{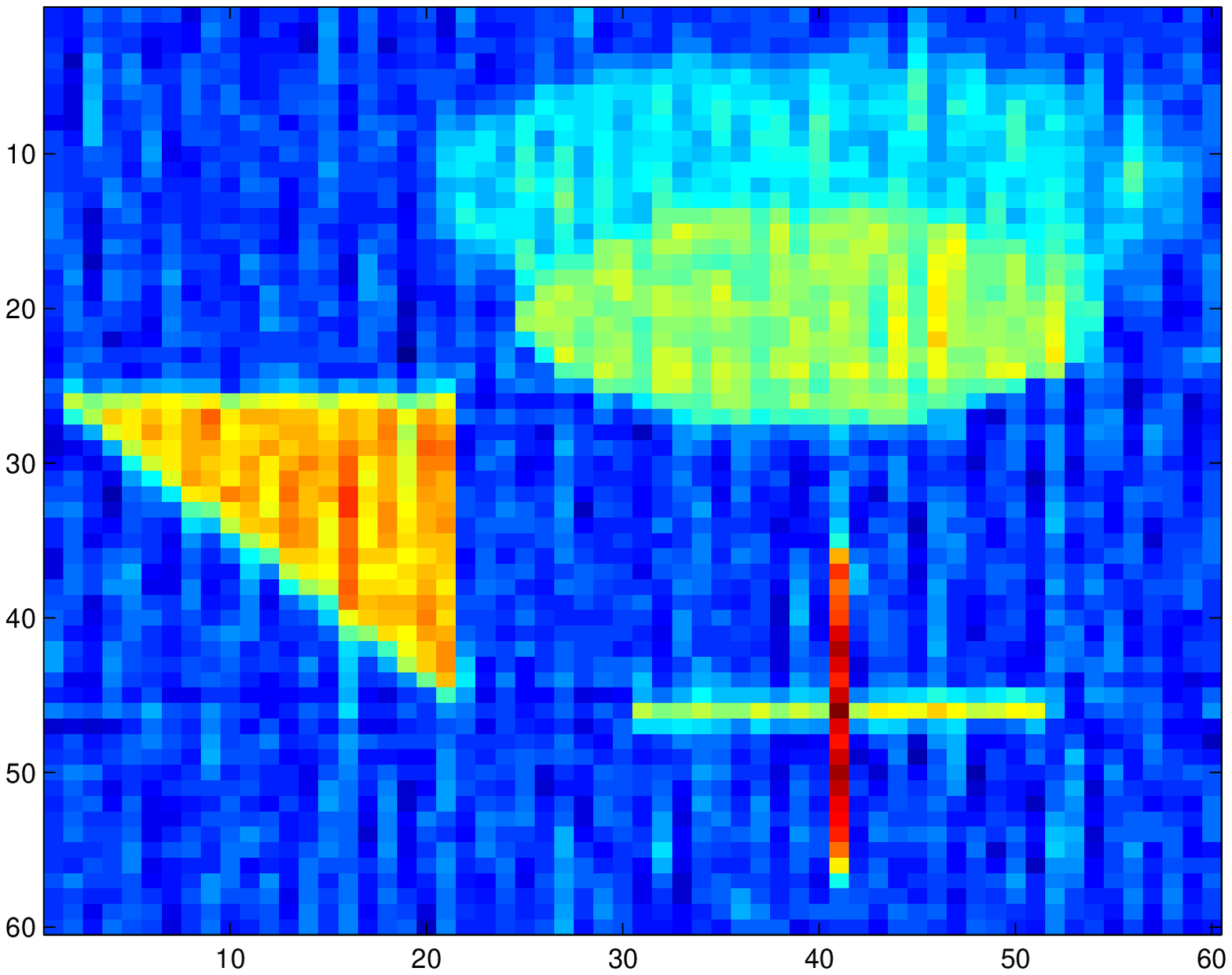}}
\subfigure[$1800(.387)$]{\includegraphics[width=.15\textwidth]{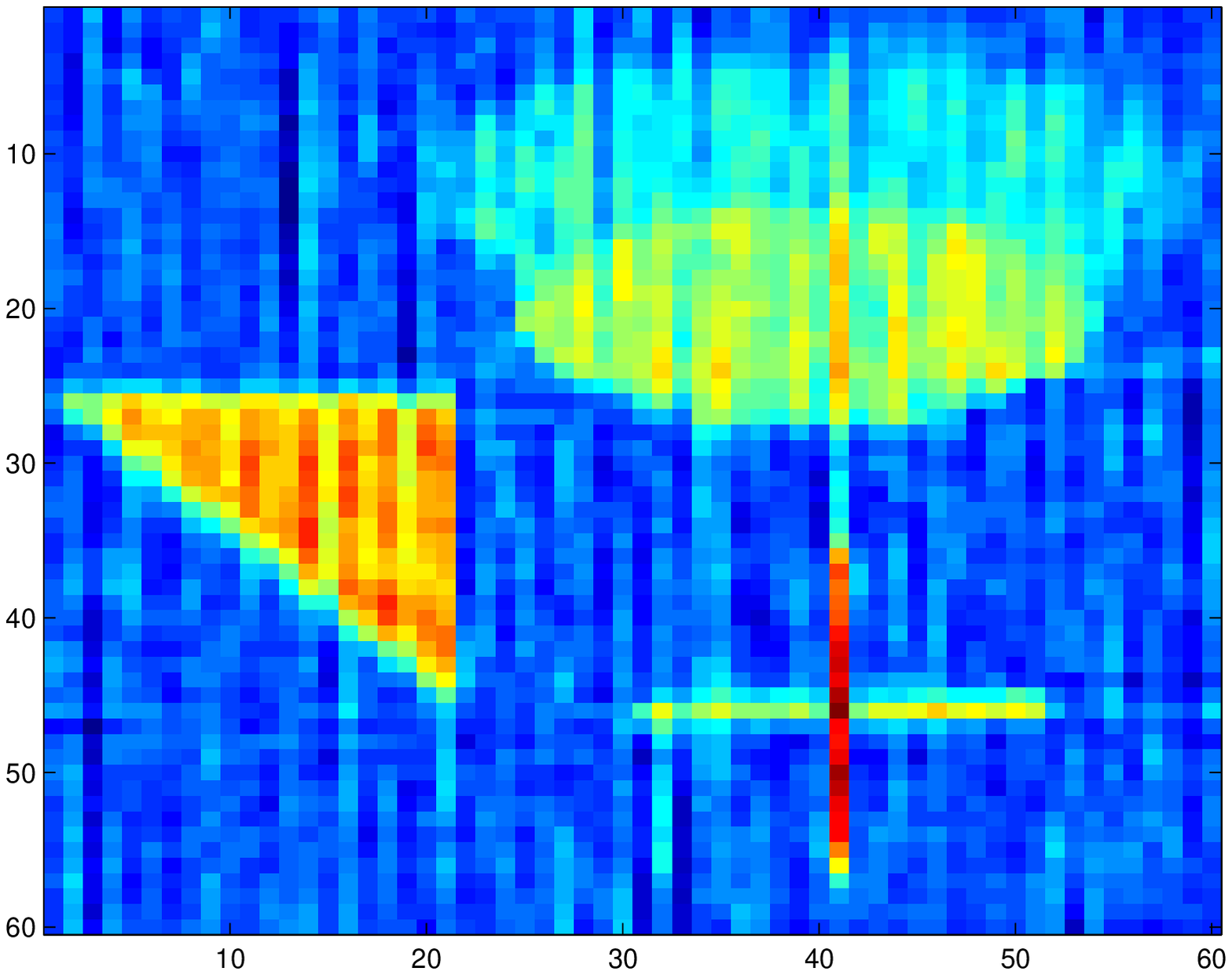}}
\subfigure[$3600(.224)$]{\includegraphics[width=.15\textwidth]{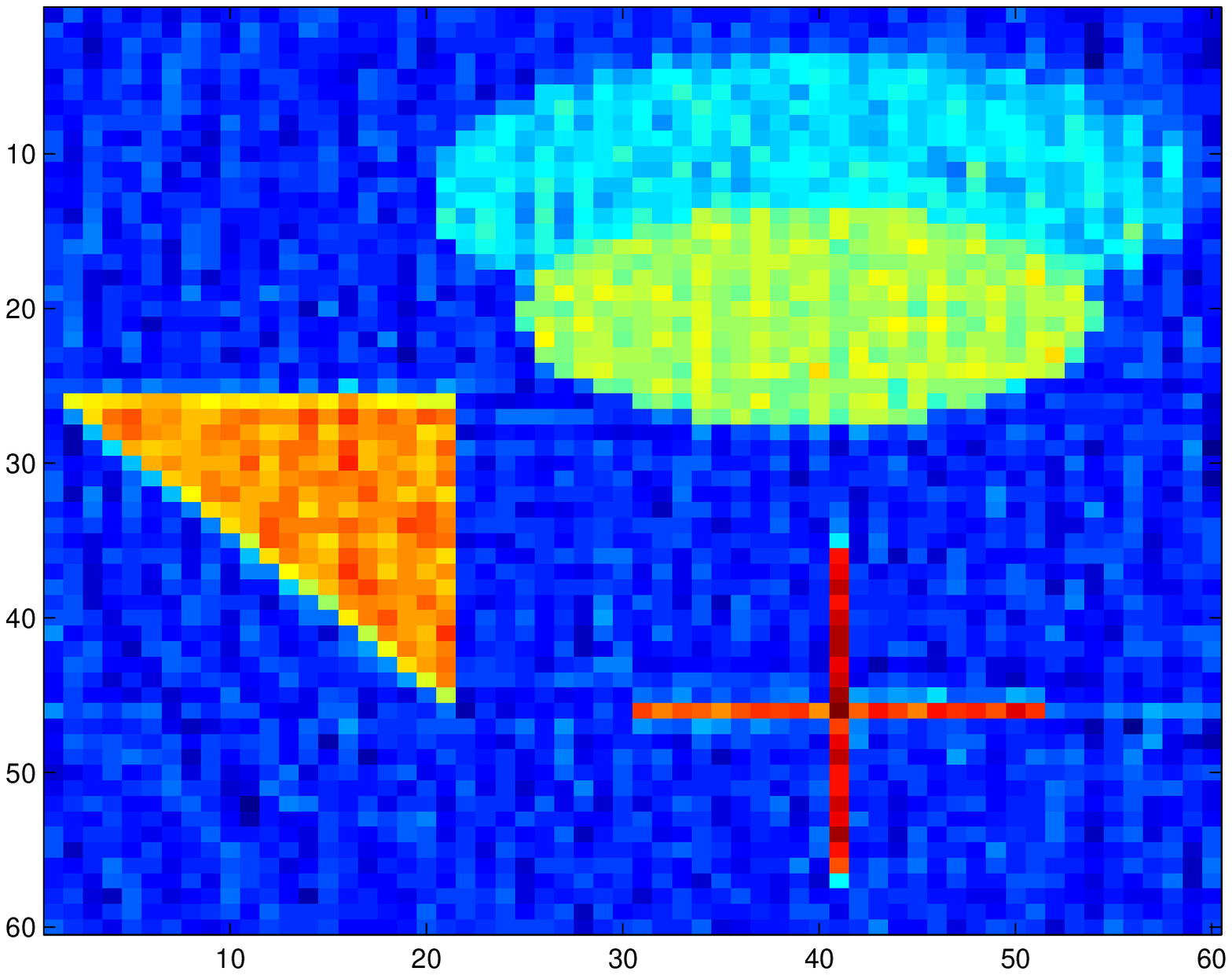}}
\subfigure[$2700(.268)$]{\includegraphics[width=.15\textwidth]{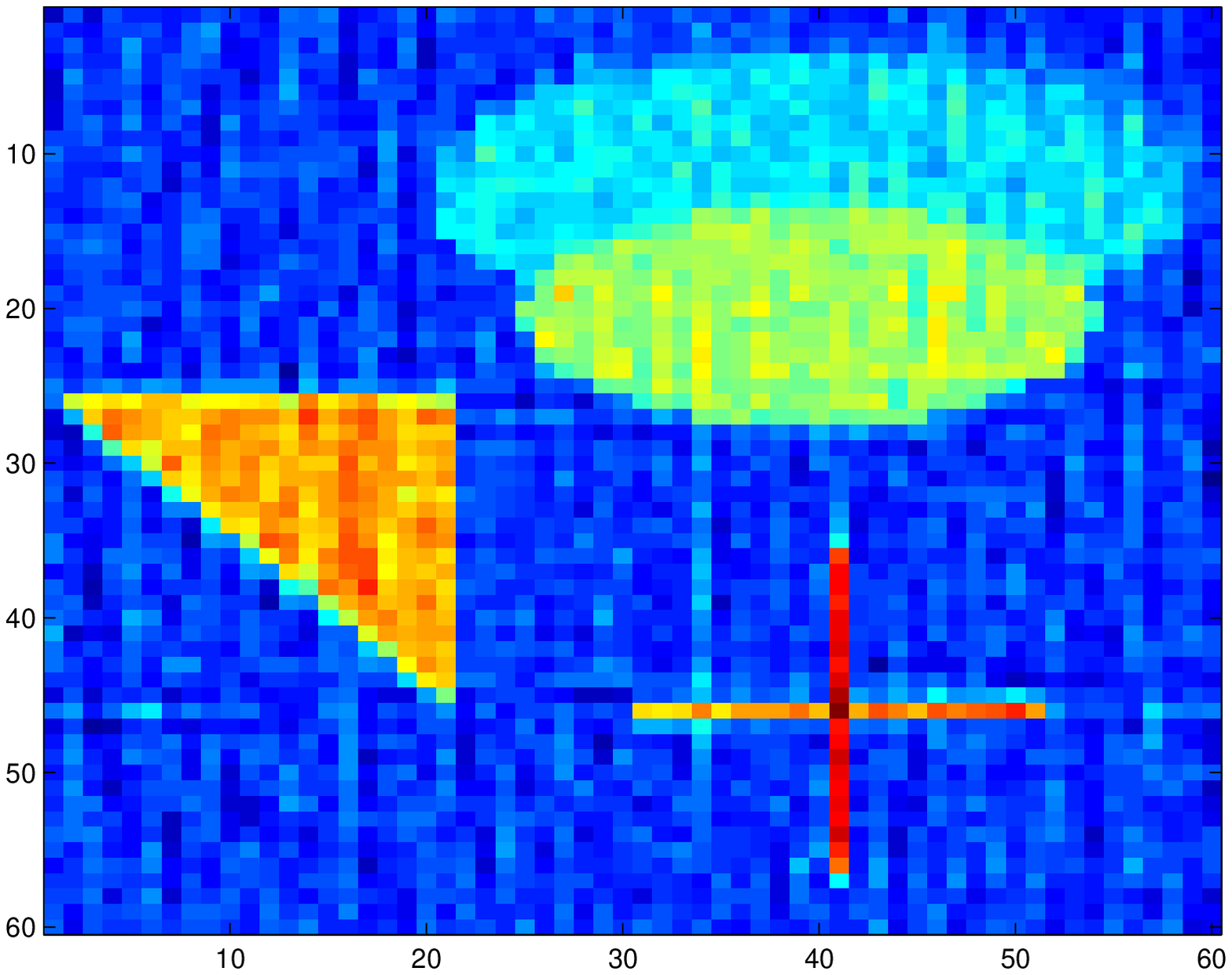}}
\subfigure[$1800(.332)$]{\includegraphics[width=.15\textwidth]{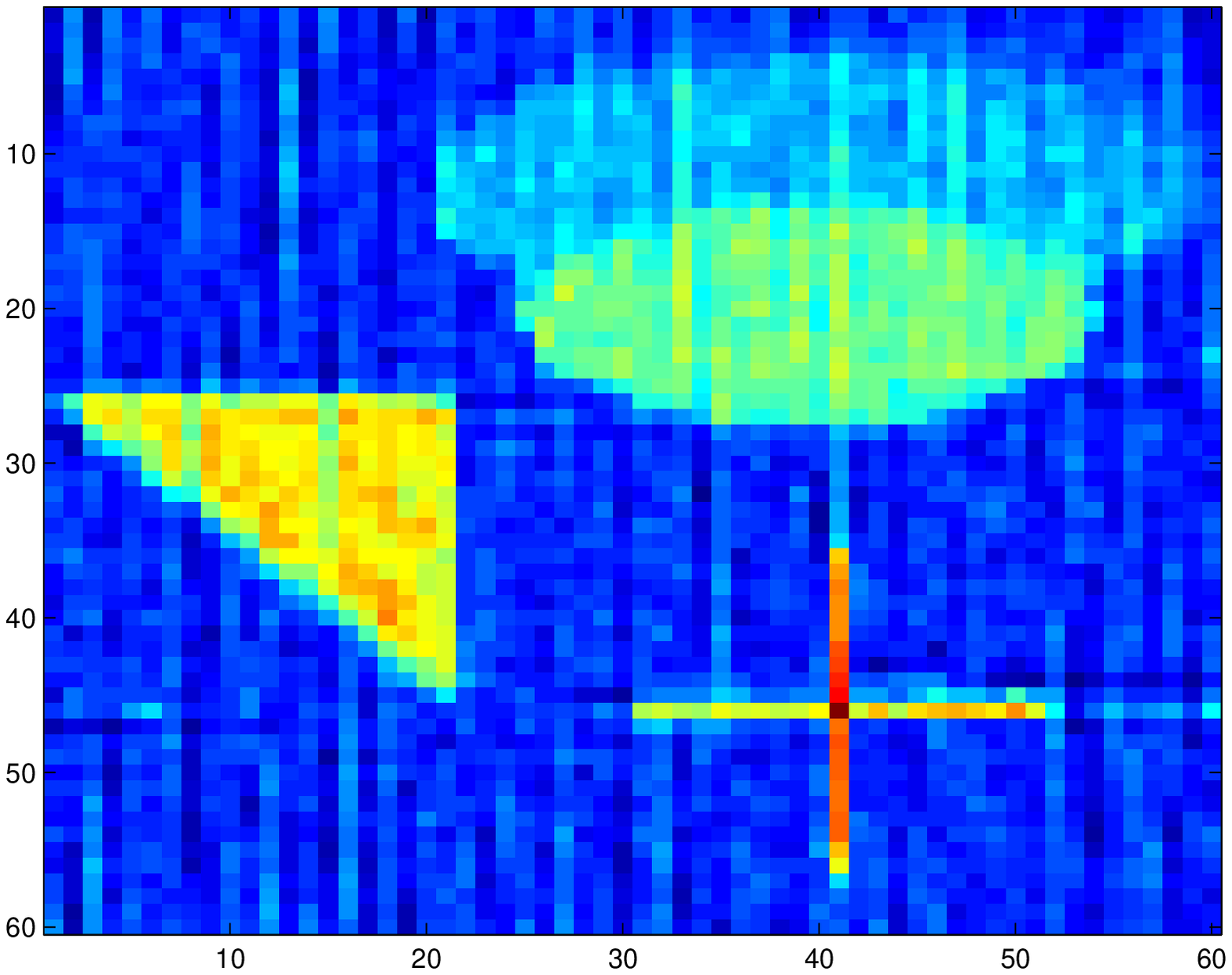}}\\
\subfigure[$3600(.302)$]{\includegraphics[width=.15\textwidth]{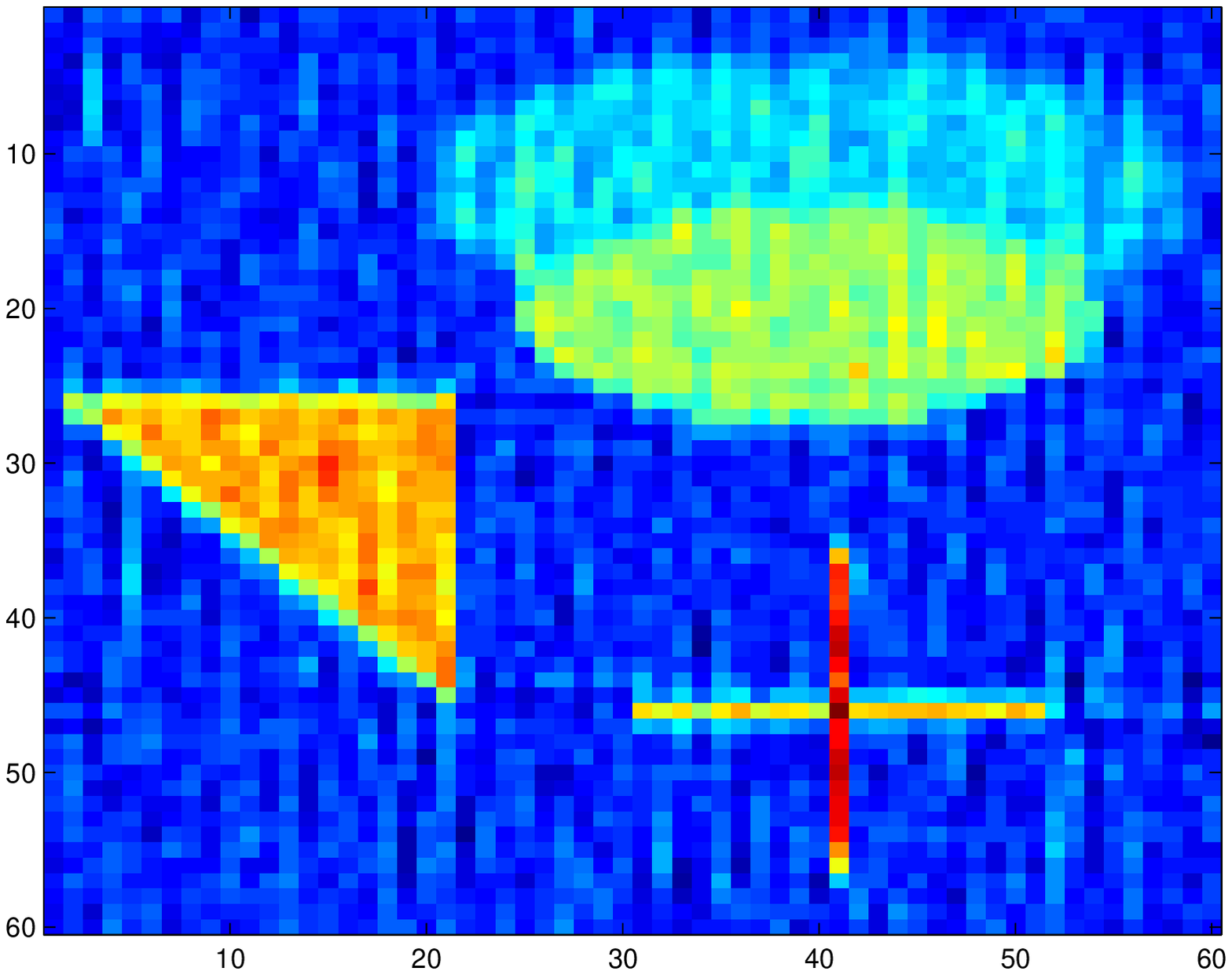}}
\subfigure[$2700(.328)$]{\includegraphics[width=.15\textwidth]{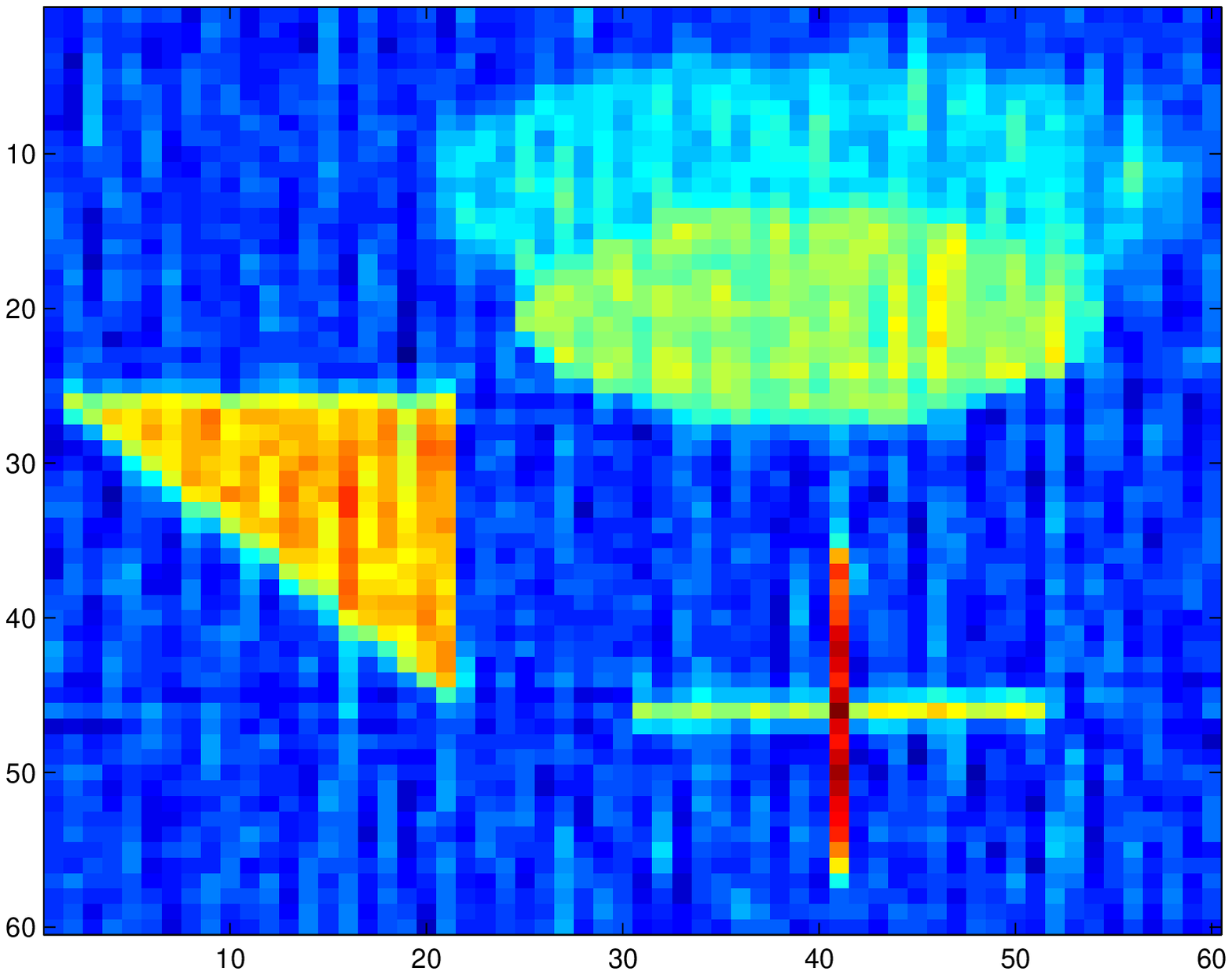}}
\subfigure[$1800(.380)$]{\includegraphics[width=.15\textwidth]{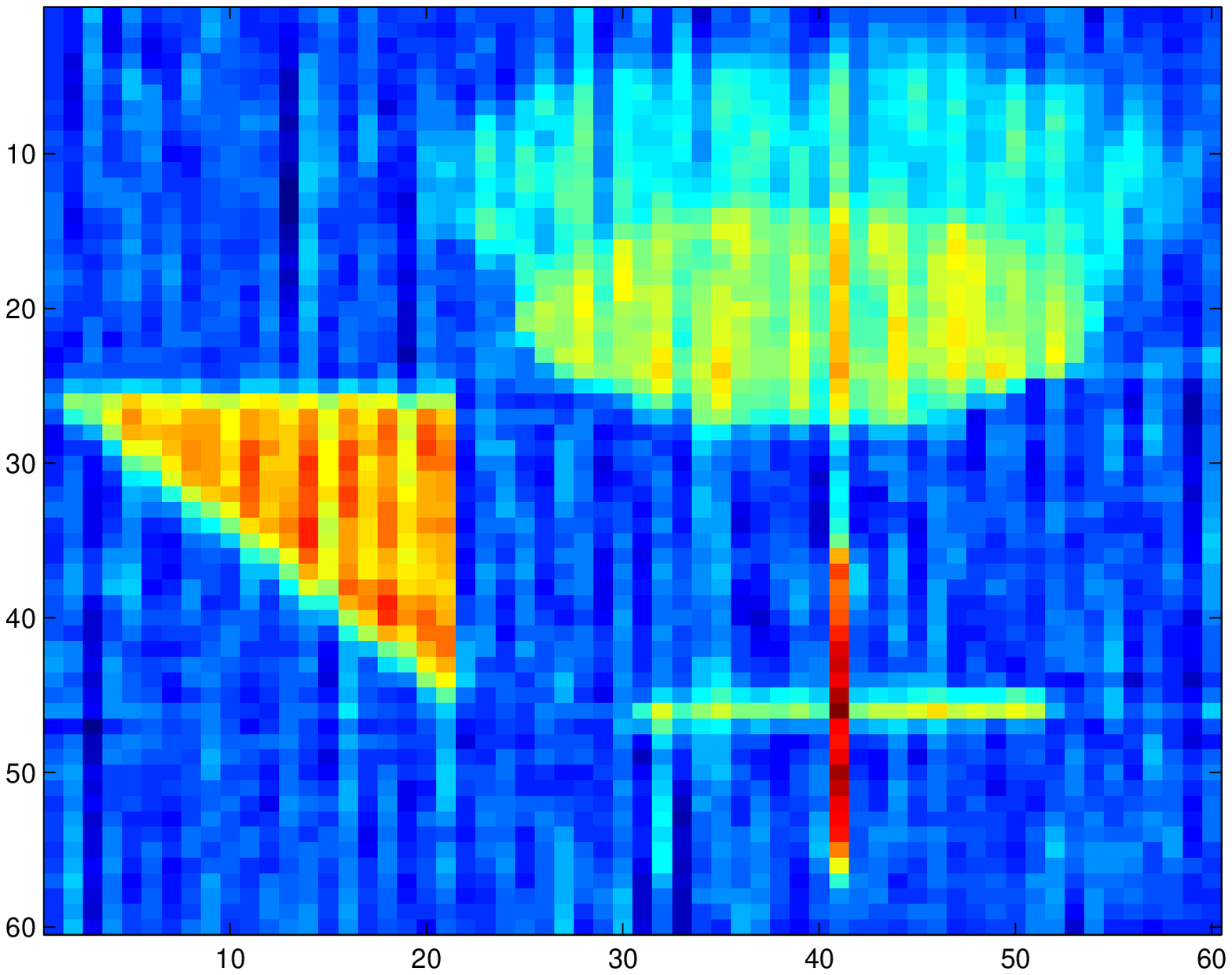}}
\subfigure[$3600(.230)$]{\includegraphics[width=.15\textwidth]{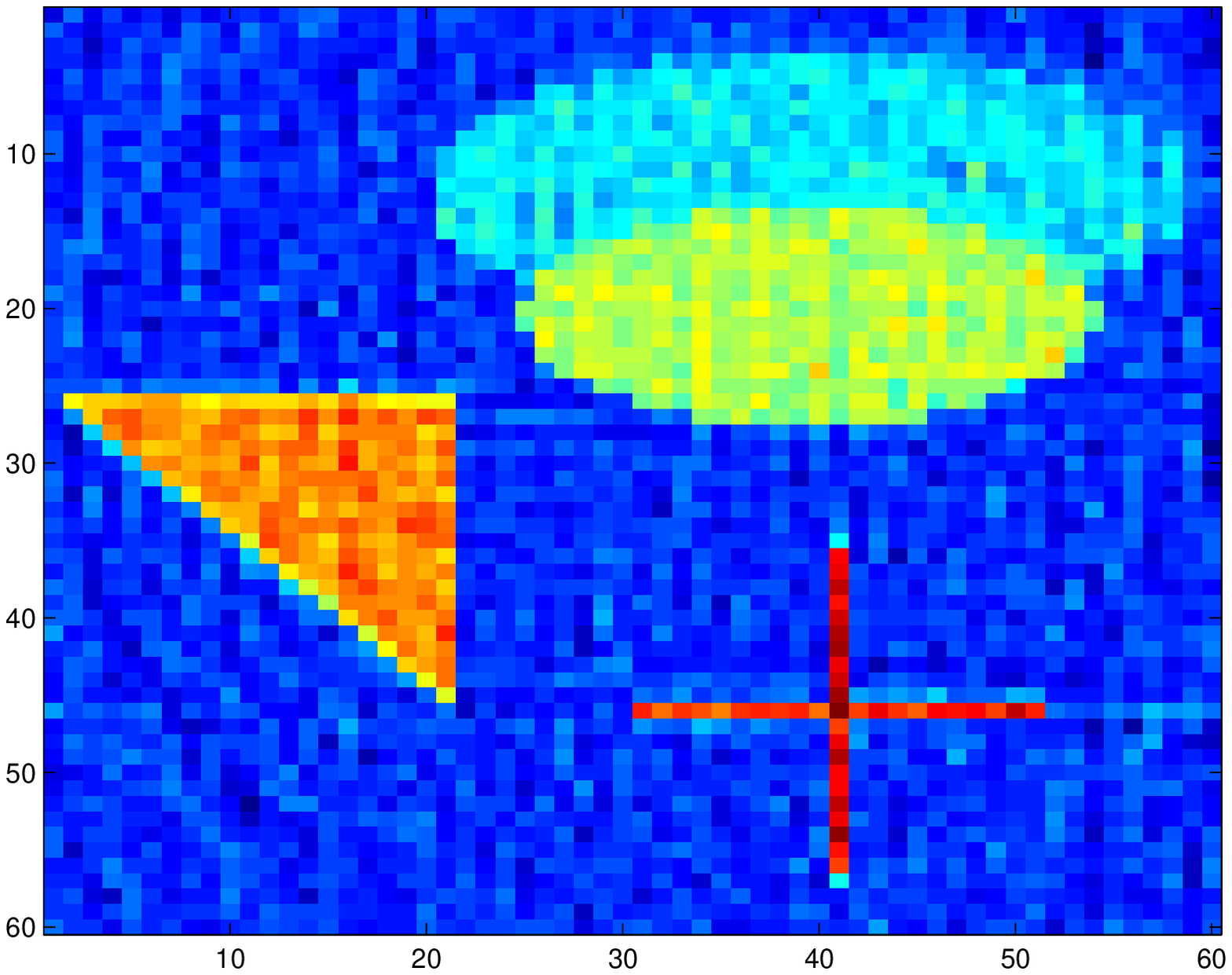}}
\subfigure[$2700(.278)$]{\includegraphics[width=.15\textwidth]{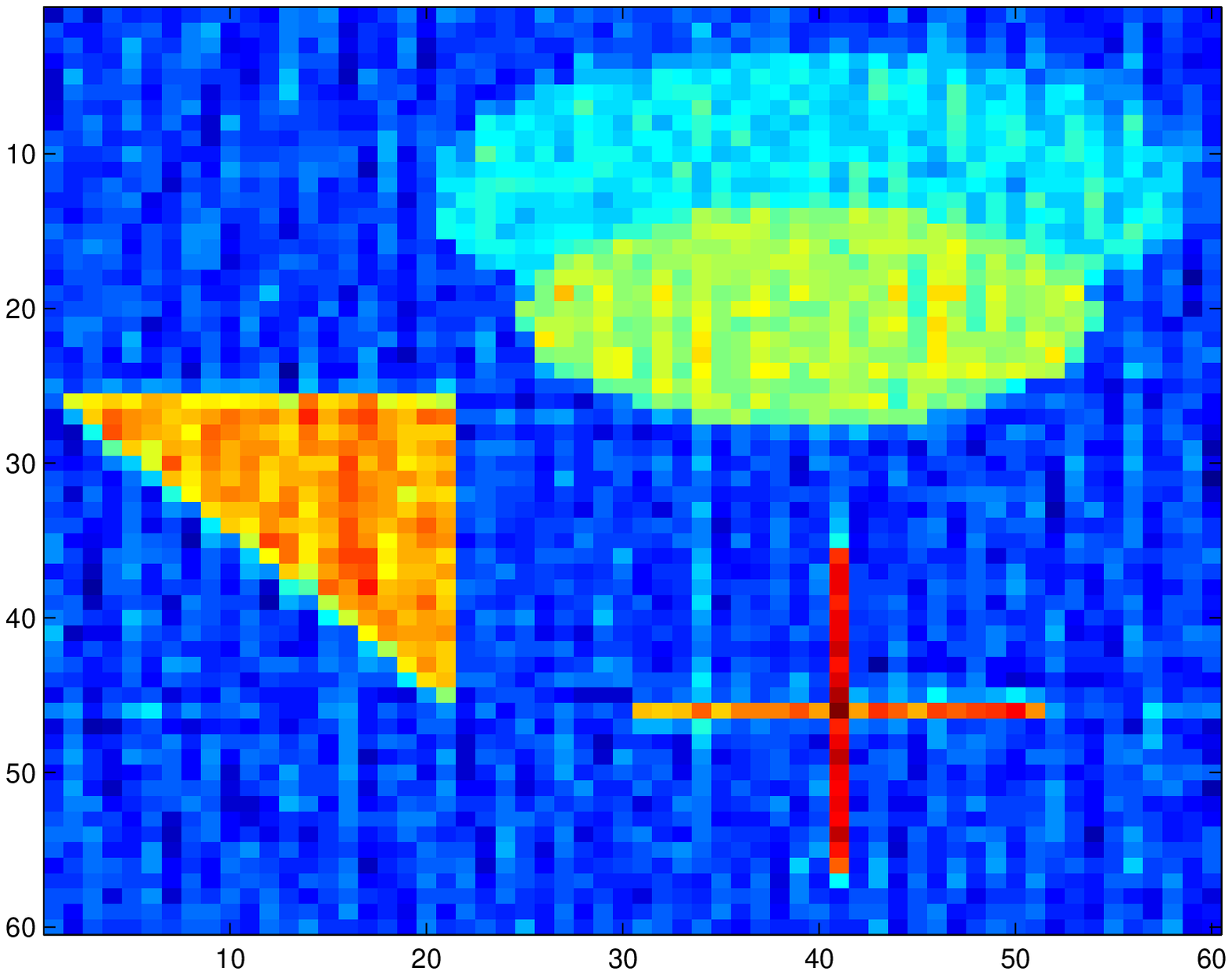}}
\subfigure[$1800(.384)$]{\includegraphics[width=.15\textwidth]{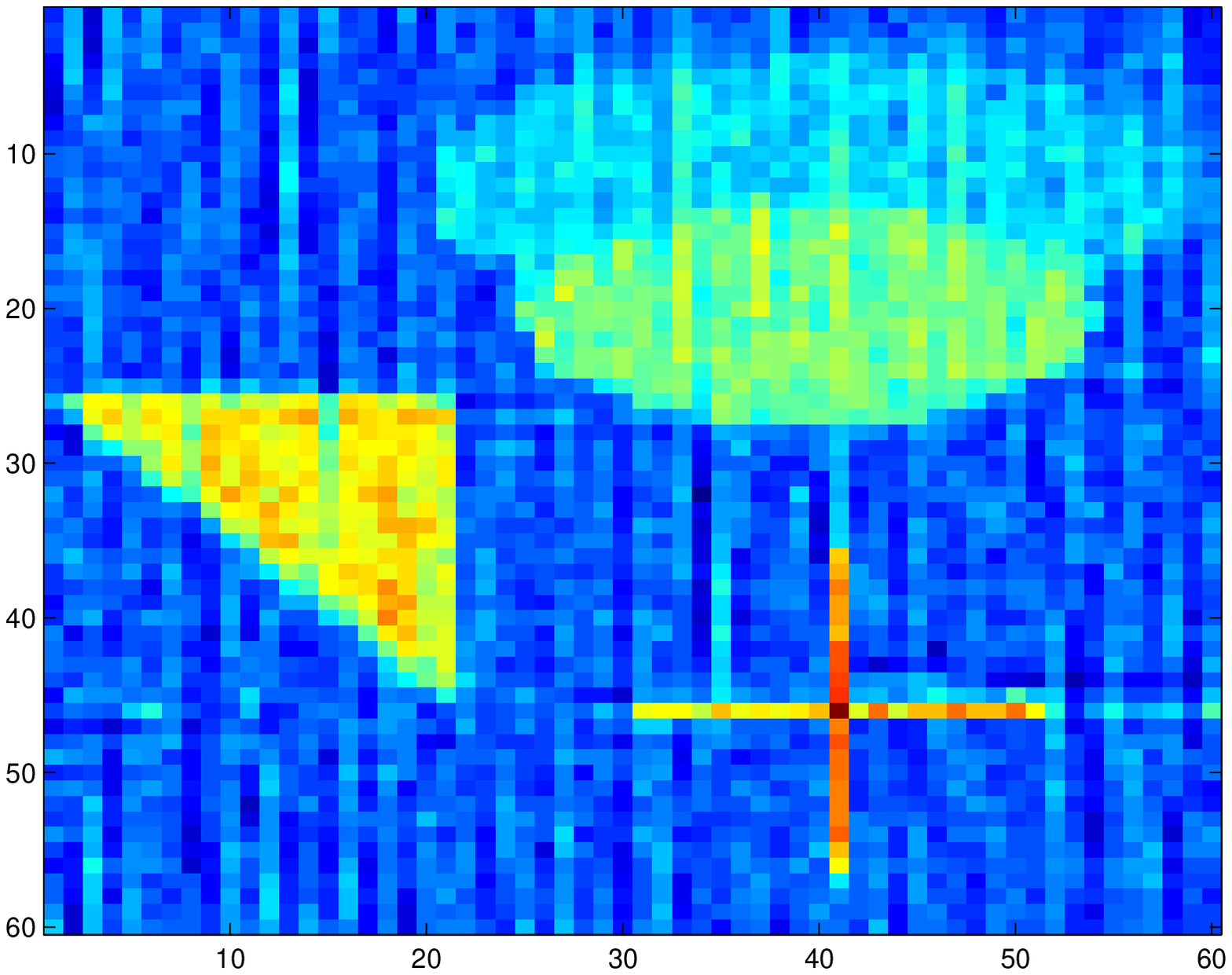}}\\
\subfigure[$3600(.297)$]{\includegraphics[width=.15\textwidth]{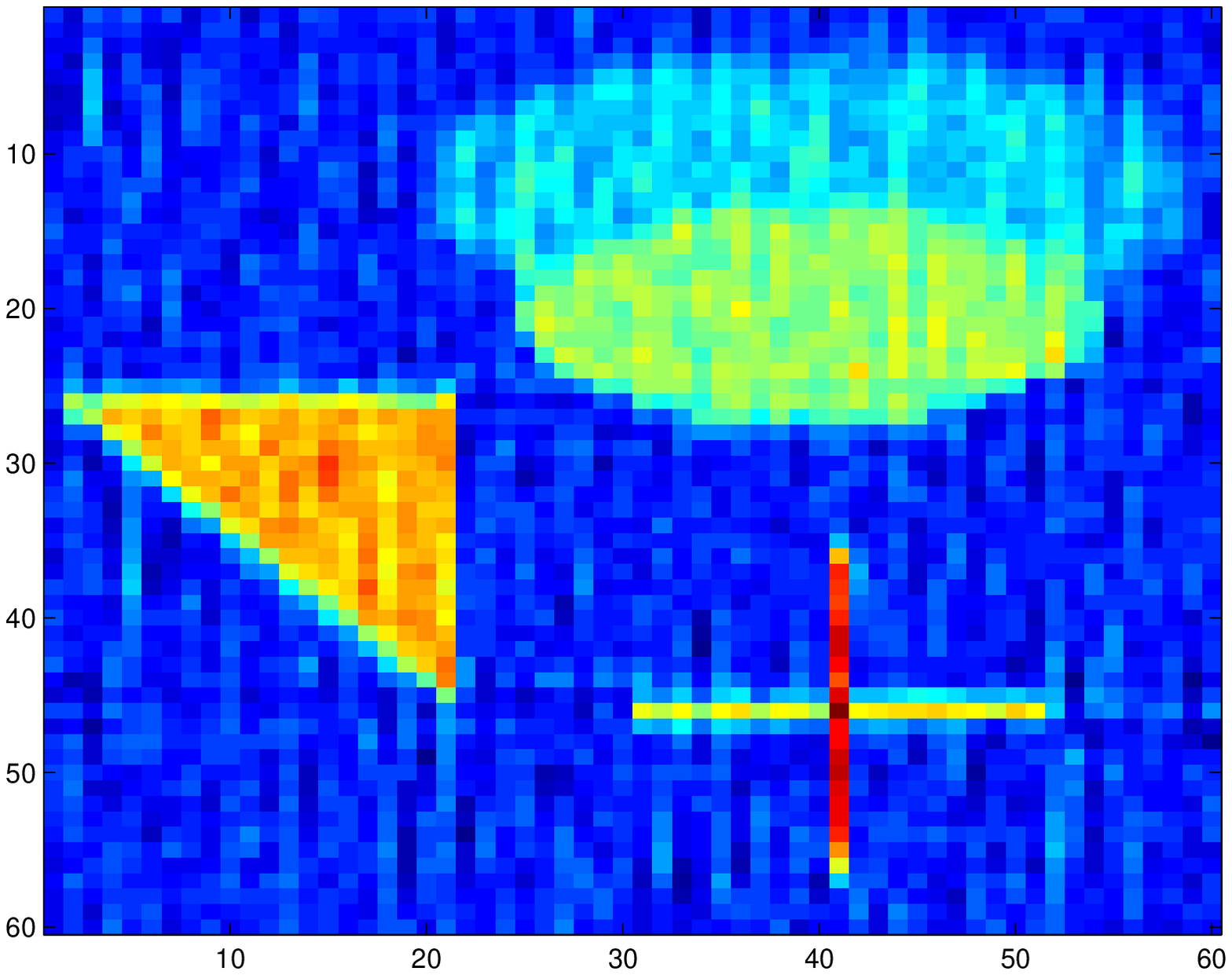}} 
\subfigure[$2700(.325)$]{\includegraphics[width=.15\textwidth]{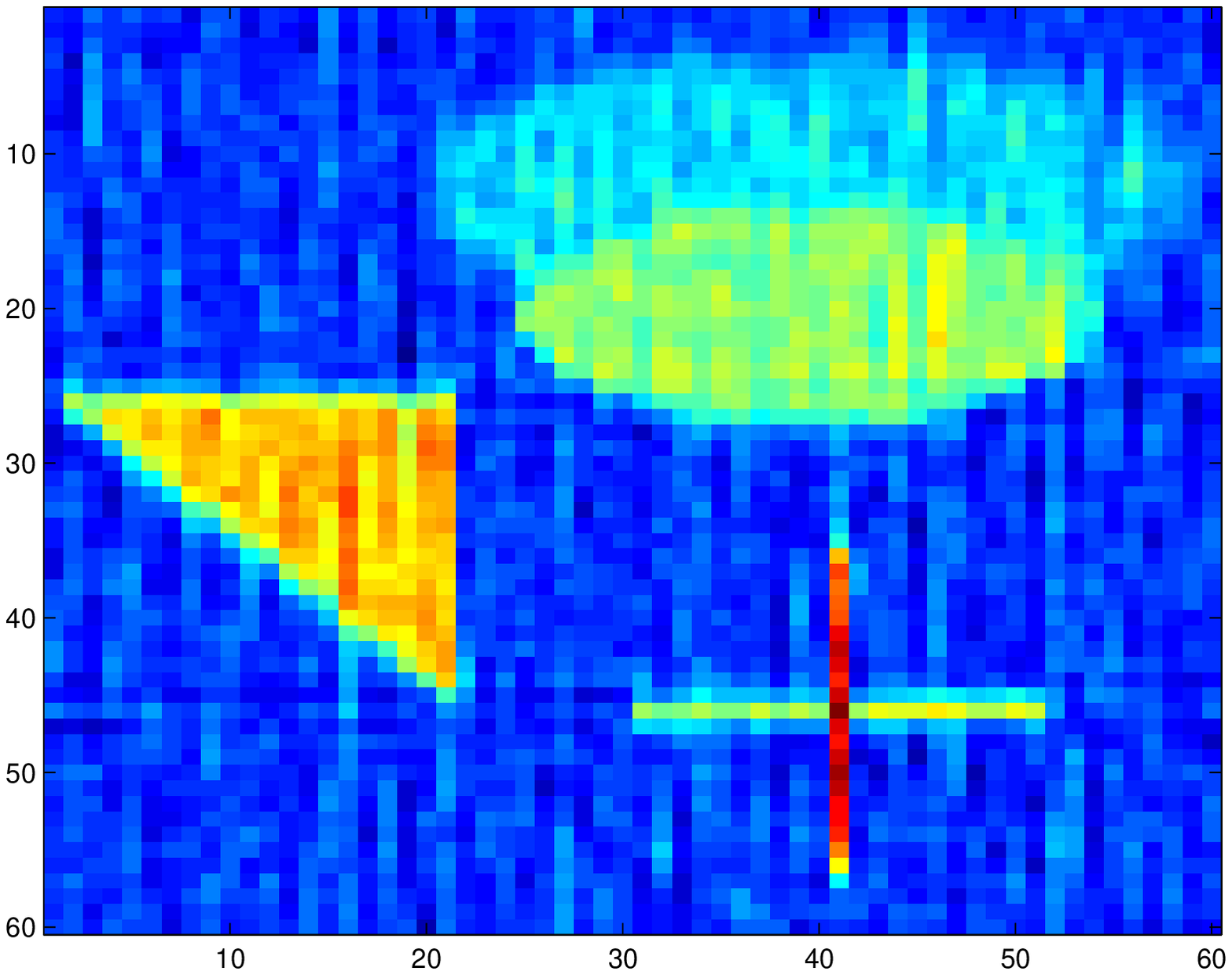}} 
\subfigure[$1800(.380)$]{\includegraphics[width=.15\textwidth]{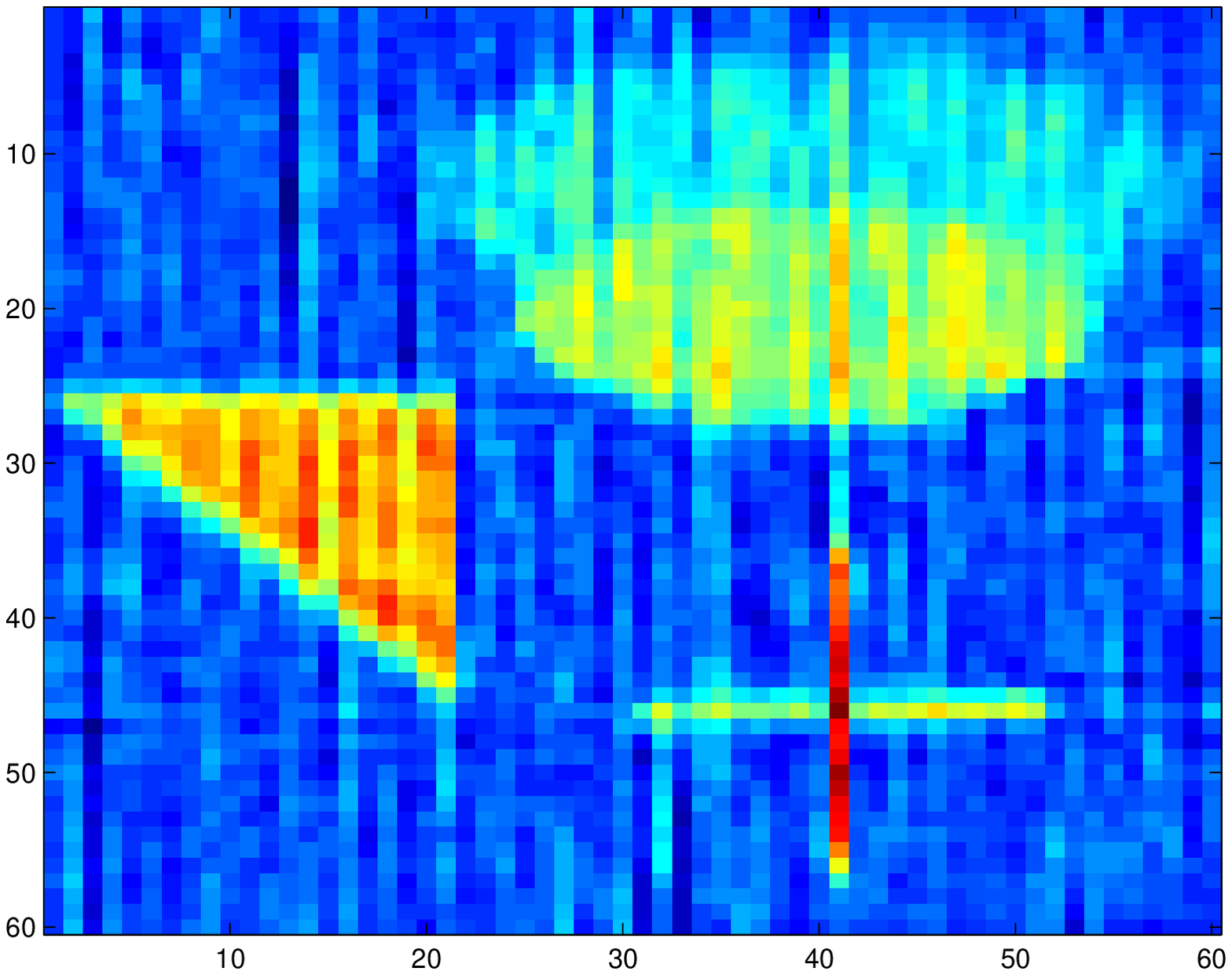}}
\subfigure[$3600(.222)$]{\includegraphics[width=.15\textwidth]{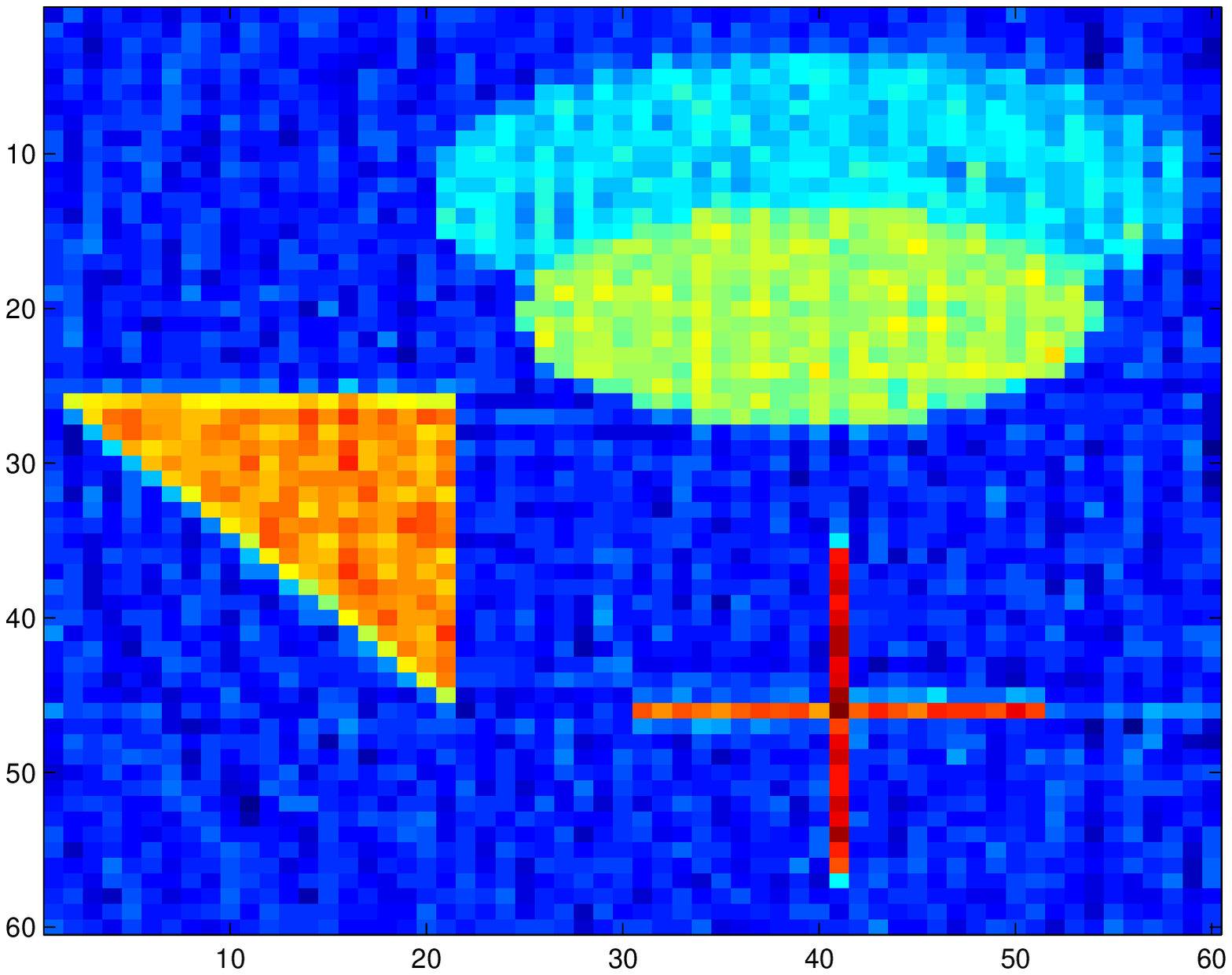}} 
\subfigure[$2700(.263)$]{\includegraphics[width=.15\textwidth]{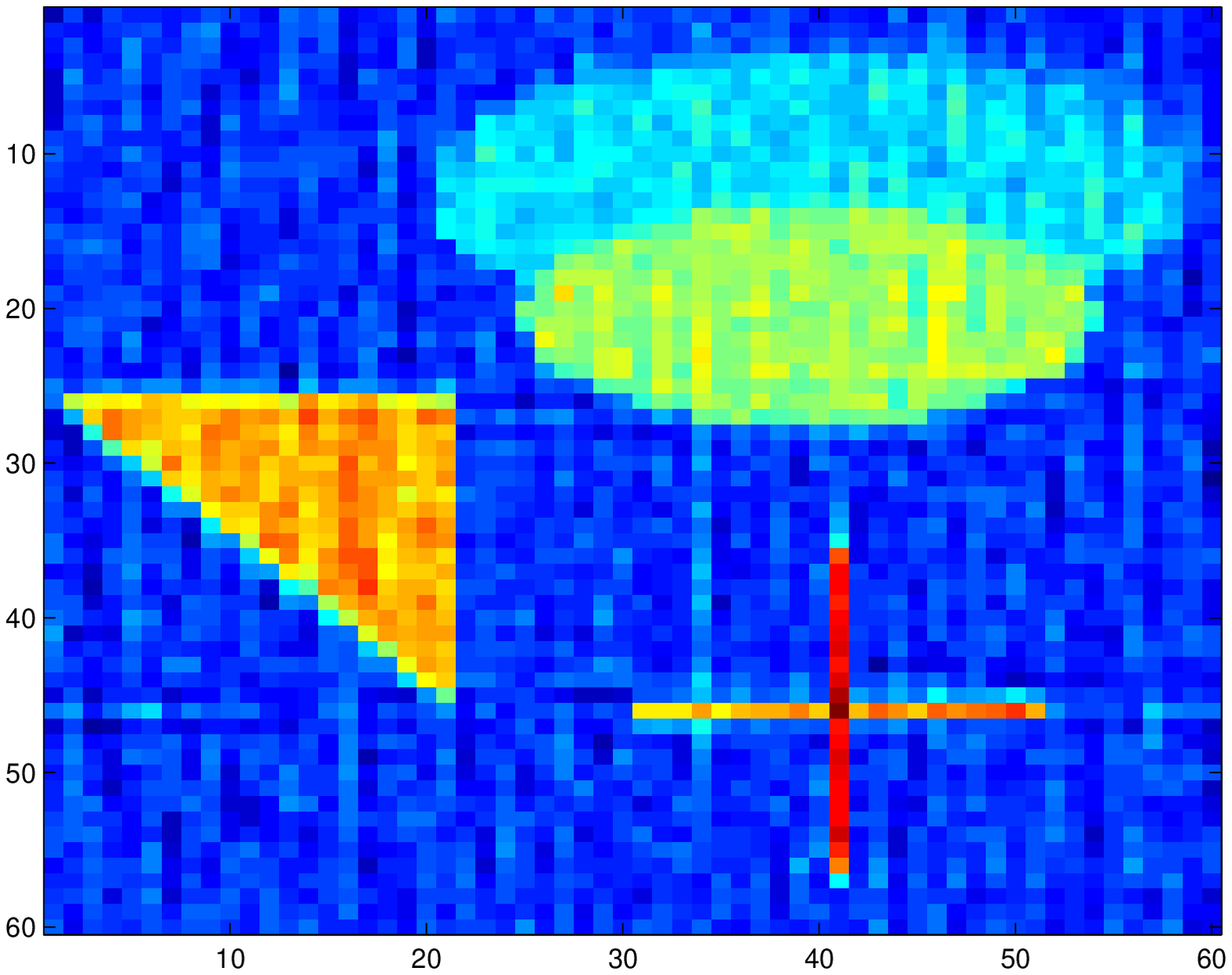}} 
\subfigure[$1800(.325)$]{\includegraphics[width=.15\textwidth]{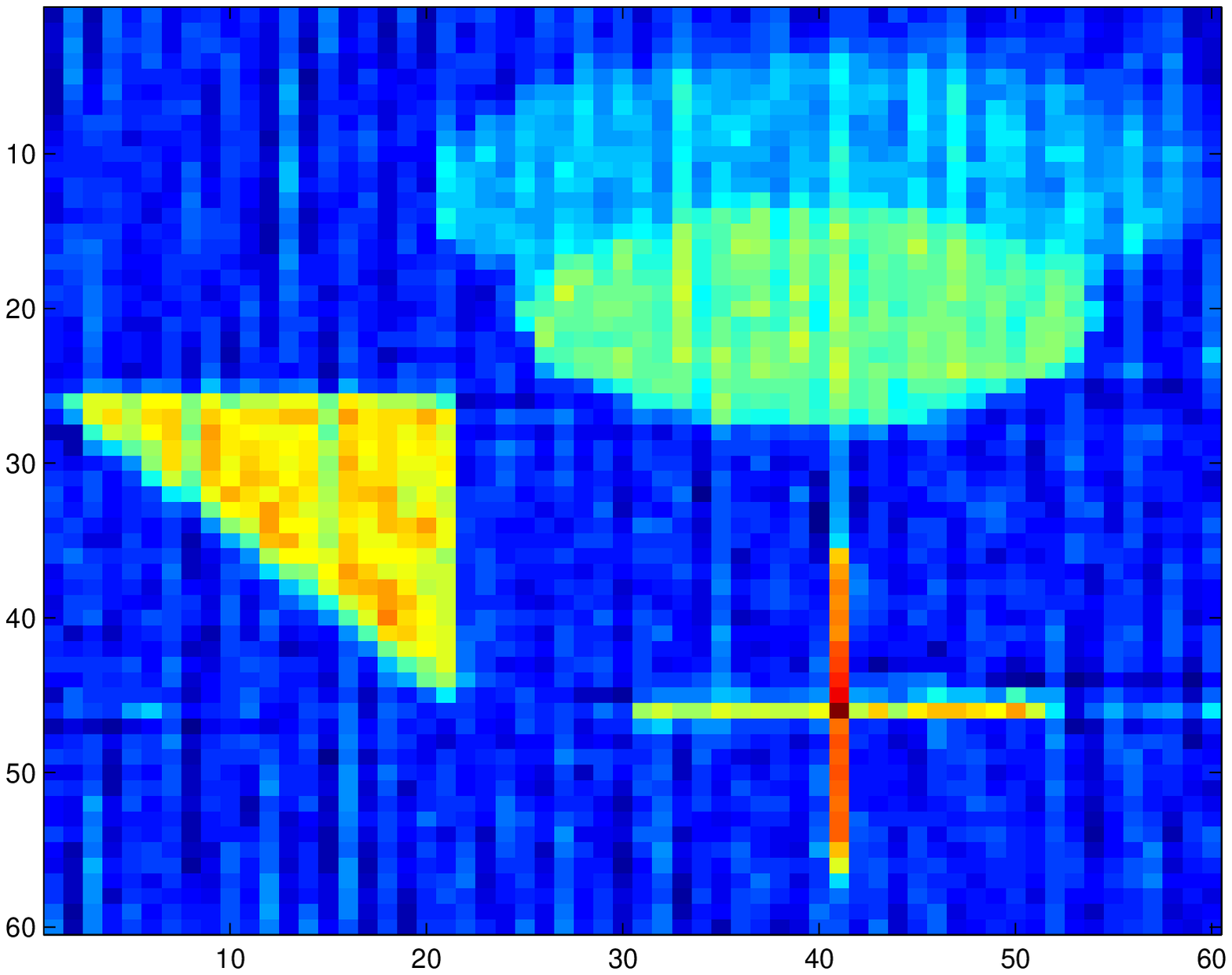}}
\end{center}\caption{Illustrative results in row $1$ for the UPRE, in row $2$ for the GCV and in row $3$ for the $\chi^2$ principle. From left to right problem size $3600$, $2700$ and $1800$, for noise level $.02$ and then $.01$. The label gives the 
the sample and the  relative error $m($error).\label{tomo1fig}}
\end{figure}

\section{Algorithmic Considerations for the Iterative MS stabilizer}\label{MSstab}
The results of Section~\ref{otherresults} demonstrate the relative success of regularization parameter estimation techniques, while also showing that in general with limited data sets improvements may be desirable. Here the iterative technique using the MS stabilizing operator which is frequently used for geophysical data inversion is considered. Its connection with the iteratively regularized norm algorithms, discussed in \cite{WoRo:07}, has apparently not been previously noted in the literature but does demonstrate the convergence of the iteration based on the 
 updating MS stabilizing operator $L$:
\begin{align}\label{MSL}
L^{(k)}&=(\mathrm{diag}((\bfm^{(k-1)}-\bfmo)^2)+\epsilon^2I)^{-1/2}.
\end{align}
Note $L^{(k)}$ is of size $n \times n$ for all $k$, and the use of small $\epsilon^2>0$ assures rank($L^{(k)})=n$, avoiding instability for the components converging to zero, $\bfm_j-(\bfmo)_j \rightarrow 0$. With this $L$ we see that $L=L(\bfm)$ and hence, in the notation of \cite{Zhd}  \eqref{objective2} is of pseudo-quadratic form and the iterative process is required. The iteration to find $\bfm^{(k)}$ from $\bfm^{(k-1)}$, as in \cite{Zhd}, replacing $L$ by $L^{(k)}:=L(\bfm^{(k-1)})$ transforms \eqref{objective2} to a standard Tikhonov regularization, equivalent to the IRN of \cite{WoRo:07}, which can be initialized with $\bfm^{(0)}=0$.    Theorem~\ref{newtheorem} can be  used with $m$ degrees of freedom.

The regularization parameter needed at each iteration can be found by applying any of the noted algorithms, e.g. UPRE, GCV, $\chi^2$, MDP, LC-curve,  at the $k^{\mathrm{th}}$ iteration, using the SVD calculated for the matrix $\tilde{G}(L^{(k)})^{-1}$, here noting that solving  the mapped right preconditioned system is equivalent to solving the original formulation, and avoids the GSVD. 
In particular  \eqref{objective}  in the MS approach is replaced by 
 \begin{align}\label{objectiveMS}
P^{\sigmam}(\bfm)= \|G \bfm-\bfdo\|_{\Cd^{-1}}^2 + (\alpha^{(k)})^2\|L^{(k)}(\bfm-\bfm^{(k-1)})\|,
\end{align}
with $L^{(k)}$ given by \eqref{MSL}, and 
$\{\alpha^{(k)}\}$ found automatically. 
In our experiments for the $2D$ gravity model, we contrast the use of an initial zero estimate of the density with  an initial estimate for $\bfm_0$  obtained from the data and based on the generalized singular values for which the  central form of the $\chi^2$ iteration is better justified.  For the case with prior information, the initial choice for $\alpha^{(1)}$ is picked without consideration of regularization parameter estimation, and  the measures for convergence are for the iterated solutions $\bfm^{(k)}$, $k\ge1$. 
\subsection{Numerical Results: A $2D$ Gravity Model}\label{results}
We contrasted the use of the regularization parameter estimations techniques on an underdetermined  $2D$ gravity model. Figure~\ref{Model}-\ref{Anomaly} shows this model and  its gravity value.
\begin{figure}[!h] 
\begin{center}
\subfigure[Original Model]{\label{Model}\includegraphics[width=.45\textwidth]{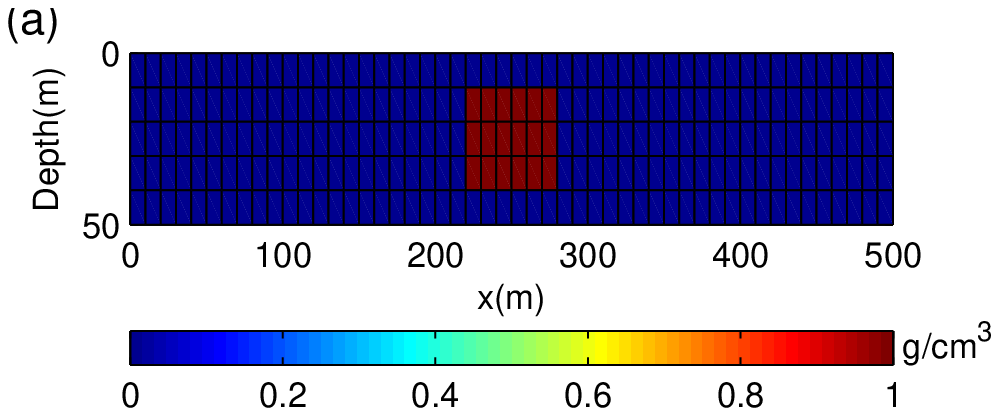}}
\subfigure[Gravity Anomaly]{\label{Anomaly}\includegraphics[width=.45\textwidth]{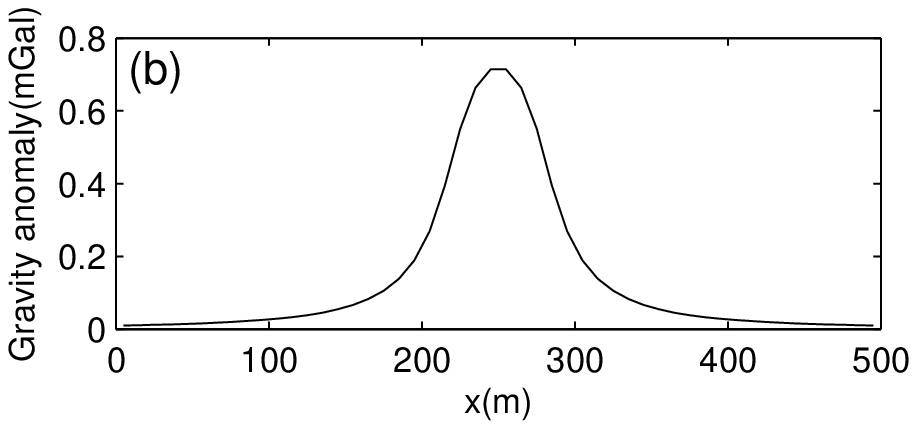}}
\caption{\label{fig1} (a) Model of a  body set in a grid of  square cells each of size $10$m, the density contrast of the body is $1\mathrm{gr}/\mathrm{cm}^3$. (b) The gravity anomaly due to the synthetic model.}
\end{center}
\end{figure} 
The synthetic model is a rectangular body, $60\mathrm{m}\times 30\mathrm{m}$, that has density contrast $1\mathrm{gr}/\mathrm{cm}^3$ with an homogeneous background. Simulation data are calculated at $50$ stations with $10\mathrm{m}$ spacing on the surface. The subsurface is divided into $50\times 5$ cells with $10\mathrm{m}\times 10\mathrm{m}$ dimension, hence in this case $m=50$ and $n=250$. In generating noise-contaminated data we generate a random matrix $\Theta$ of size $m\times 50$, with columns $\Theta^c$, $c=1:50$, using the MATLAB function $\bf{randn}$. Then setting $\bfd^{c}=\bfd+\left(\eta_1\left(\bfd_{\mathrm{exact}}\right)_i+\eta_2\|\bfd_{\mathrm{exact}}\|\right)\Theta^c$, generates $50$ copies of the right-hand vector $\bfd$. Results of this $20\%$ under sampling are presented for $3$ noise levels, namely ($\eta_1=0.01,\eta_2=0.001$;$\eta_1=0.03,\eta_2=0.005$ and $\eta_1=0.05,\eta_2=0.01$).

In the experiments we contrast not only the GCV, UPRE and $\chi^2$ methods, but also the MDP and L-curve which are the standard techniques in the related geophysics literature. Details are as follows:
\begin{description}
\item[Depth Weighting Matrix]  Potential field inversion  requires the inclusion of a depth weighting matrix within the regularization term. We use $\beta=0.6$ in the diagonal matrix $(W_{\mathrm{depth}})_{jj}=z_j^{-\beta}$  for cell $j$ at depth $z_j$, and  at each step form the column scaling through $\tilde{G}^{(k)}=\tilde{G}(W_{\mathrm{depth}})^{-1}(L^{(k)})^{-1}$, where in $L^{(k)}$ $\epsilon=.02$. 
\item[Initialization] 
 In all cases the inversion is initialized with a $\bfm^{(0)}=\bfm_0$ which is obtained as the solution of the regularized problem, with depth weighting, and found for   the  fixed choice  $\alpha^{(0)} = (n/m)\max(\gamma_i)/\mathrm{mean}(\gamma_i)$, using the singular values of the weighted $\tilde{G}$. All subsequent iterations calculate $\alpha^{(k)}$ using the chosen regularization parameter estimation technique. 
 \item[Stopping Criteria] The algorithm terminates for $k=k_{\mathrm{final}}$ when one of the following conditions is met, with $\tau=.01$, 
\subitem(i) a sufficient decrease in the functional is observed, $P^{\alpha^{(k-1)}}-P^{\alpha^{(k)}}< \tau(1+P^{\alpha^{(k)}})$, 
\subitem(ii)  the change in the density satisfies $\|\bfm^{(k-1)}-\bfm^{(k)} \| < \sqrt{\tau} (1 + \|\bfm^{(k)}\|)$,  
\subitem(iii)  a maximum number of iterations, $K$, is reached, here $K=20$.
\item[Bound Constraints] Based on practical knowledge the density is constrained to lie between $[0,1]$ and any values outside the interval are projected to the closest bound. 

\item[$\chi^2$ algorithm]  The Newton algorithm used for the  $\chi^2$ algorithm is iterated to tolerance determined by a confidence interval $\theta=.95$ in \eqref{rangeP}, dependent on the number of degrees of freedom, corresponding to $0.6271$ for $50$ degrees of  freedom. The  maximum degrees of freedom is adjusted dynamically dependent on the number of significant singular values, with the tolerance adjusted at the same time. We note that the number of degrees of freedom often drops with the iteration and thus the tolerance increases, but that the difference in results with choosing lower tolerance $\theta=.90$, leads to almost negligible change in the results. 
\item[Exploring $\alpha$] At each step for the L-curve, MDP, GCV and UPRE, the solution is found at each iteration for $1000$ choices of $\alpha$ over a range dictated by the current singular values, see the discussion for the L-cuve in e.g. \cite{ABT, Regtools}. 
\item[MDP algorithm] To find $\alpha$ by the MDP we interpolate $\alpha$ against the  weighted residual for $1000$ values for $\alpha$ and use the Matlab function \texttt{interp1} to find the $\alpha$ which solves for the degrees of freedom. Here we use $\delta=m$ so as to avoid the complication in the comparison of how to scale $M$, i.e. we use $\rho=1$. 
\end{description}
  Tables~\ref{tab1}-\ref{tab3}  contrast the performance of the $\chi^2$ discrepancy, MDP, LC, GCV and UPRE methods with respect to relative error, $ {\|\left(\bfm_{\mathrm{exact}}-\bfm^{(K)}\right)\|_2}/{\|\bfm_{\mathrm{exact}}\|_2}$, and the average regularization parameter calculated at the final iteration.   We also record the average number of iterations required to convergence. In Table~\ref{tab1} we also give the relative errors after one just one iteration of the MS and with the zero initial condition. 
  \begin{table}[!h]
 \begin{center}
\caption{Mean and standard deviation of the relative error measured in the $2-$norm with respect to the known solution over 50 runs. Again the \textit{best} results in each case are indicated by the boldface entries.}\label{tab1}
\vspace{0.25cm}
\begin{tabular}{|c|c|c|c|c|c|}\hline
&  \multicolumn{5}{c|}{Method}  \\ \hline 
Noise &UPRE  &GCV    &$\chi^2$  &MDP  &LC  \\  \hline  
$\eta_1 ,\eta_2$ &  \multicolumn{5}{c|}{Results after just one step, non zero initial condition} \\ \hline
$0.01,0.001$ &$.331(.008)$&$\textbf{.325(.008)}$&$\textbf{.325(.008)}$&$.355(.009)$&$.447(.055)$  \\
$0.03,0.005$ &$\textbf{.353(.019)}$&$.354(.042)$&$.361(.020)$&$.418(.025)$&$.374(.052)$\\
$0.05,0.01$   &$\textbf{.392(.034)}$&$.409(.062)$&$.416(.040)$&$.478(.043)$&$.463(.067)$\\ \hline
&  \multicolumn{5}{c|}{Results using the non zero initial condition} \\ \hline
$0.01,0.001$ &$.323( .009)$&$.314(.010)$&$\textbf{.317(.009)}$&$.352(.011)$&$.489(.078)$   \\
$0.03,0.005$ &$\textbf{.339(.022)}$&$ .338(.040)$&$.359(.022 )$&$.413(.026)$&$.369(.053)$  \\
$0.05,0.01$   &$\textbf{.374(.041)}$&$.393(.068)$&$.414(.041)$&$.470(.046)$&$.460(.070)$  \\ \hline
&  \multicolumn{5}{c|}{Results using the initial condition $\bfm_0=0$} \\ \hline
$0.01,0.001$ &$.322(.001)$&$.312(.011)$&$\textbf{.315(.001)}$&$.359(.009)$&$.593(.014)$\\
$0.03,0.005$ &$\textbf{.333(.020)}$&$.334(.037)$&$.352(.021)$&$.425(.026)$&$.451(.067)$\\
$0.05,0.01$   &$\textbf{.357(.030)}$&$.440(.086)$&$.388(.034)$&$.477(.046)$&$.487(.082)$\\ \hline
\end{tabular}
\end{center}
\end{table}
 
%
\begin{table}[!h]
\begin{center}
\caption{Mean and standard deviation of $\alpha^{k_{\mathrm{final}}}$ over $50$ runs.}\label{tab2}
\vspace{0.25cm}
\begin{tabular}{|c|c|c|c|c|c|}\hline
Noise &  \multicolumn{5}{c|}{Method}  \\ \hline
$\eta_1 ,\eta_2$ &UPRE  &GCV    &$\chi^2$  &MDP  &LC  \\  \hline  
$0.01 ,0.001$ &$35.49(5.17)$&$14.93(9.85)$&$91.32(30.56)$&$53.42(7.43)\,\,$&$3.83(1.91)$ \\
$0.03 ,0.005$ &$11.02(2.62)$&$\,\, 4.18(2.58)$&$72.71(32.72 )$&$30.90(8.21)\,\,$&$0.86(0.08)$ \\
$0.05 ,0.01$   &$\,\, 7.64(4.35)$&$\,\, 3.28(3.05)$&$89.79(27.97)$&$25.81(14.87)$&$0.46(0.05)$\\ \hline
\end{tabular}
\end{center}
\end{table}
\begin{table}[!h]
\begin{center}
\caption{Mean and standard deviation of the number of iterations $k_{\mathrm{final}}$ to meet the convergence criteria over $50$ runs. Again the \textit{best} results in each case are indicated by the boldface entries.}\label{tab3}
\vspace{0.25cm}
\begin{tabular}{|c|c|c|c|c|c|}\hline
Noise &  \multicolumn{5}{c|}{Method}  \\ \hline
$\eta_1 ,\eta_2$ &UPRE  &GCV    &$\chi^2$  &MDP  &LC  \\  \hline  
$0.01 ,0.001$ &$18.94( 0.31)$&$14.78(5.83)$&$16.26(3.00 )$&$18.32(1.10)$&$ \textbf{6.30(1.64)}$   \\
$0.03 ,0.005$ &$11.90( 2.76)$&$\,\, 9.22(2.86)$&$ \,\, \textbf{5.50(1.39)}$ &$\,\, 7.68(1.80)$&$7.90(2.48)$ \\
$0.05 ,0.01$  & $\,\, 7.82(1.73)$   &$\,\, 8.22(2.41)$&$\,\, \textbf{5.10(0.58)}$ &$\,\, 5.72( 0.97)$&$7.84(2.41)$ \\ \hline
\end{tabular}
\end{center}
\end{table}
  \begin{figure}[!h] 
\begin{center}
\subfigure[Initial Gravity]{\label{IGA}\includegraphics[height=.07\textheight,width=.45\textwidth]{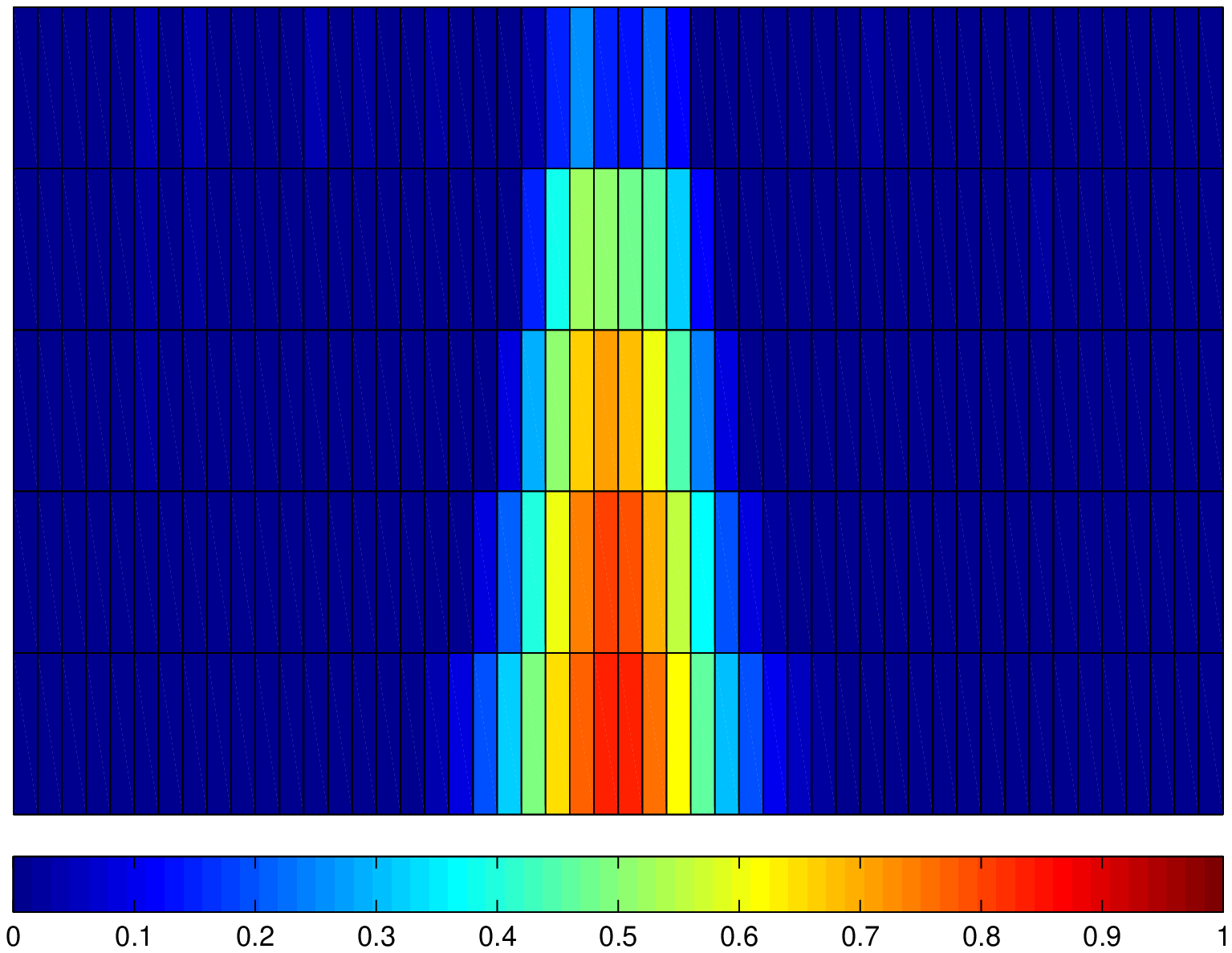}}
\subfigure[Initial Gravity]{\label{IGB}\includegraphics[height=.07\textheight,width=.45\textwidth]{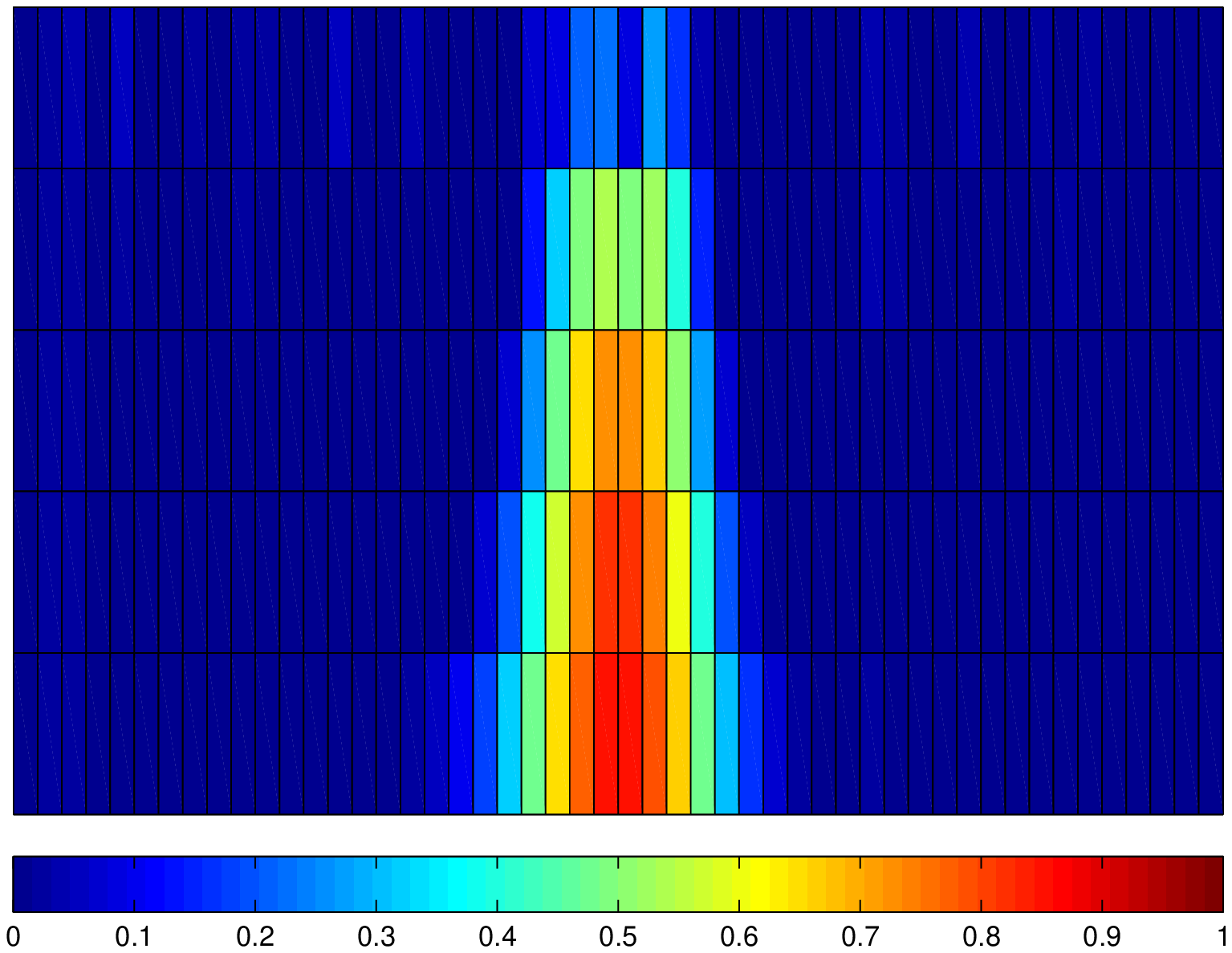}}
\subfigure[UPRE: error $.3181$]{\label{UPa}\includegraphics[height=.07\textheight,width=.45\textwidth]{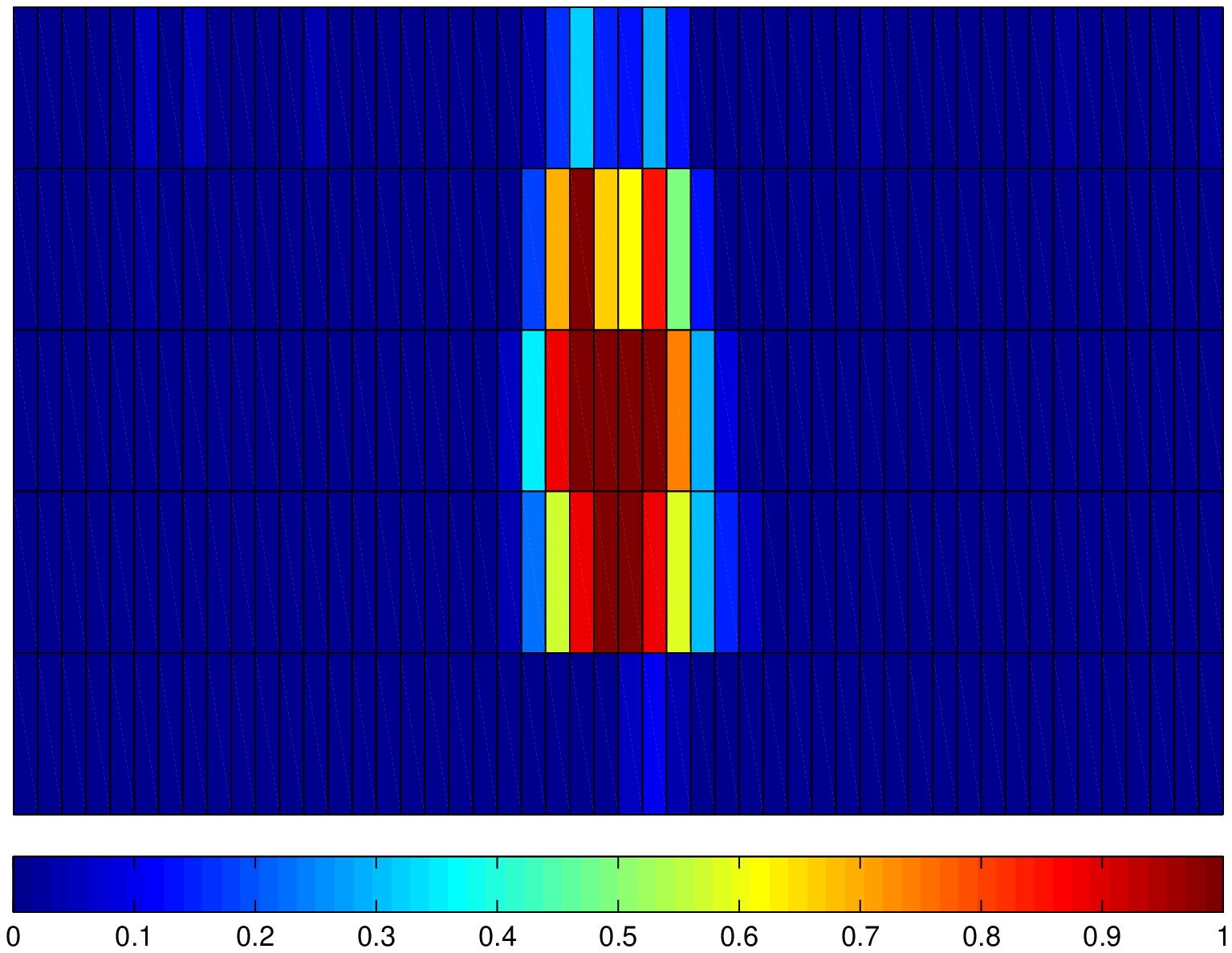}}
\subfigure[UPRE: error $.3122$]{\label{UPb}\includegraphics[height=.07\textheight,width=.45\textwidth]{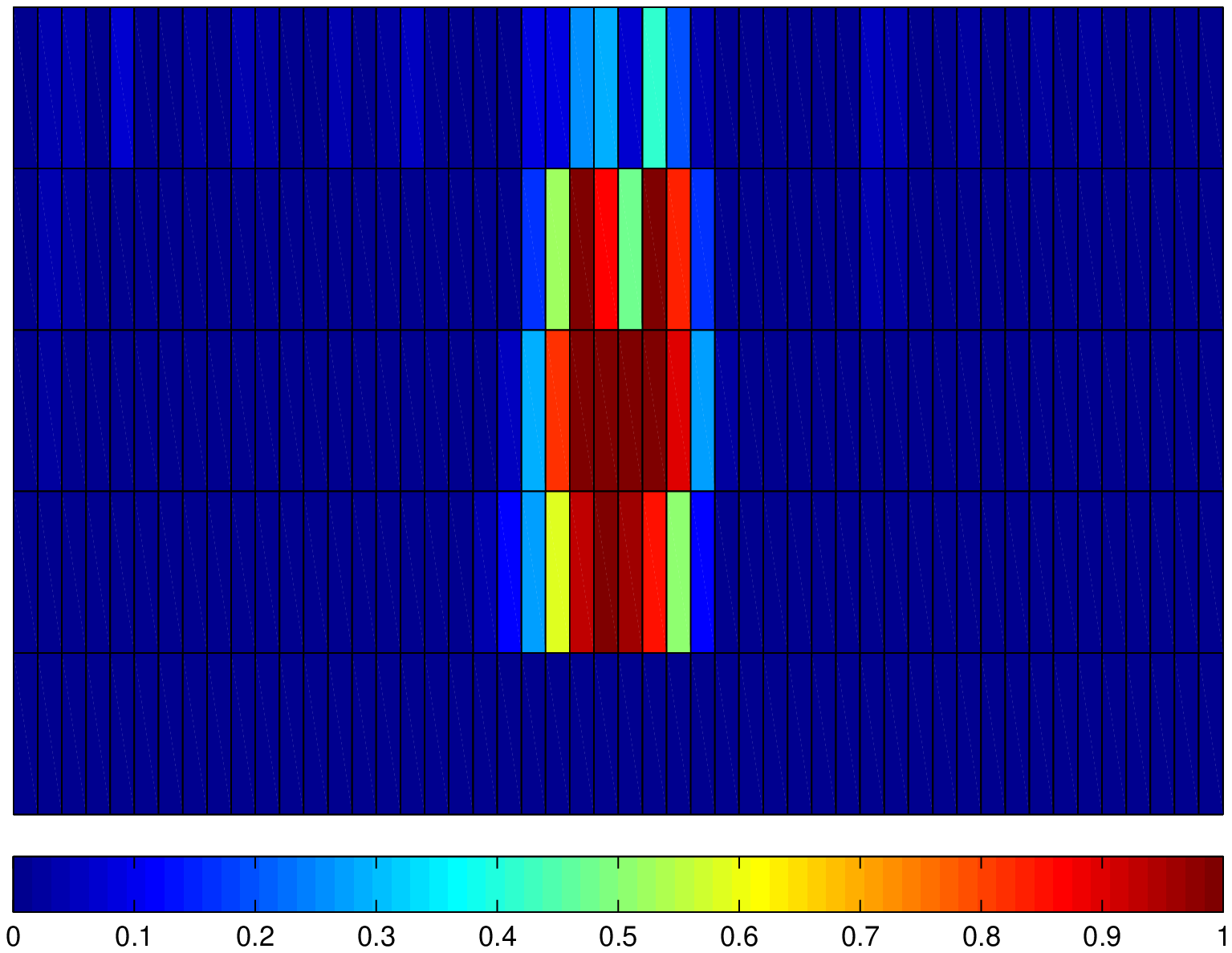}}
\subfigure[GCV: error $.3196$]{\label{GCVa}\includegraphics[height=.07\textheight,width=.45\textwidth]{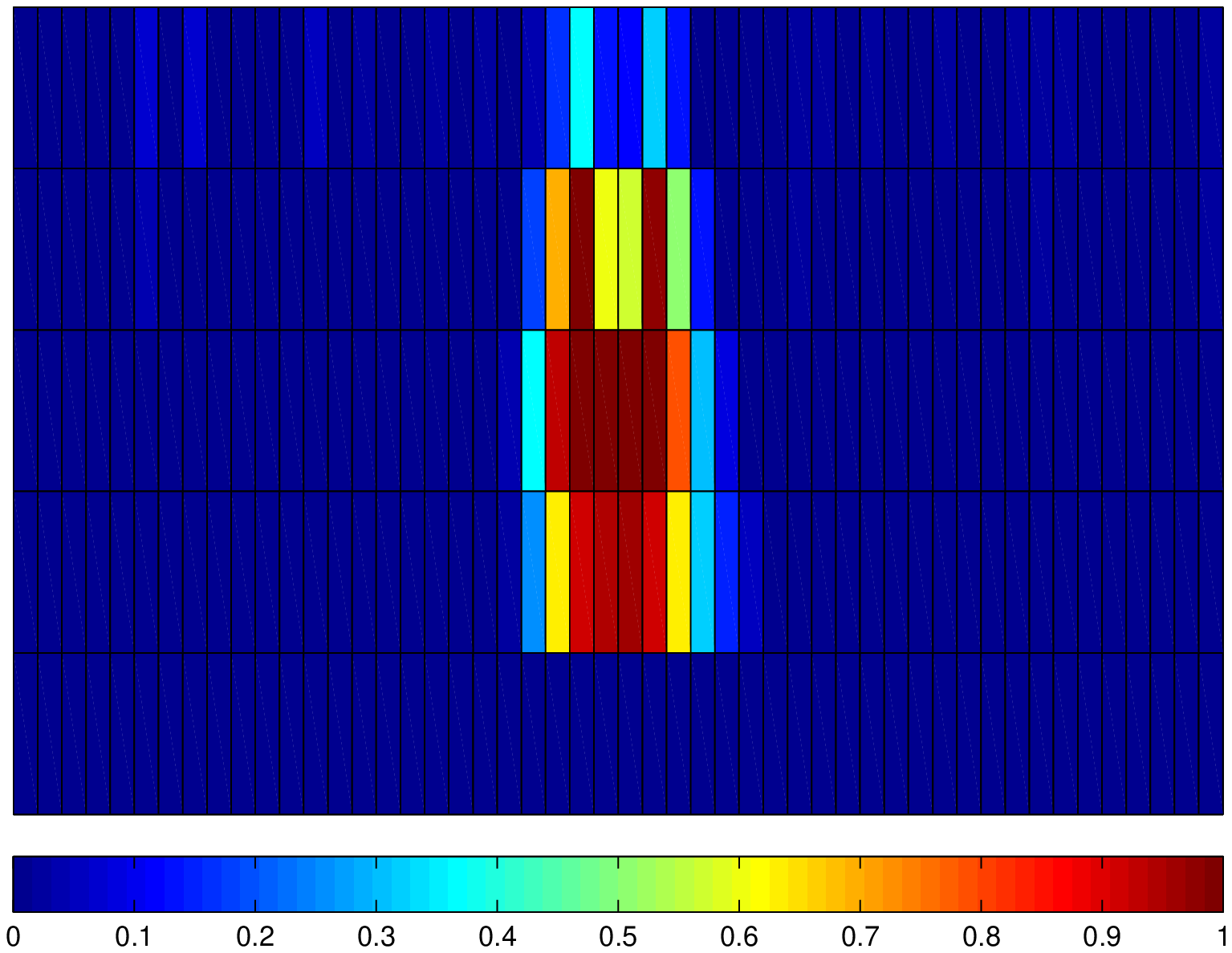}}
\subfigure[GCV: error $.3747$]{\label{GCVb}\includegraphics[height=.07\textheight,width=.45\textwidth]{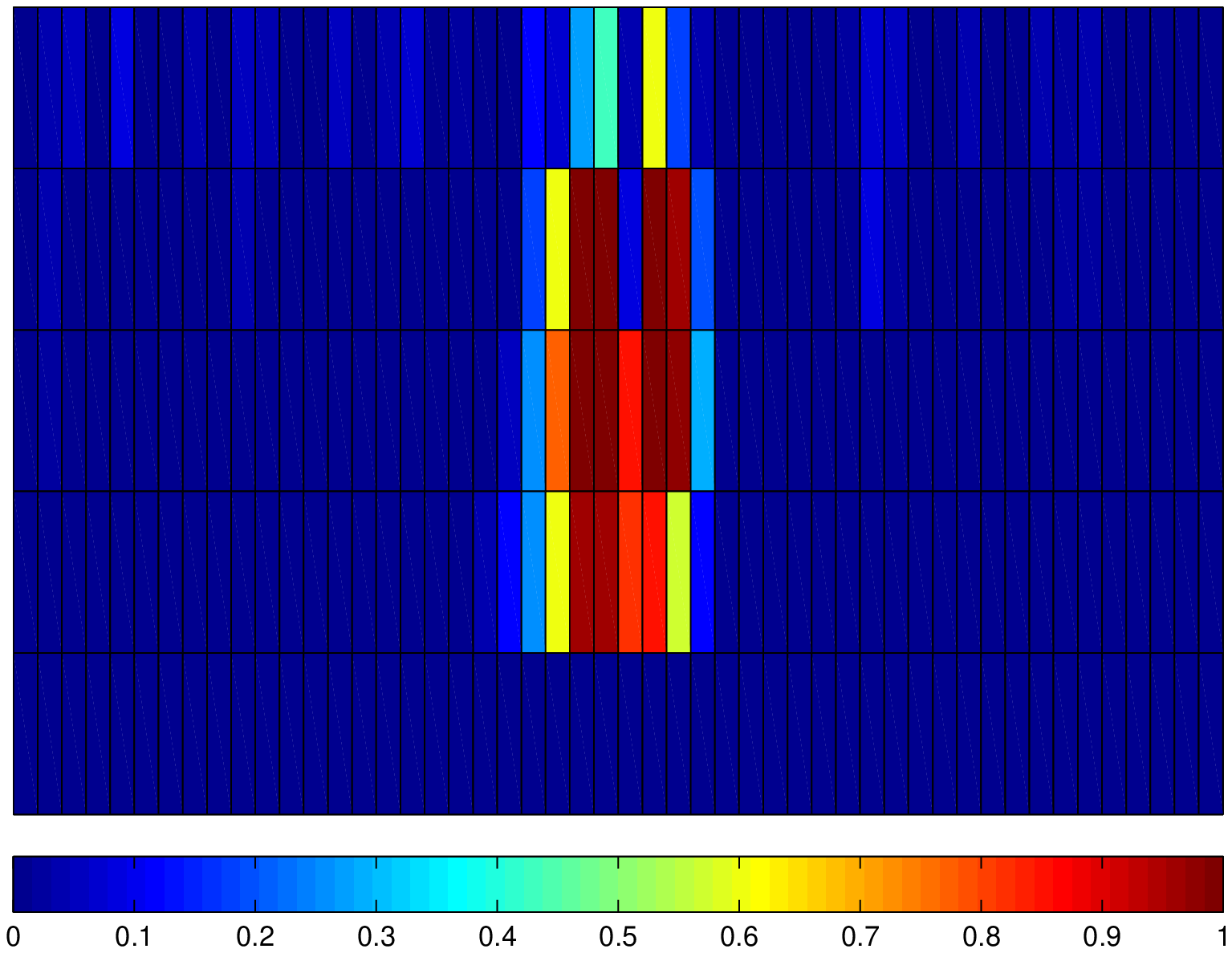}}
\subfigure[$\chi^2$: error $.3381$]{\label{Chia}\includegraphics[height=.07\textheight,width=.45\textwidth]{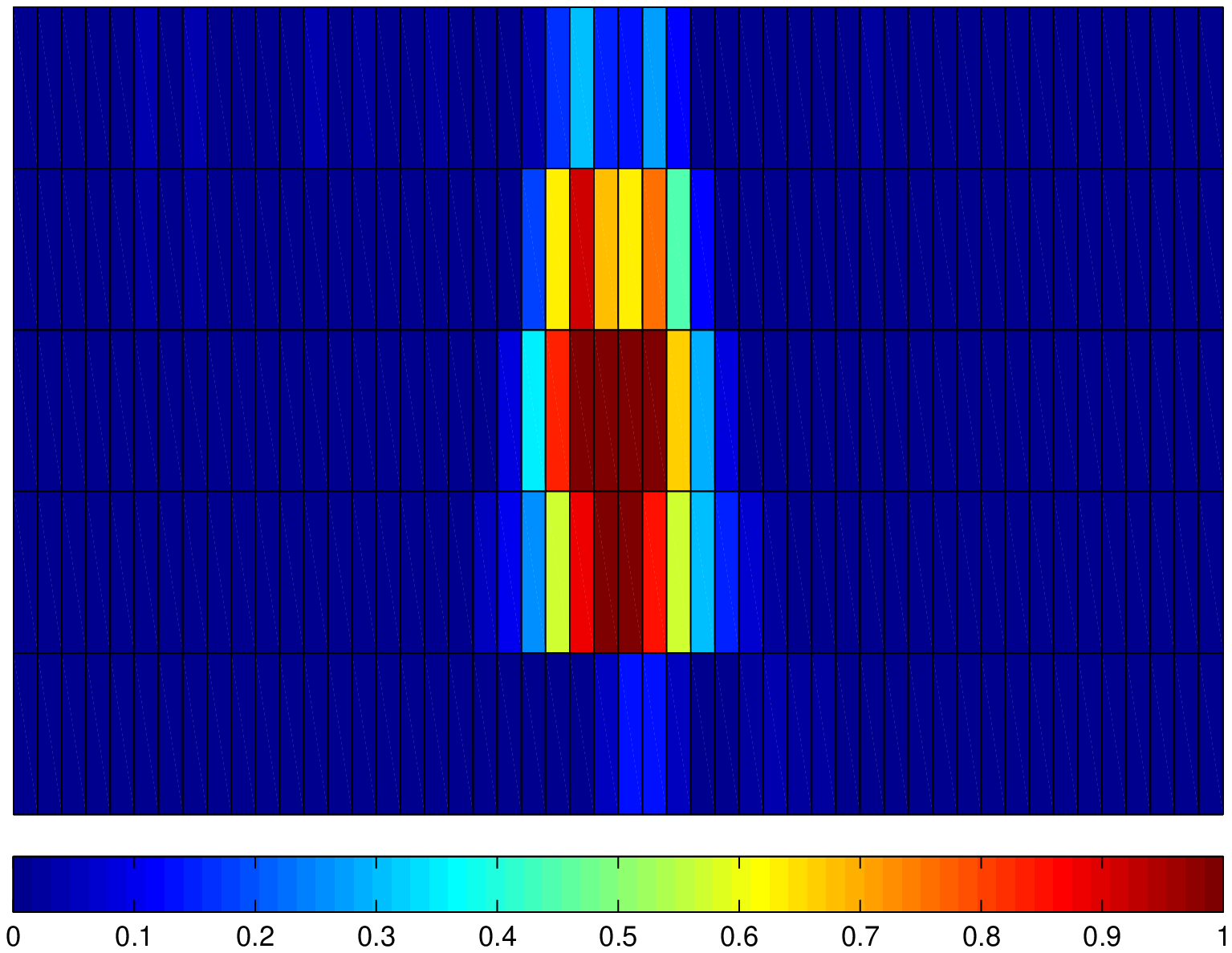}}
\subfigure[$\chi^2$: error $.3154$]{\label{Chib}\includegraphics[height=.07\textheight,width=.45\textwidth]{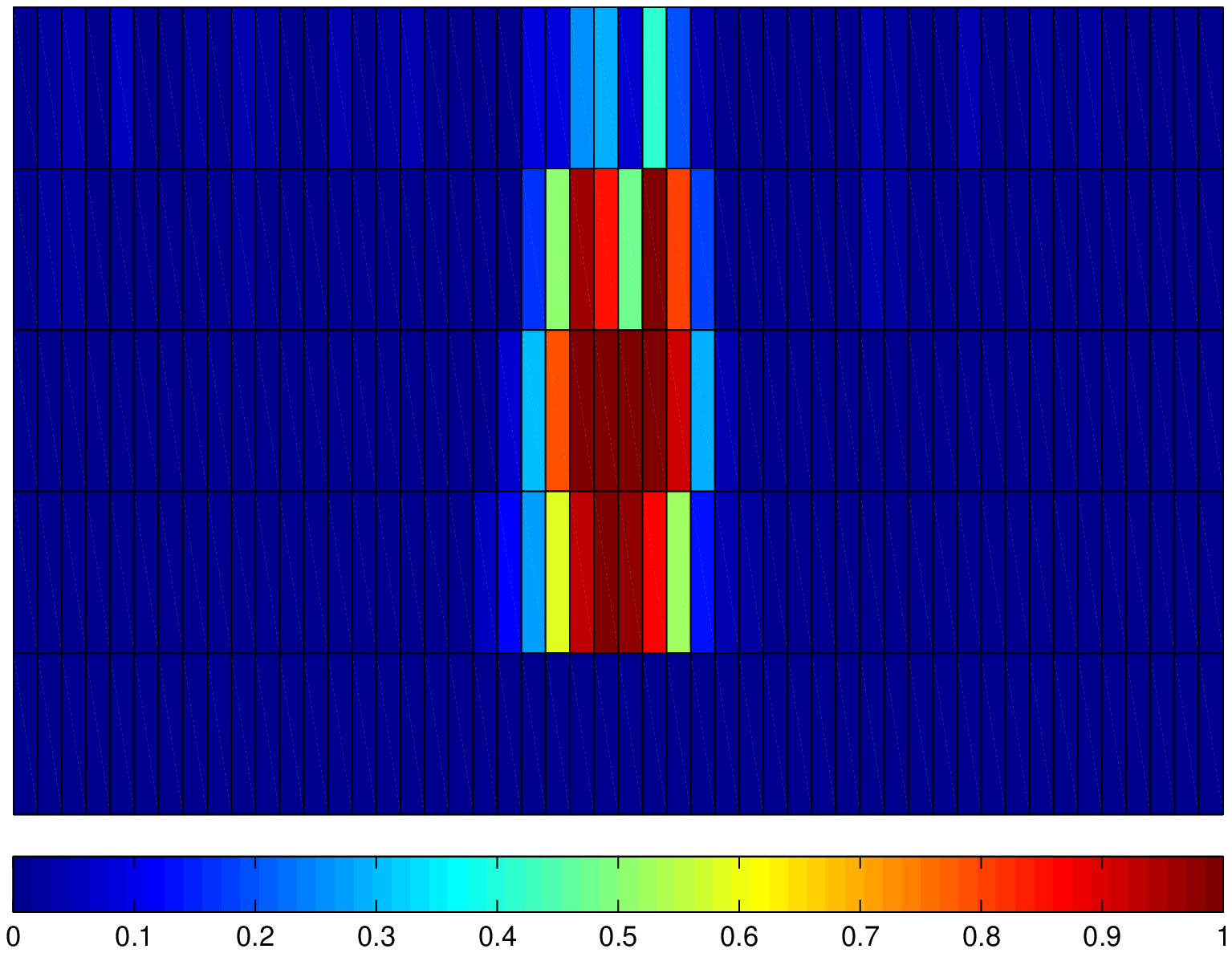}}
\subfigure[MDP: error $.3351$ ]{\label{MDPa}\includegraphics[height=.07\textheight,width=.45\textwidth]{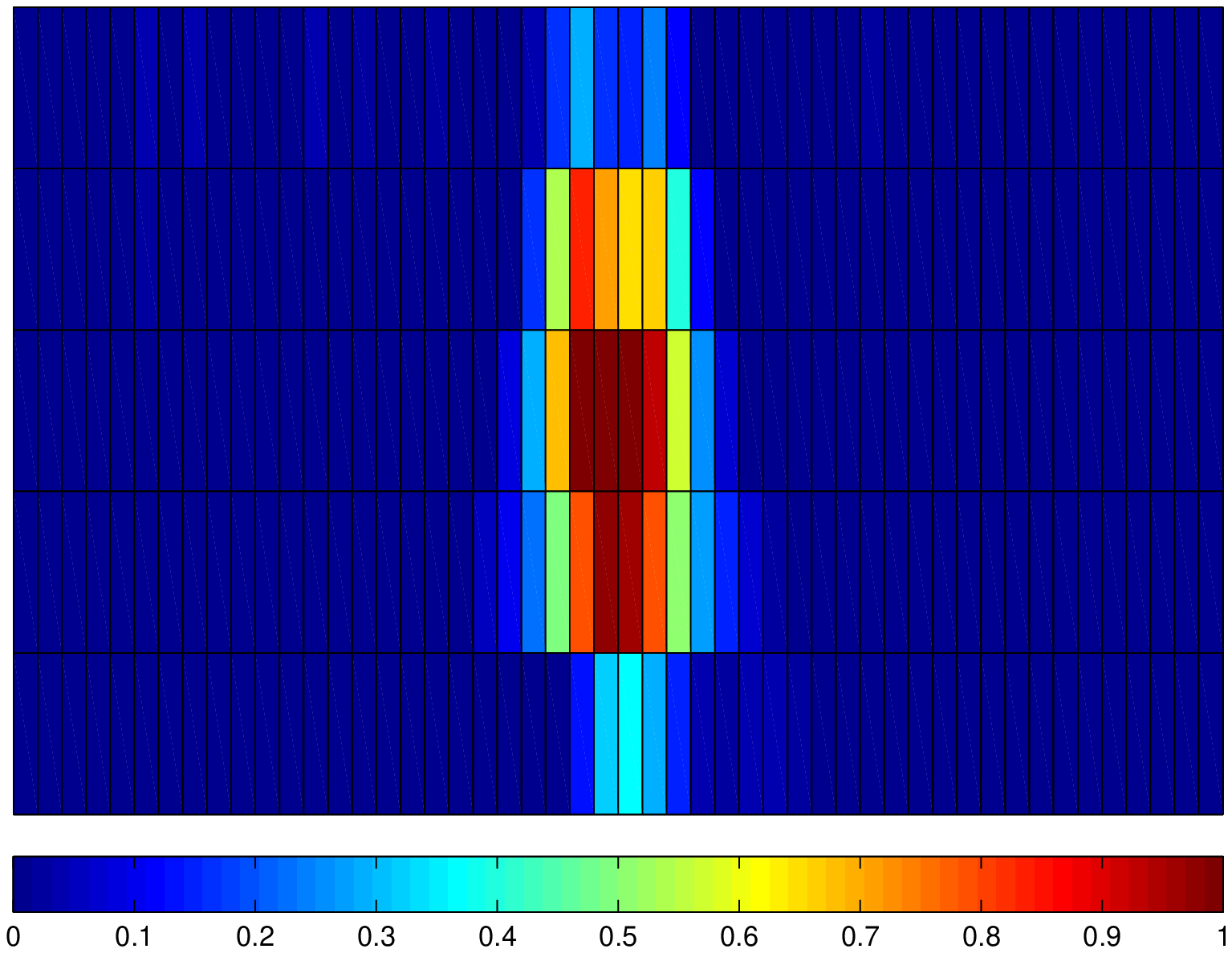}}
\subfigure[MDP: error $.3930$]{\label{MDPb}\includegraphics[height=.07\textheight,width=.45\textwidth]{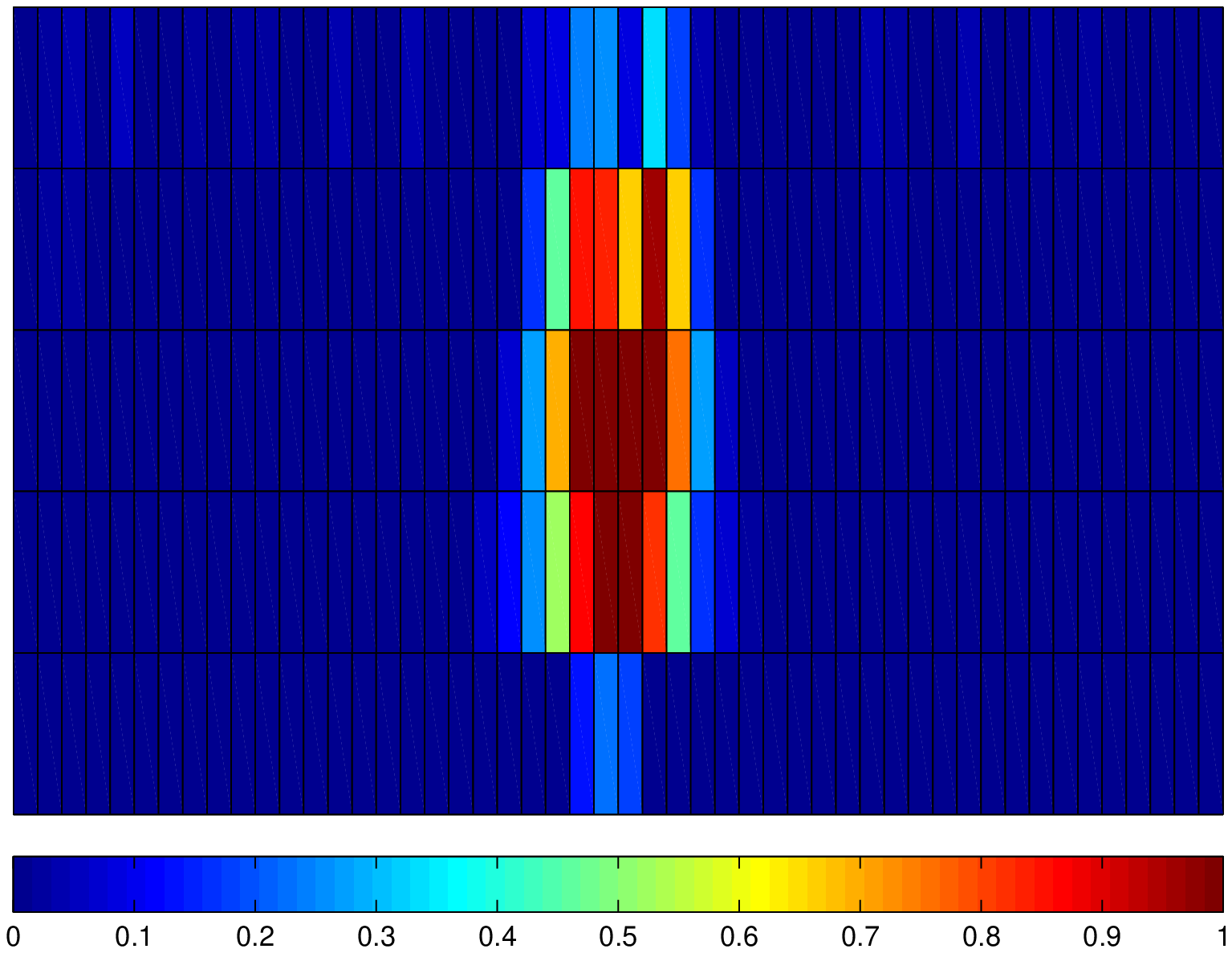}}
\subfigure[LC: error $.3328$]{\label{LCa}\includegraphics[height=.07\textheight,width=.45\textwidth]{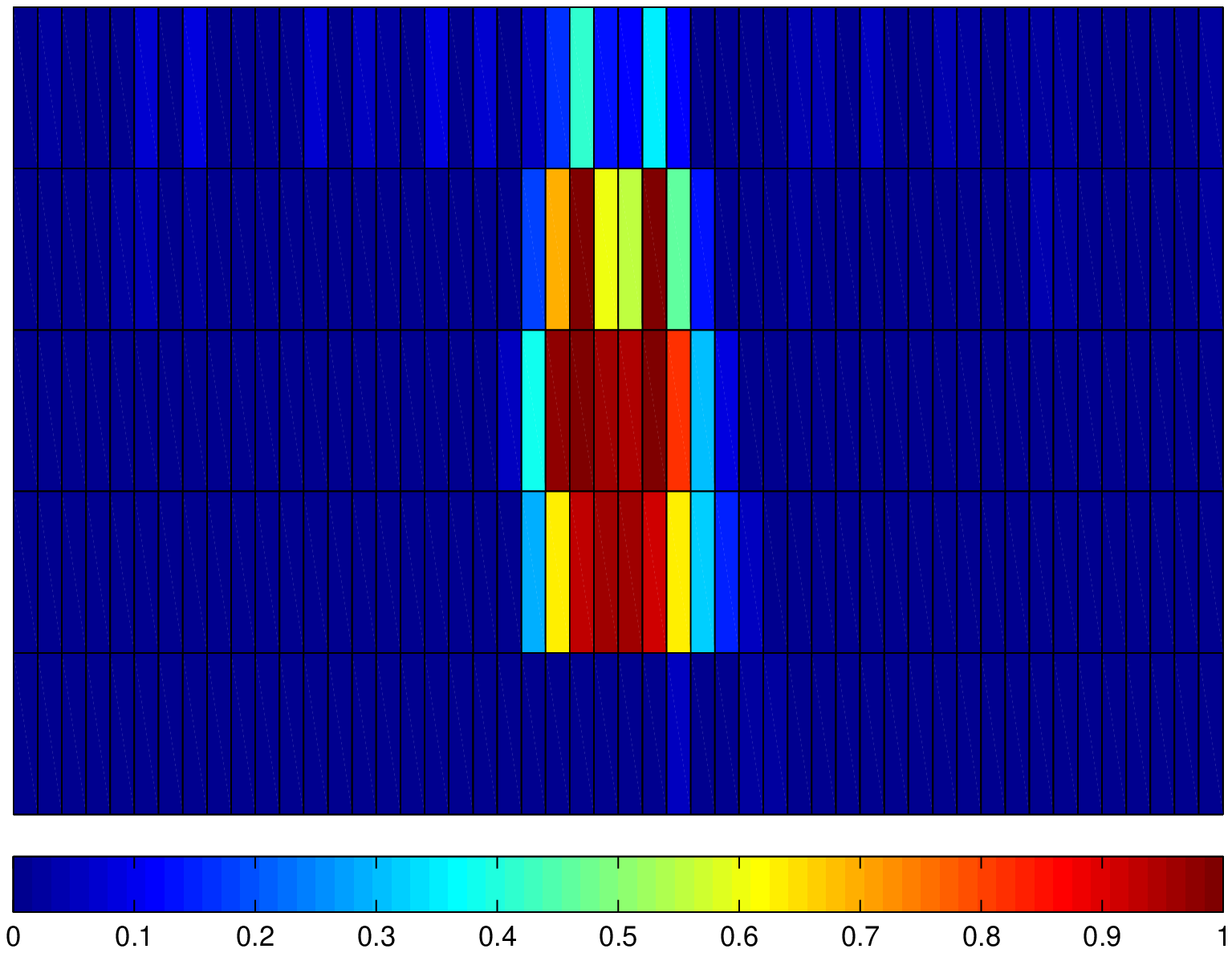}}
\subfigure[LC: error $.4032$]{\label{LCb}\includegraphics[height=.07\textheight,width=.45\textwidth]{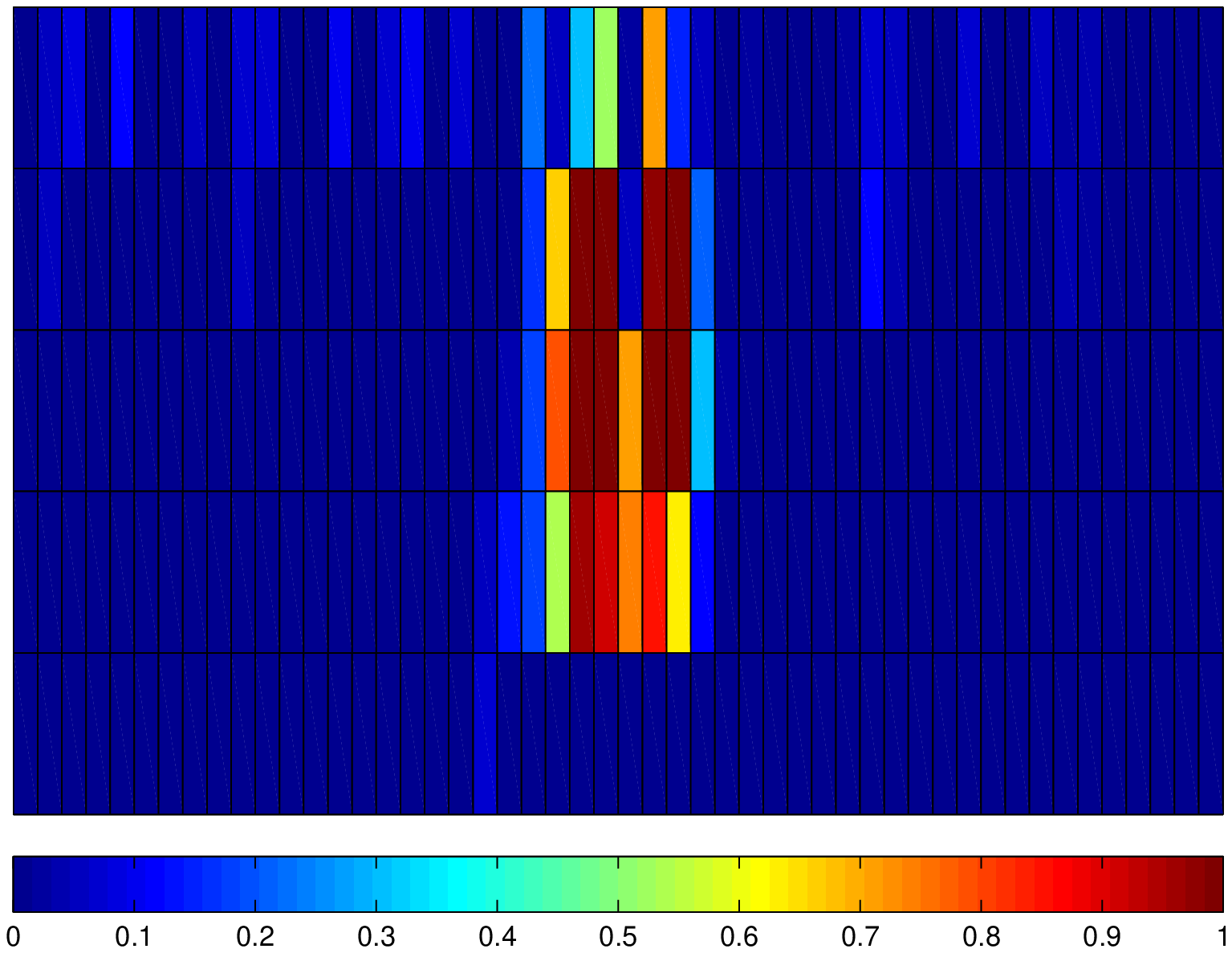}}

\caption{\label{fig2} Density model obtained from inverting the noise-contaminated data. The regularization parameter was found using the UPRE  in \ref{UPa}-\ref{UPb}, the GCV in  \ref{GCVa}-\ref{GCVb}, the $\chi^2$ in \ref{Chia}-\ref{Chib}, the MDP  in \ref{MDPa}-\ref{MDPb}, and the L-Curve in \ref{LCa}-\ref{LCb}. In each case the initial value $\bfm_0^{(0)}$ is illustrated in \ref{IGA}-\ref{IGB}, respectively. The data are two cases with noise level, $\eta_1=.03$ and $\eta_2=.005$, with on the left a typical result, sample $37$ and  and on the right  one of the few cases of $50$ with sometimes larger error, sample $22$. One can see that results are overall either consistently good or consistently poor, except that the $\chi^2$ and UPRE results are not bad in either case.}
\end{center}
\end{figure} 
  
With respect to the relative error one can see that  the error increases with the noise level, except that the L-curve appears to solve the second noise level situation with more accuracy. In most regards the UPRE, GCV and $\chi^2$ methods behavior similarly, with relative stability of the error (smaller standard deviation in the error), and increasing error with noise level. On the other hand, the final value of the regularization parameter is not a good indicator of whether a solution is over or under smoothed, contrast e.g. the $\chi^2$ and GCV methods. The $\chi^2$ method is overall cheaper, fewer iterations are required and the cost per iteration is cheap, not relying on an exploration with respect to $\alpha$, interpolation or function minimization. The illustrated results in Figure~\ref{fig2}, for the second noise level, $\eta_1=.03$ and $\eta_2=.005$, for a typical result, sample $37$ and   one of the few cases from $50$ with larger error, sample $22$, demonstrate that all methods achieve some degree of acceptable solution with respect to moving from an initial estimate which is inadequate to a more refined solution.  In all cases the geometry and density of the reconstructed models are close to those of the original model. 

To demonstrate that the choice of the initial $\bfm_0$ is useful for all methods, and not only the $\chi^2$ method we show the same results as in Figure~\ref{fig2} but initialized with $\bfm_0=0$. In most cases the solutions that are obtained are less stable, indicating that the initial estimate is useful in constraining the results to reasonable values, however most noticeably not for the $\chi^2$ method, but for the MDP and L-curve algorithms. We also illustrate the results obtained after just one iteration in Figure~\ref{fig4} with the initial condition $\bfm_0$  according to Figure~\ref{fig2} to demonstrate the need for the iteration to generally stabilize the results. These results confirm the relative errors shown in Table~\ref{tab1} for averages of the errors over the $50$ cases. 

  \begin{figure}[!h] 
\begin{center}
\subfigure[UPRE: error $.3174$]{\label{UPa}\includegraphics[height=.07\textheight,width=.45\textwidth]{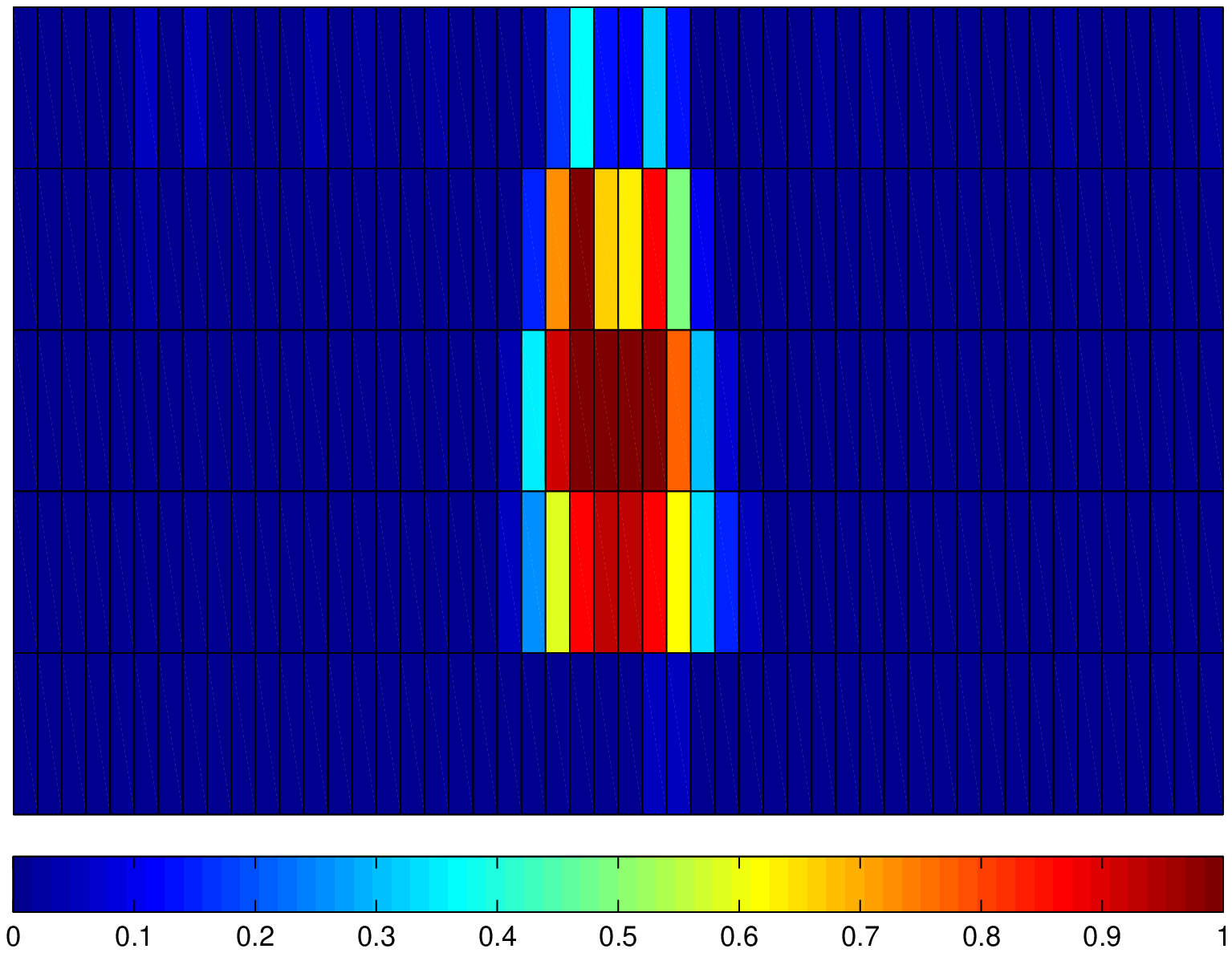}}
\subfigure[UPRE: error $.3240$]{\label{UPb}\includegraphics[height=.07\textheight,width=.45\textwidth]{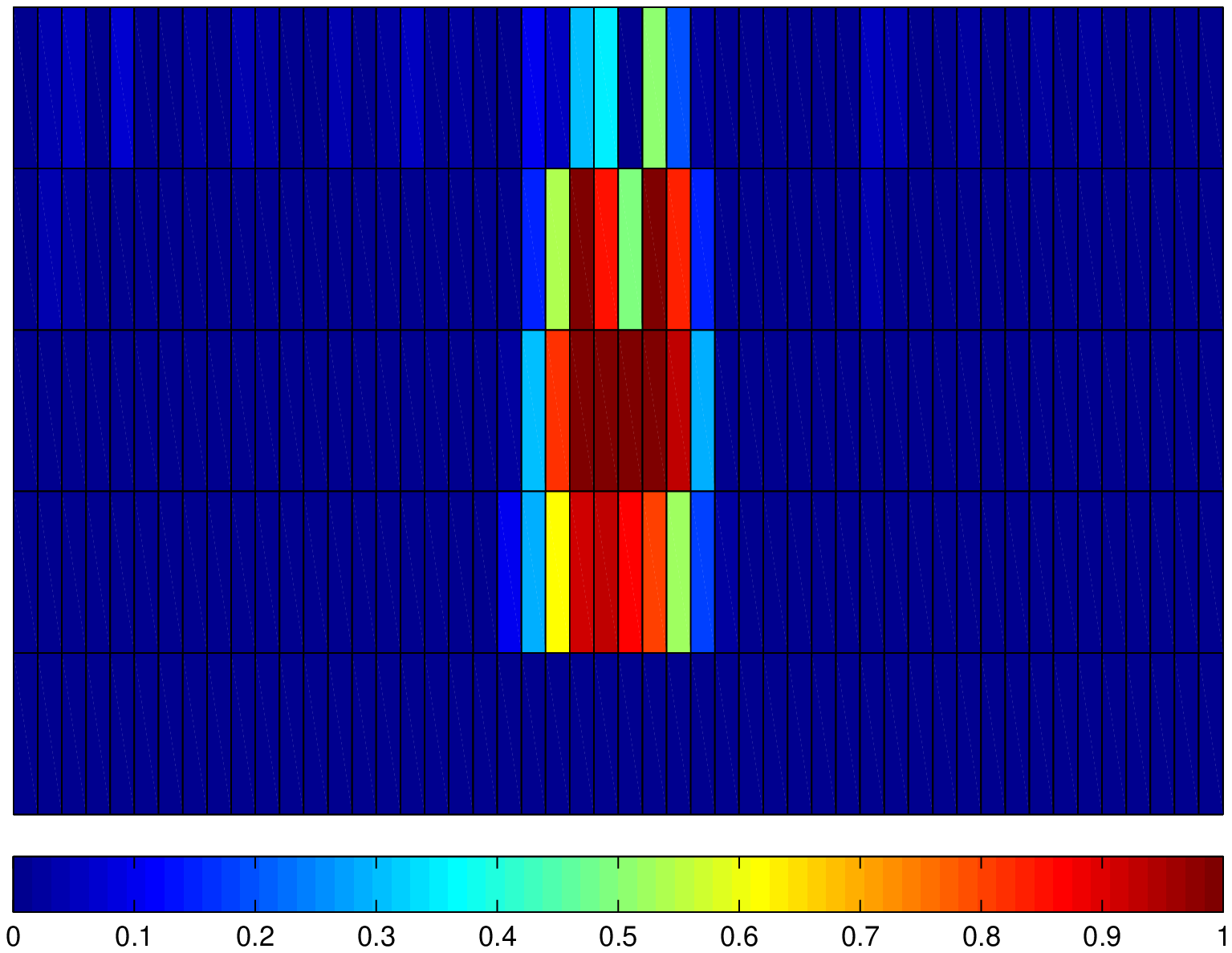}}
\subfigure[GCV: error $.3162$]{\label{GCVa}\includegraphics[height=.07\textheight,width=.45\textwidth]{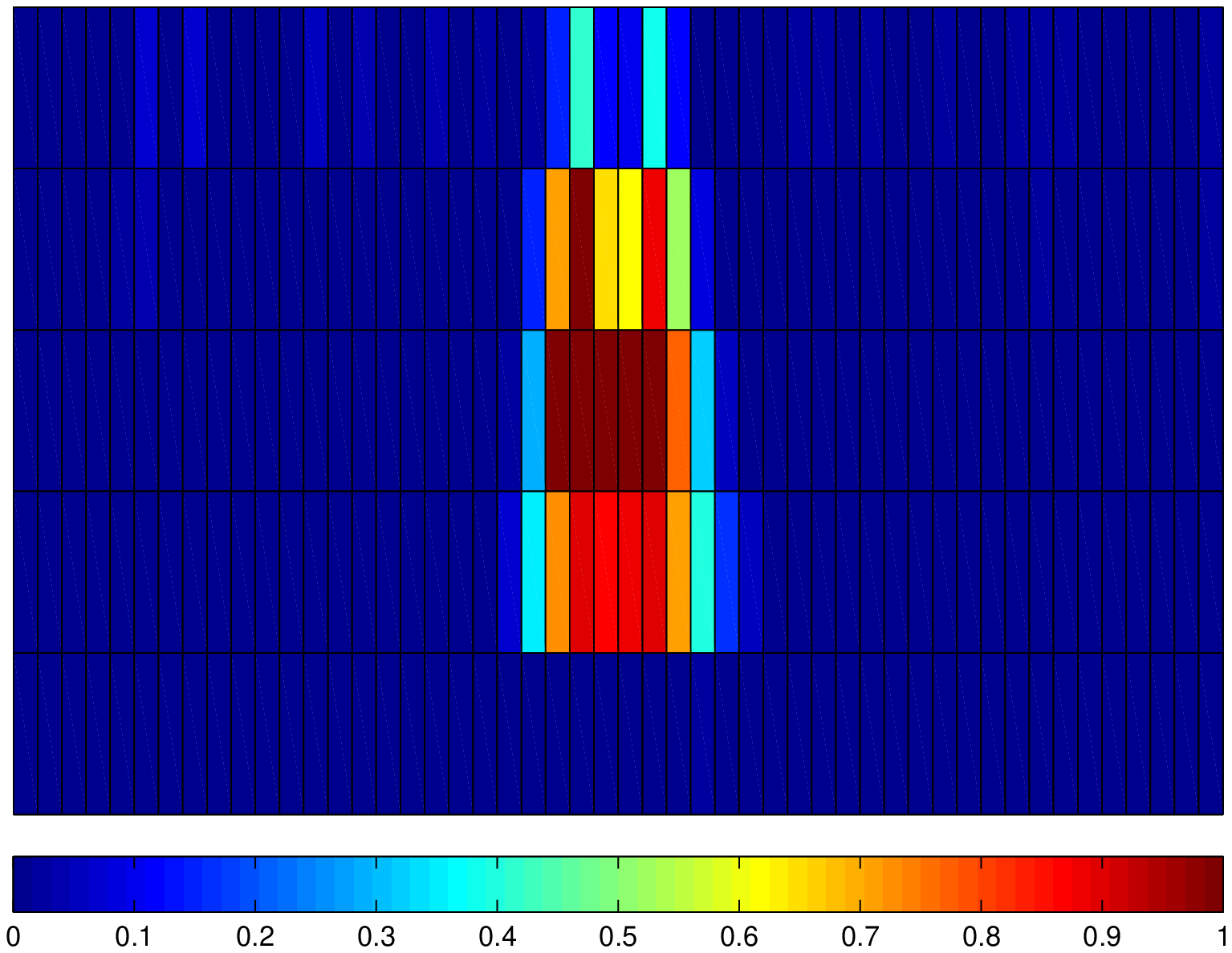}}
\subfigure[GCV: error $.3718$]{\label{GCVb}\includegraphics[height=.07\textheight,width=.45\textwidth]{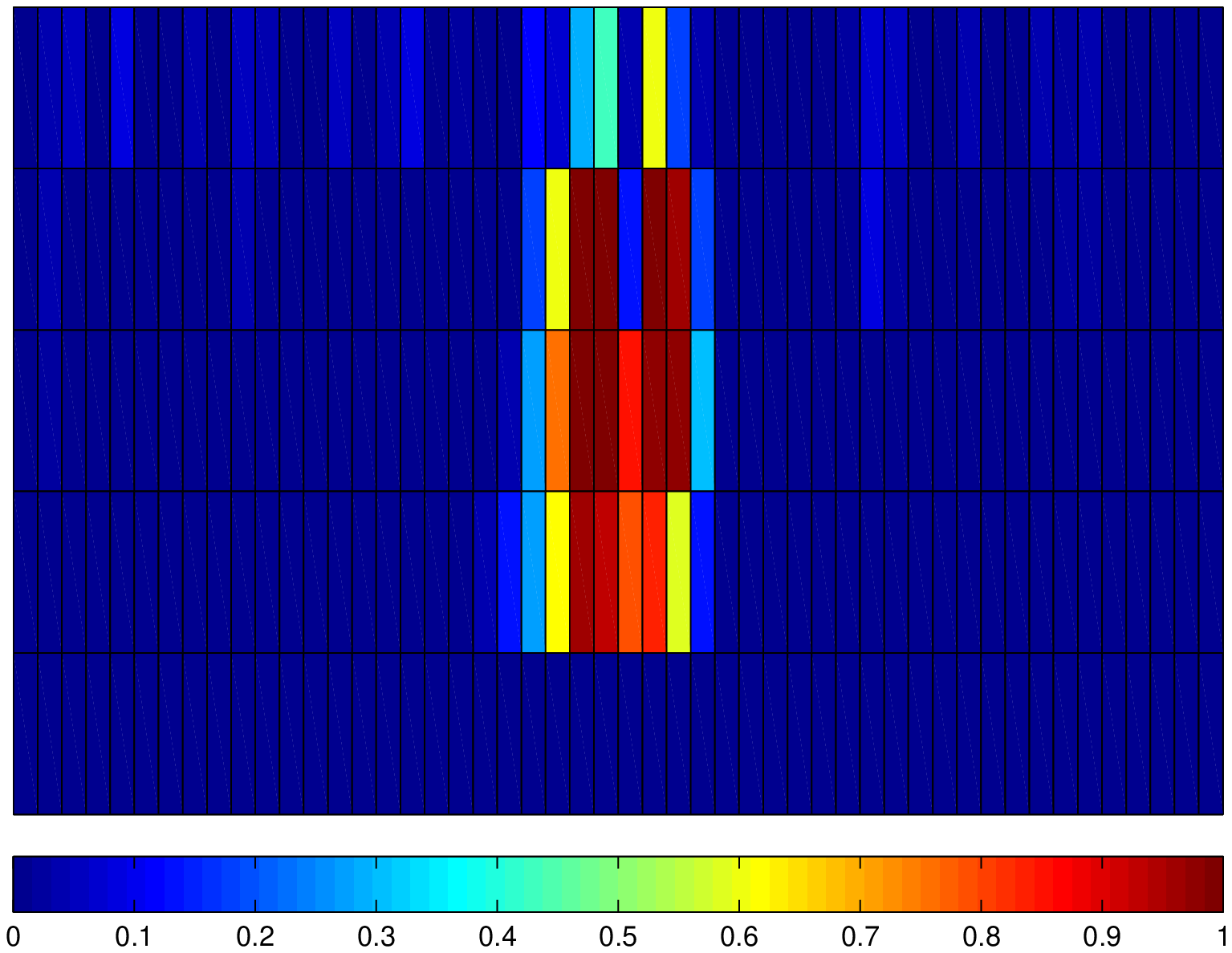}}
\subfigure[$\chi^2$: error $.3356$]{\label{Chia}\includegraphics[height=.07\textheight,width=.45\textwidth]{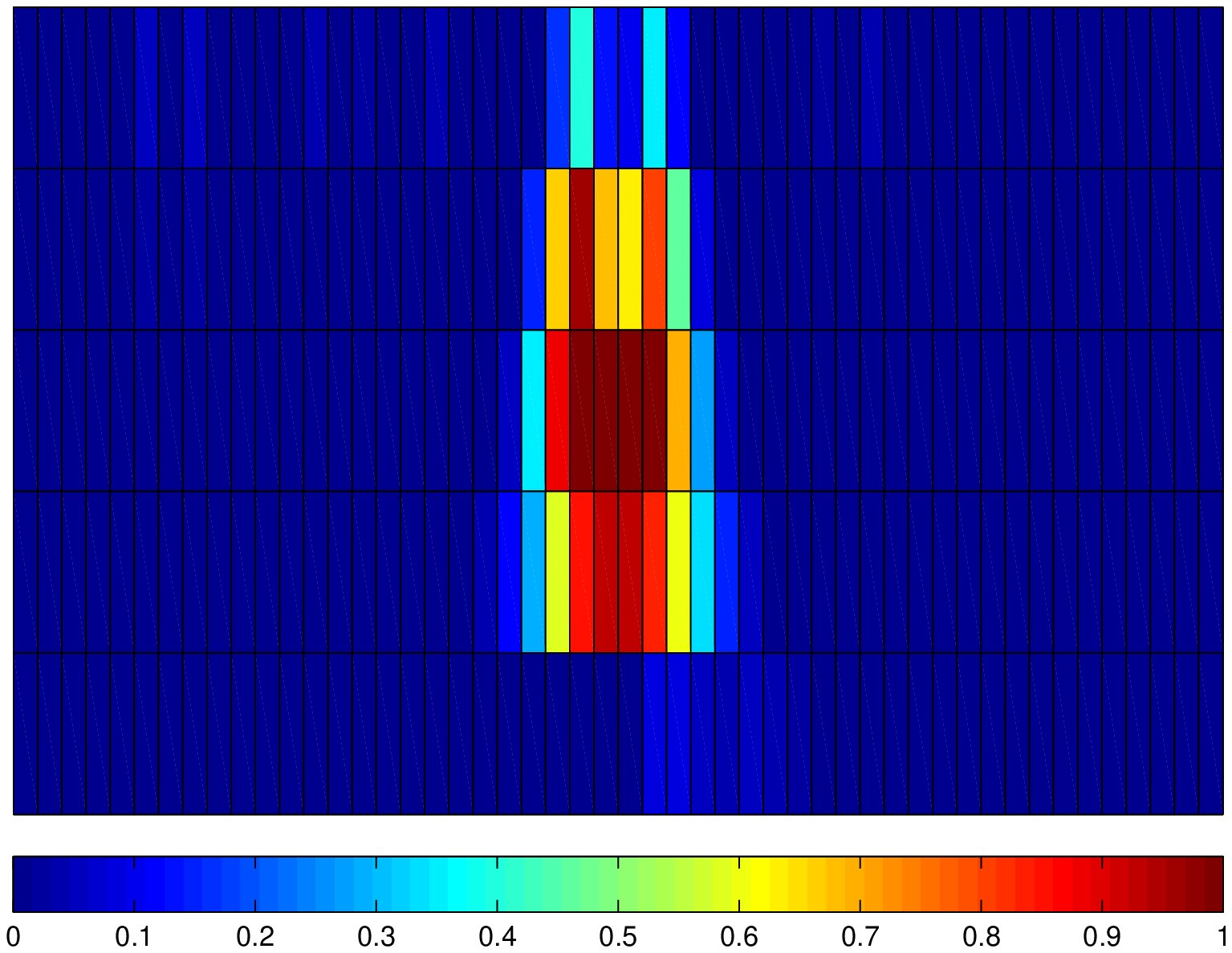}}
\subfigure[$\chi^2$: error $.3314$]{\label{Chib}\includegraphics[height=.07\textheight,width=.45\textwidth]{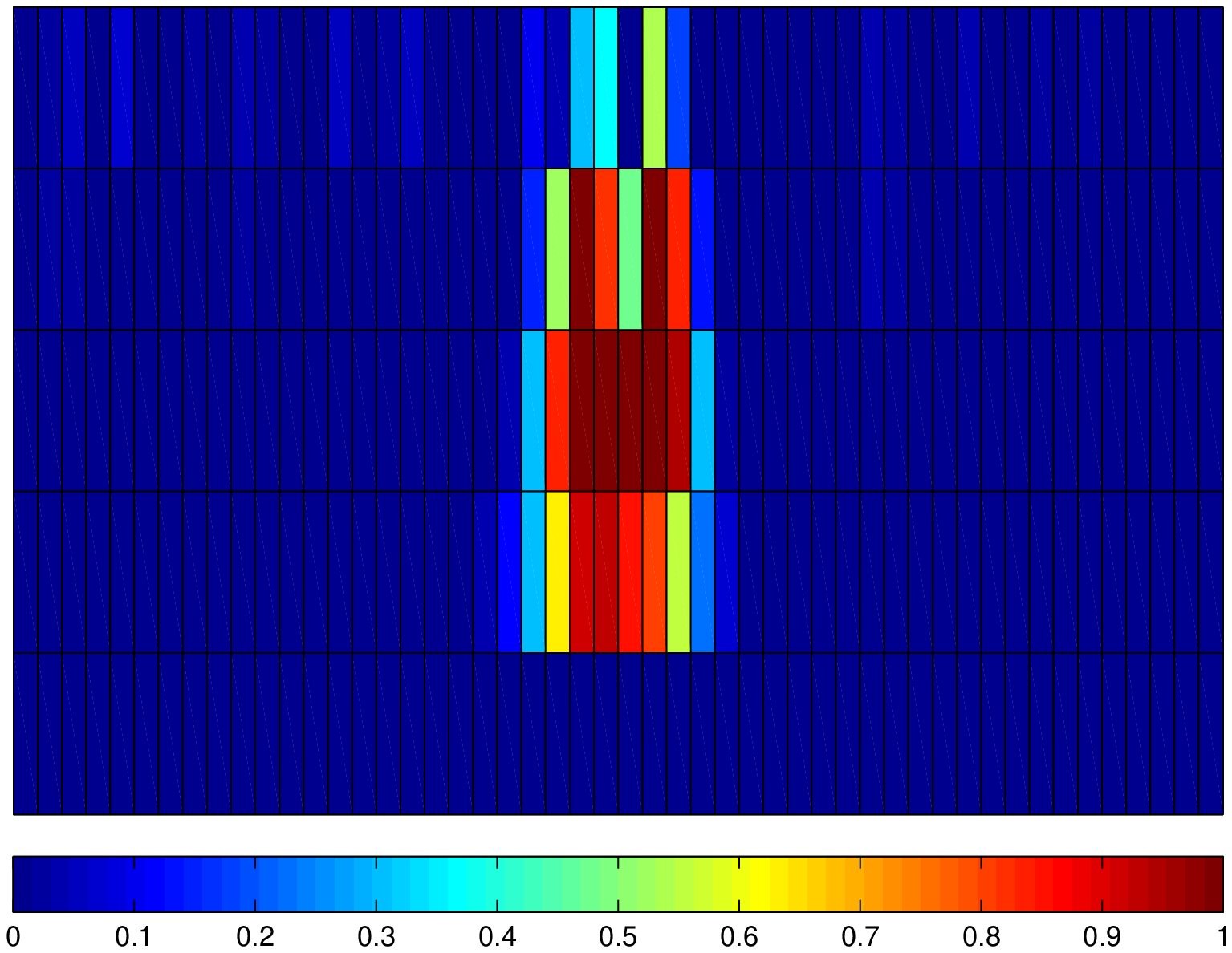}}
\subfigure[MDP: error $.4042$ ]{\label{MDPa}\includegraphics[height=.07\textheight,width=.45\textwidth]{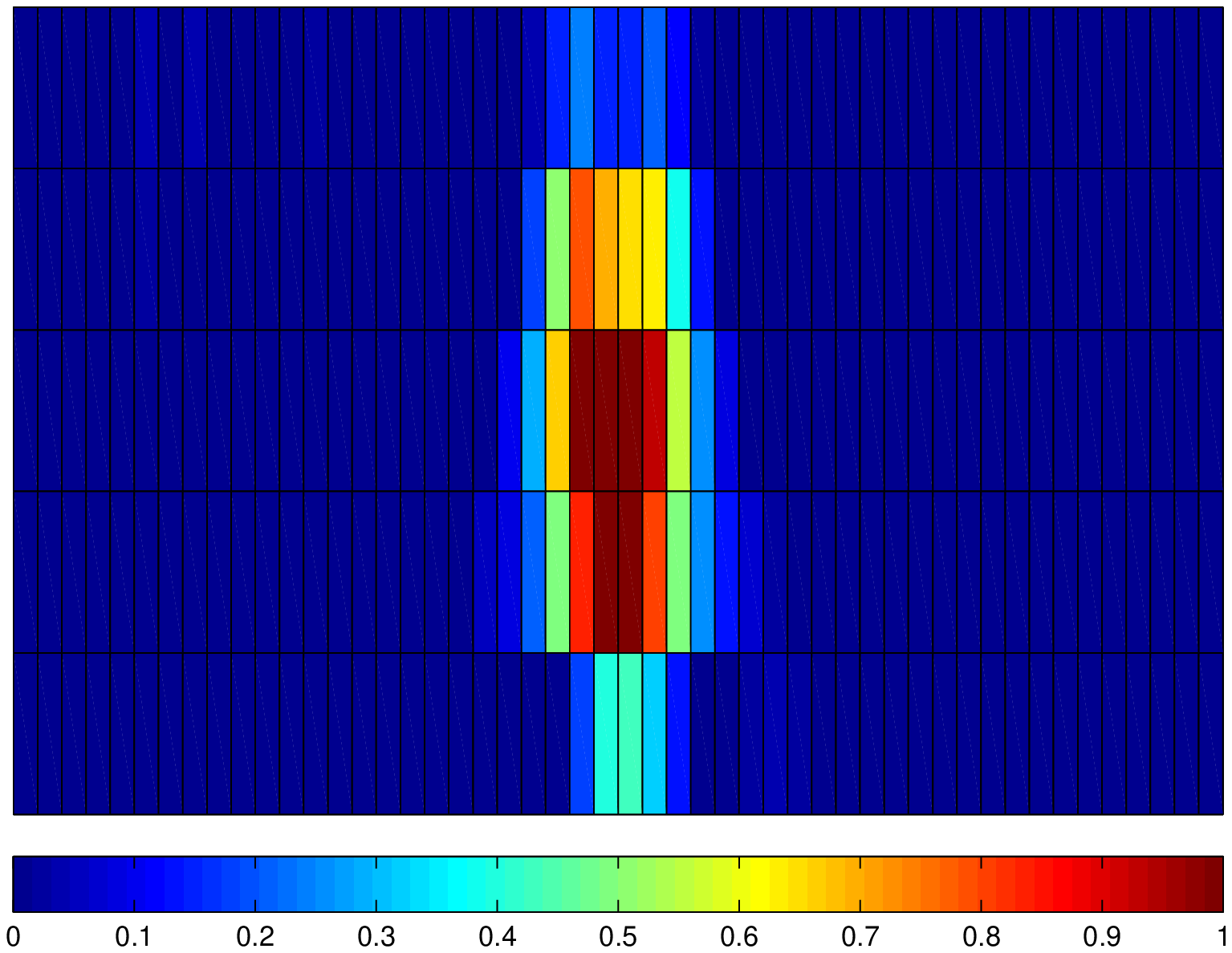}}
\subfigure[MDP: error $.3356$]{\label{MDPb}\includegraphics[height=.07\textheight,width=.45\textwidth]{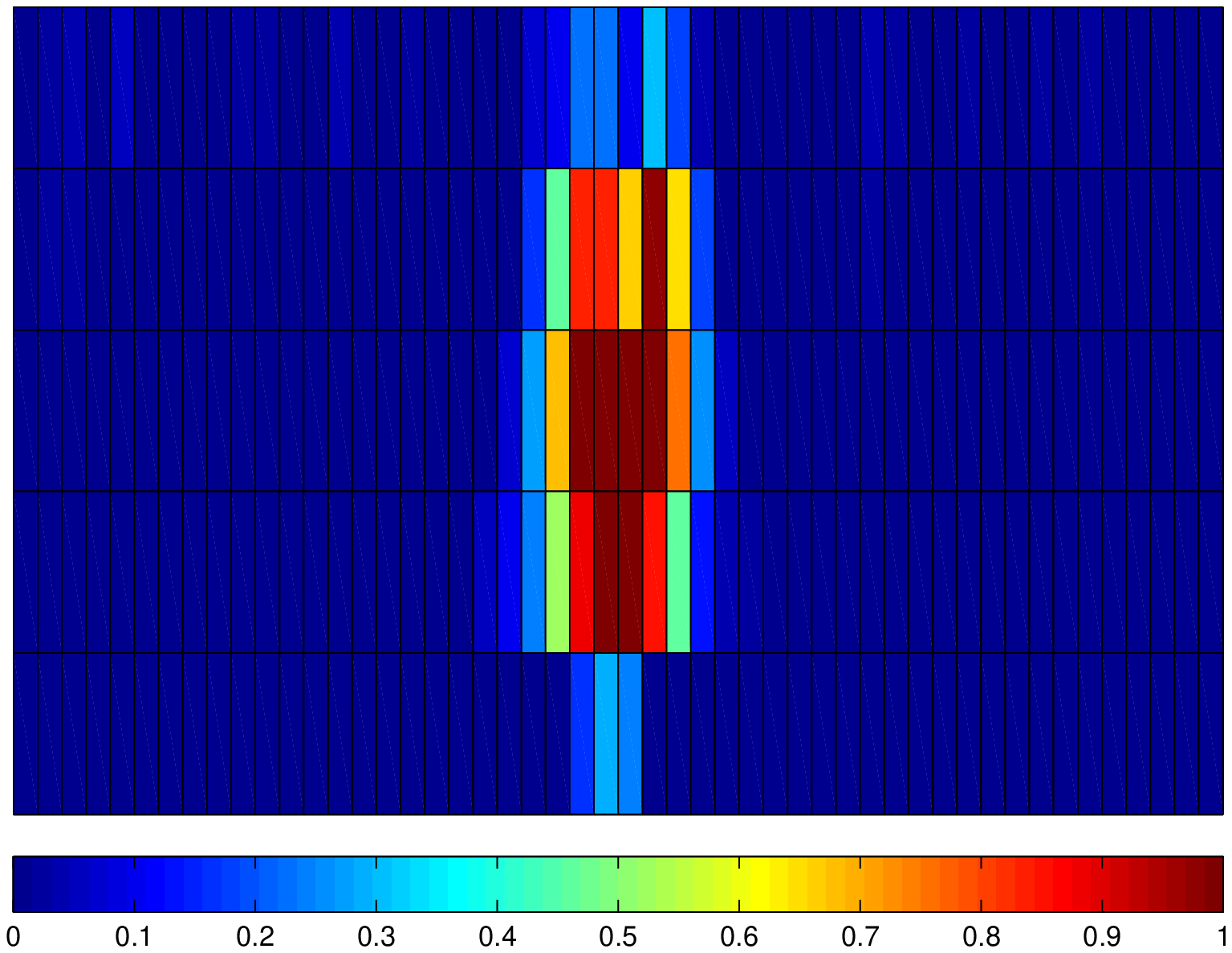}}
\subfigure[LC: error $.4420$]{\label{LCa}\includegraphics[height=.07\textheight,width=.45\textwidth]{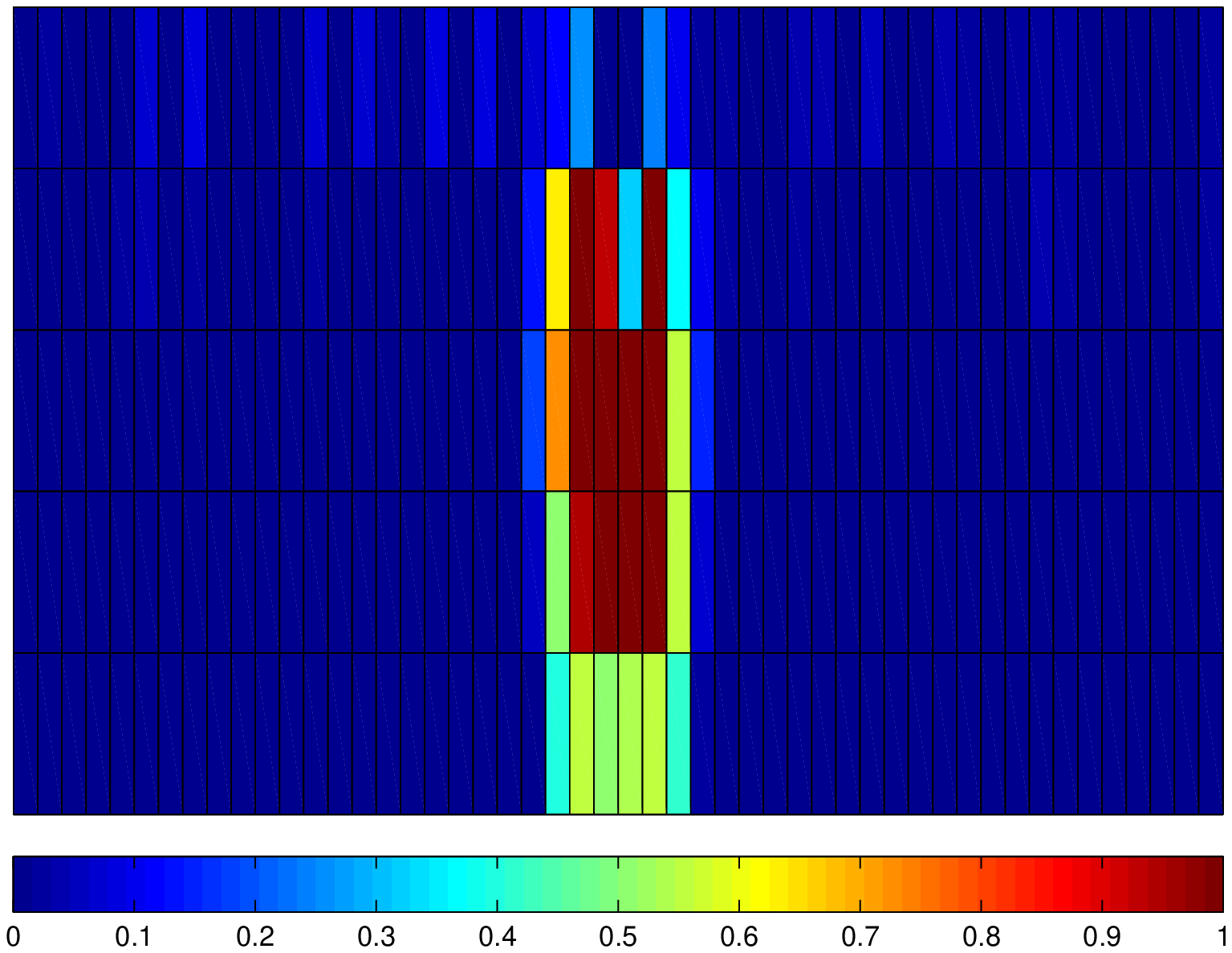}}
\subfigure[LC: error $.4555$]{\label{LCb}\includegraphics[height=.07\textheight,width=.45\textwidth]{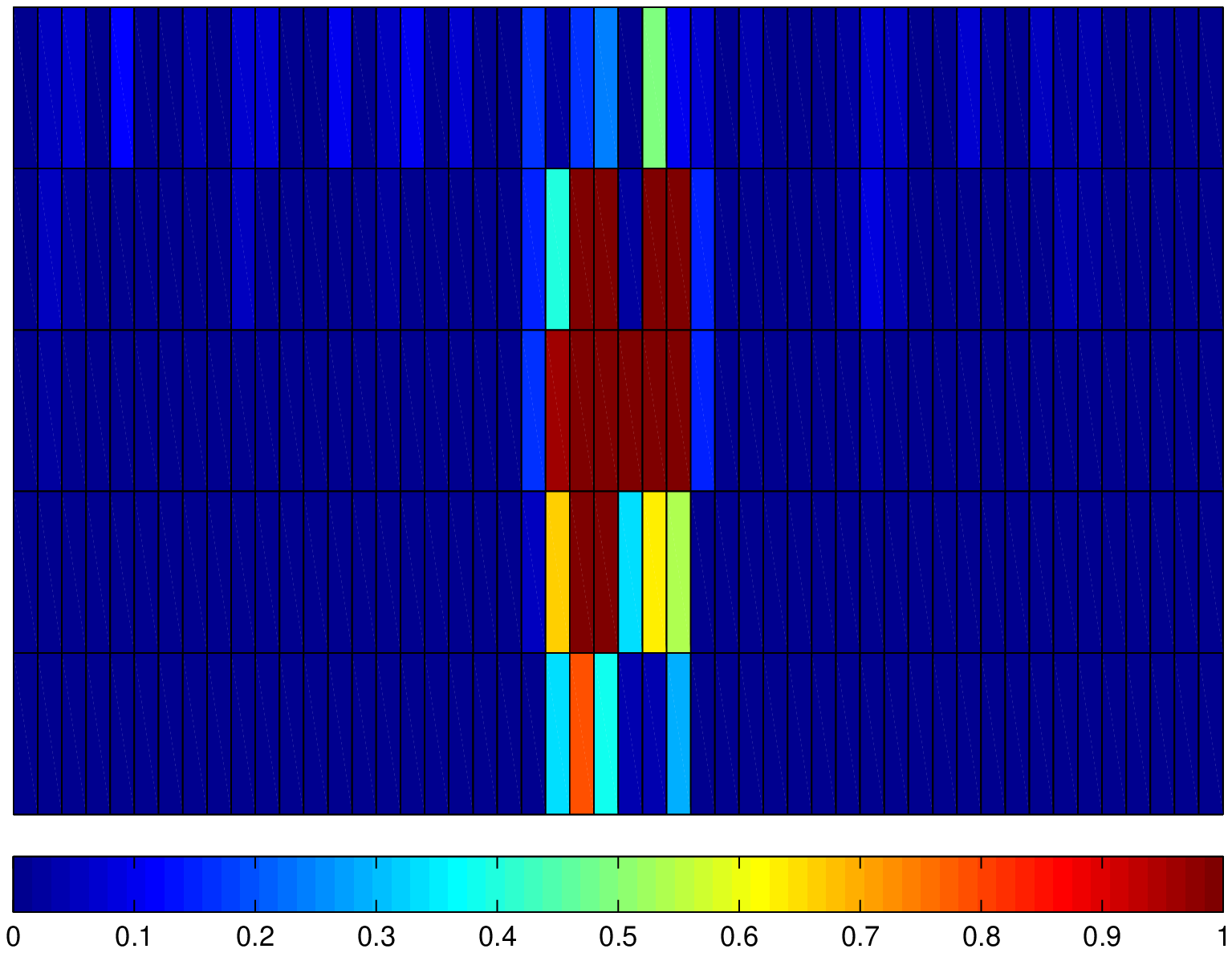}}

\caption{\label{fig3} Density model obtained from inverting the noise-contaminated data, as in Figure~\ref{fig2} except initialized with $\bfm_0=0$}
\end{center}
\end{figure} 

  \begin{figure}[!h] 
\begin{center}
\subfigure[UPRE: error $ .3330$]{\label{UPa}\includegraphics[height=.07\textheight,width=.45\textwidth]{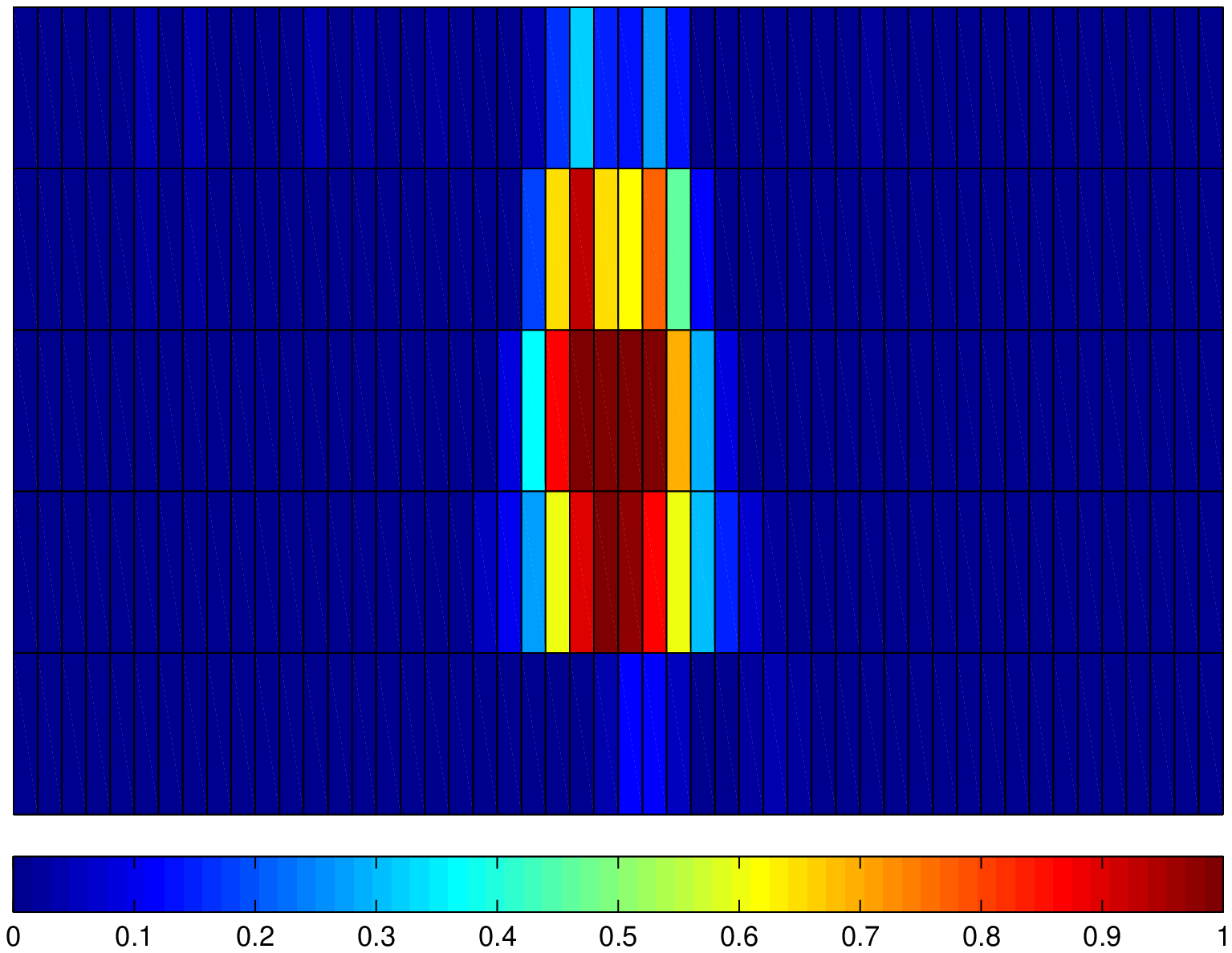}}
\subfigure[UPRE: error $ .3214$]{\label{UPb}\includegraphics[height=.07\textheight,width=.45\textwidth]{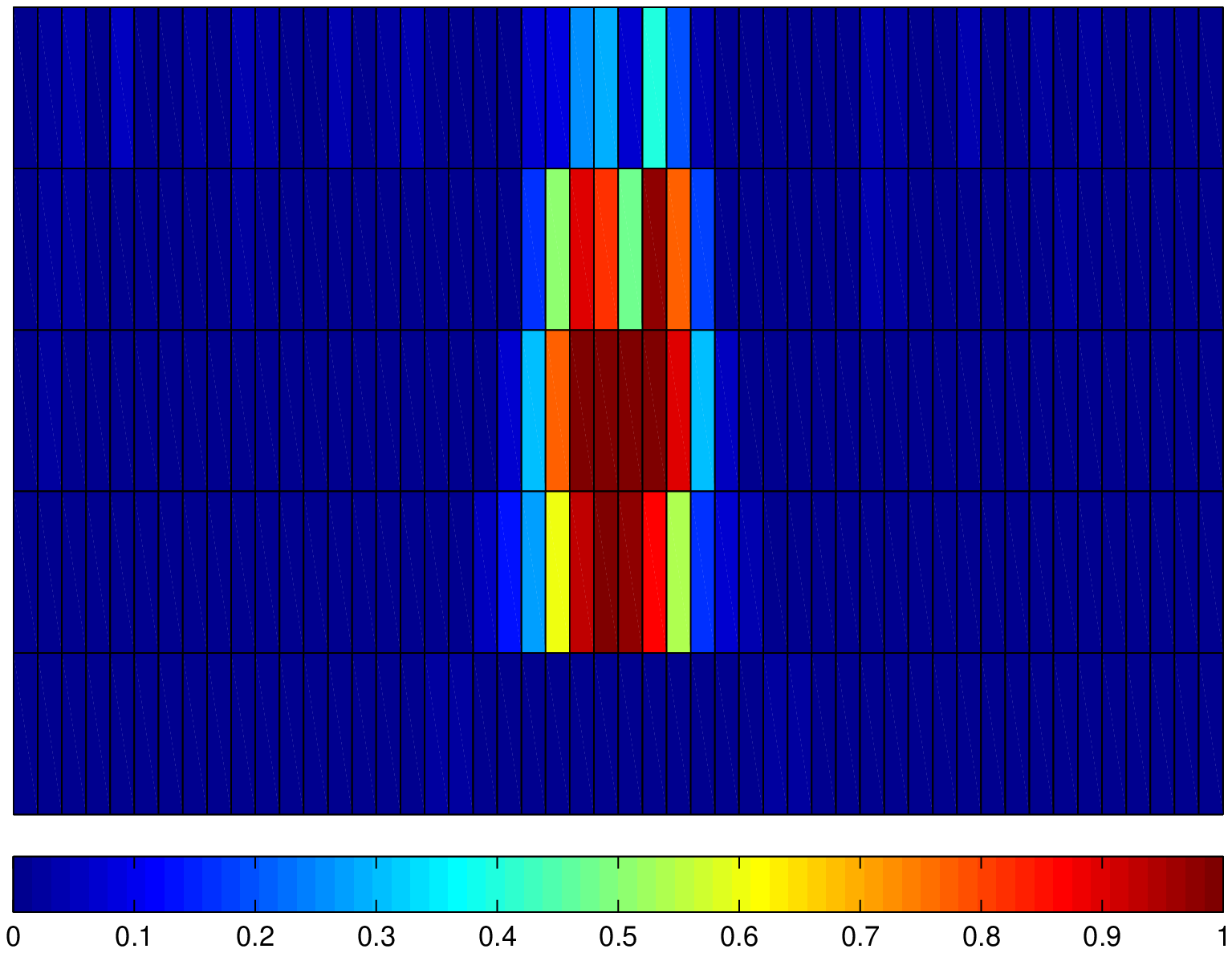}}
\subfigure[GCV: error $.3316$]{\label{GCVa}\includegraphics[height=.07\textheight,width=.45\textwidth]{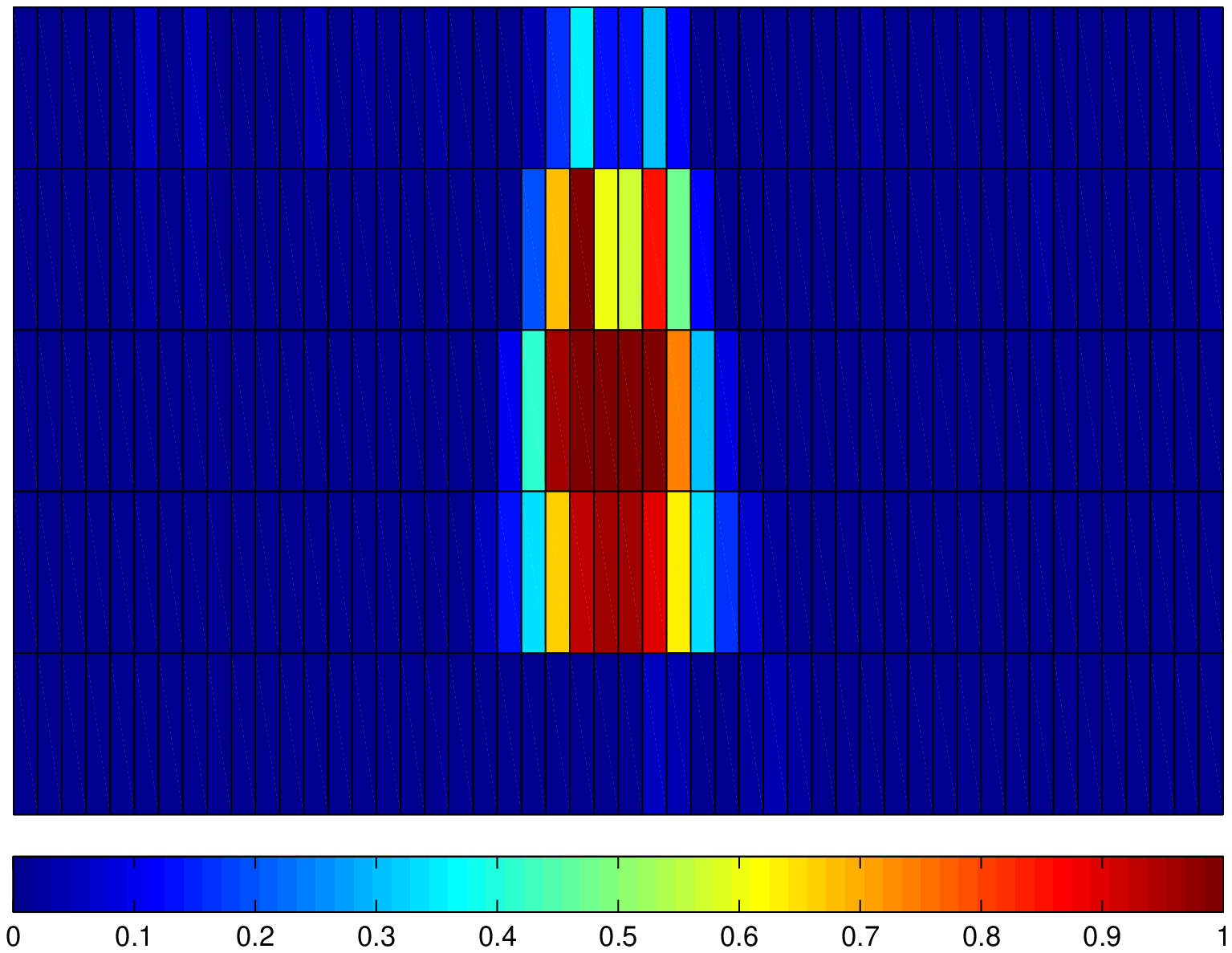}}
\subfigure[GCV: error $.3693$]{\label{GCVb}\includegraphics[height=.07\textheight,width=.45\textwidth]{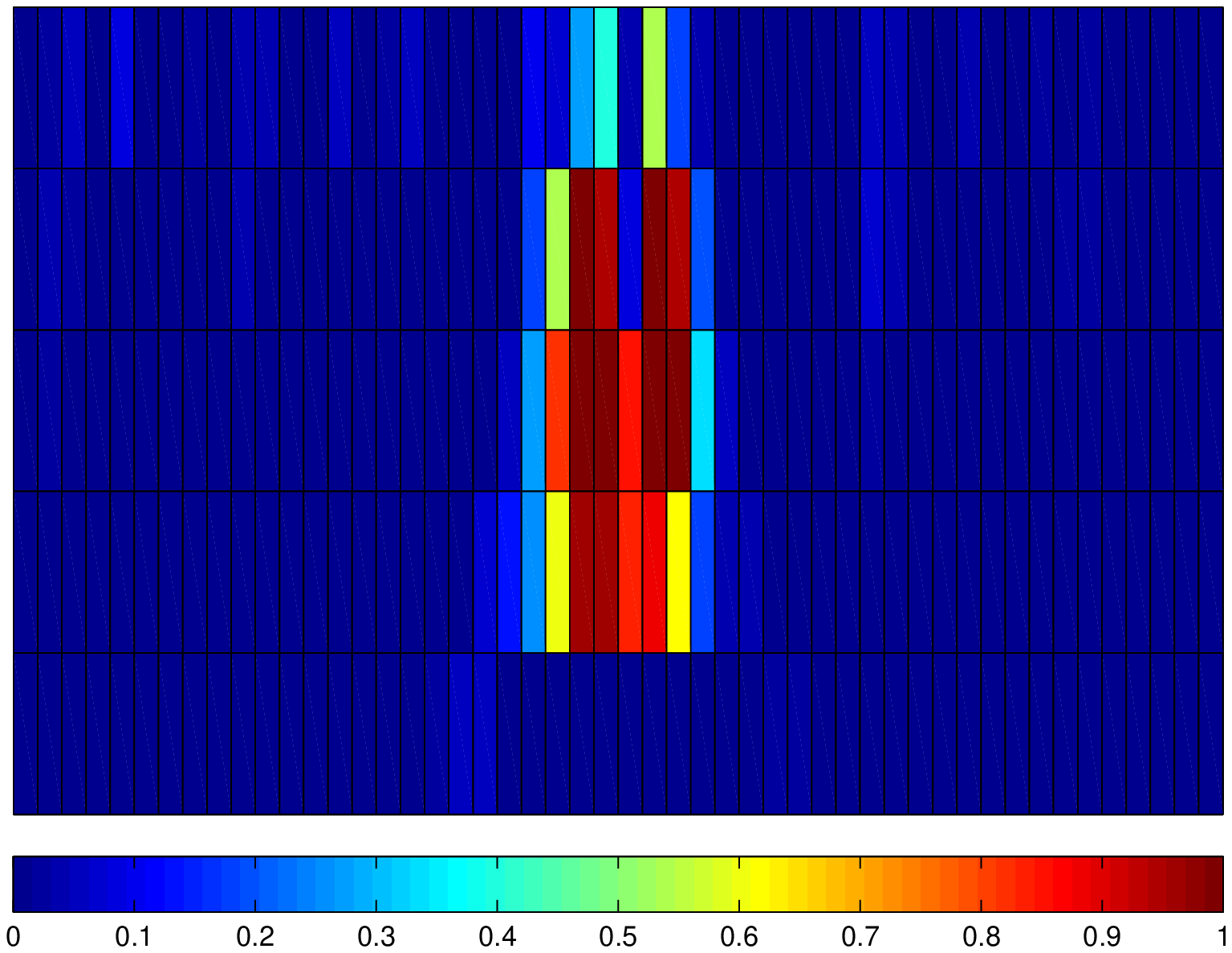}}
\subfigure[$\chi^2$: error $.3398$]{\label{Chia}\includegraphics[height=.07\textheight,width=.45\textwidth]{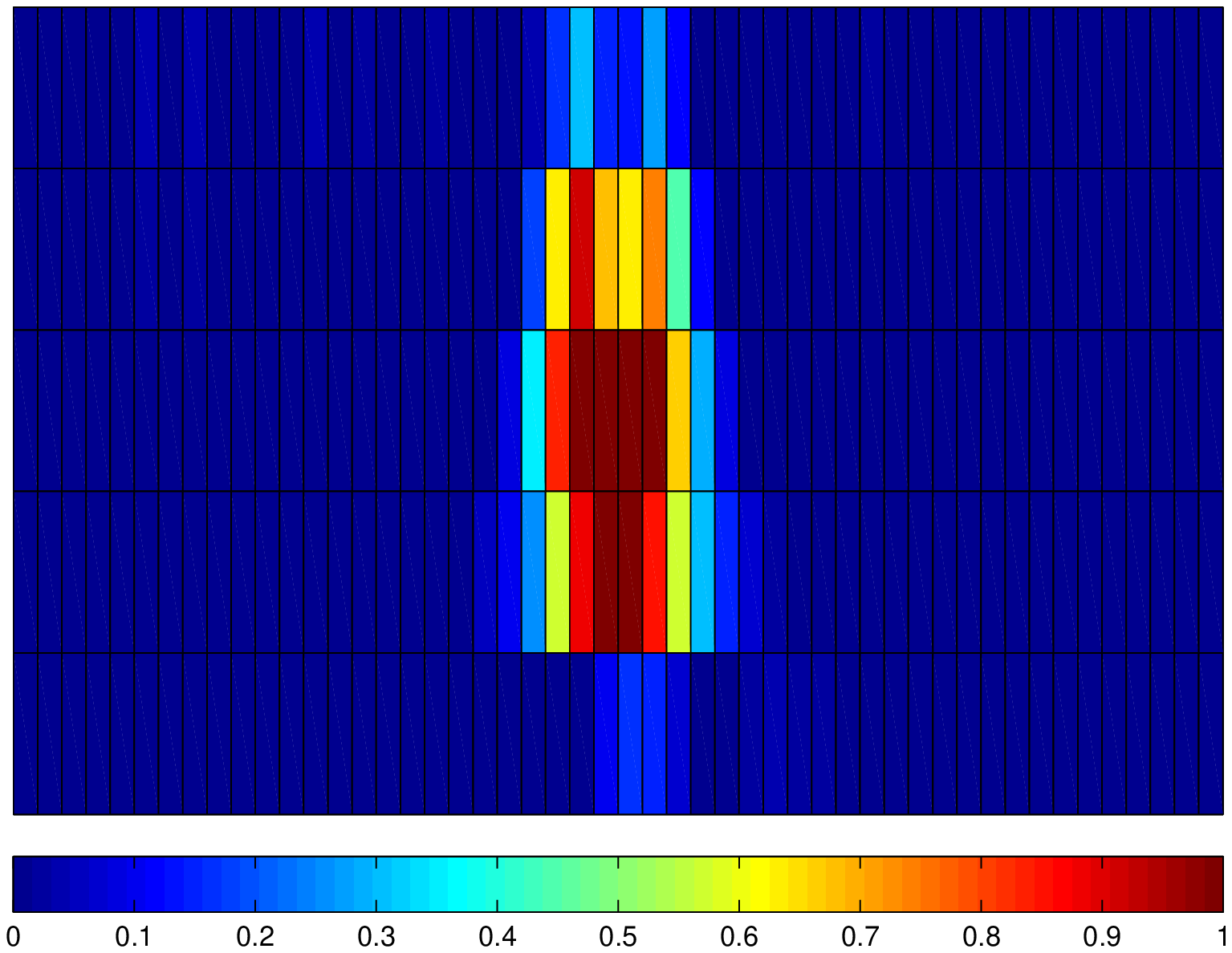}}
\subfigure[$\chi^2$: error $.3217$]{\label{Chib}\includegraphics[height=.07\textheight,width=.45\textwidth]{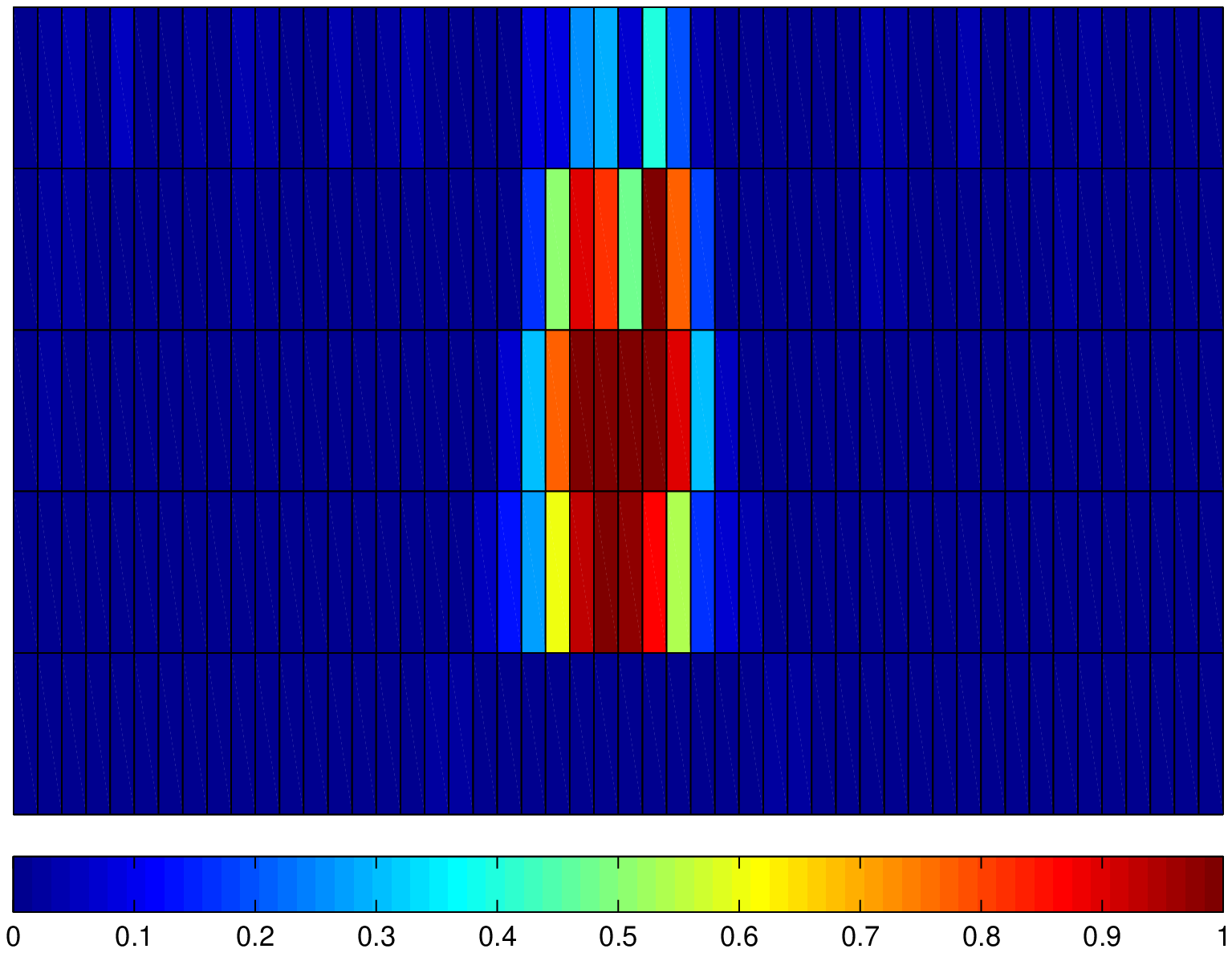}}
\subfigure[MDP: error $ .4006$ ]{\label{MDPa}\includegraphics[height=.07\textheight,width=.45\textwidth]{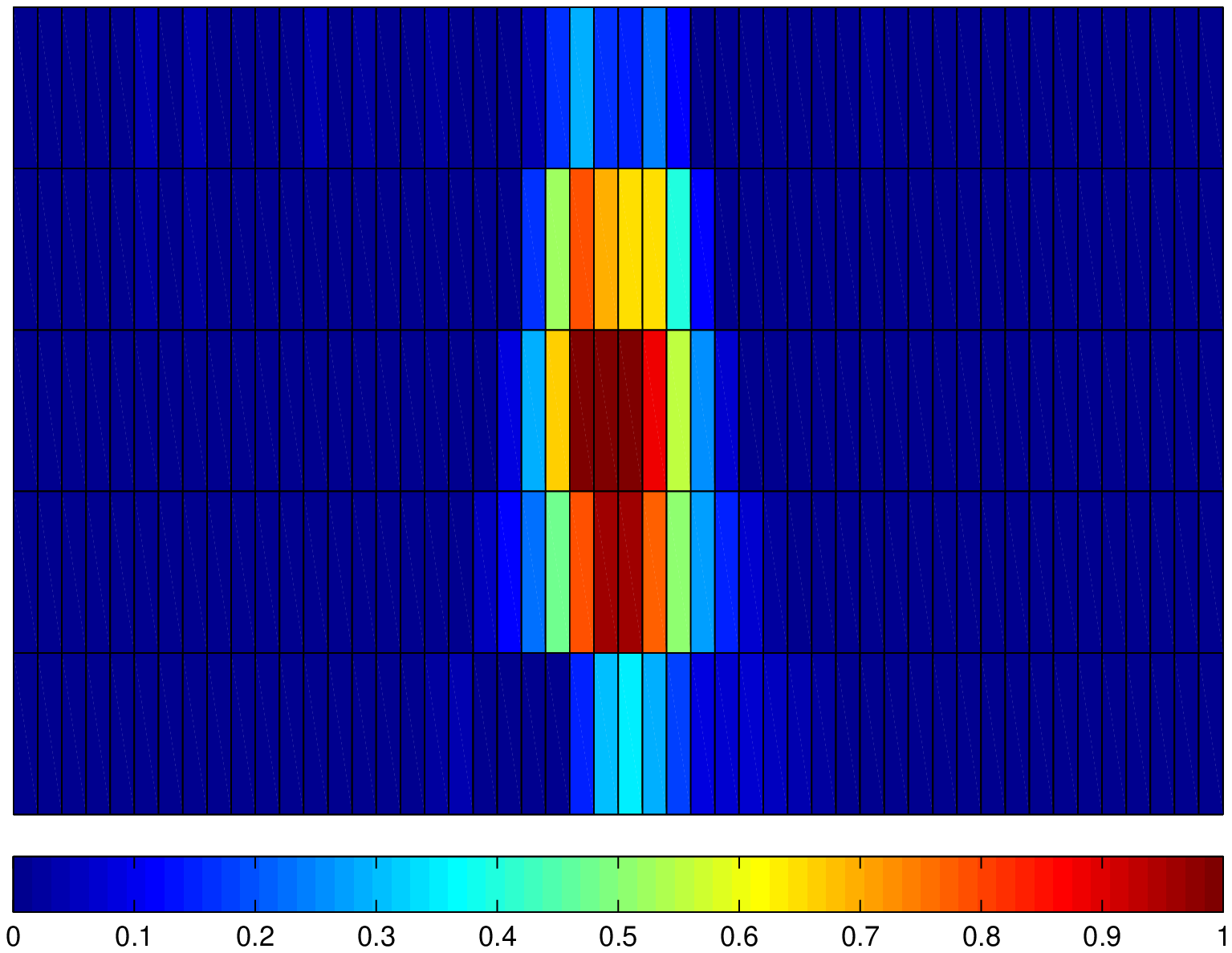}}
\subfigure[MDP: error $ .3458$]{\label{MDPb}\includegraphics[height=.07\textheight,width=.45\textwidth]{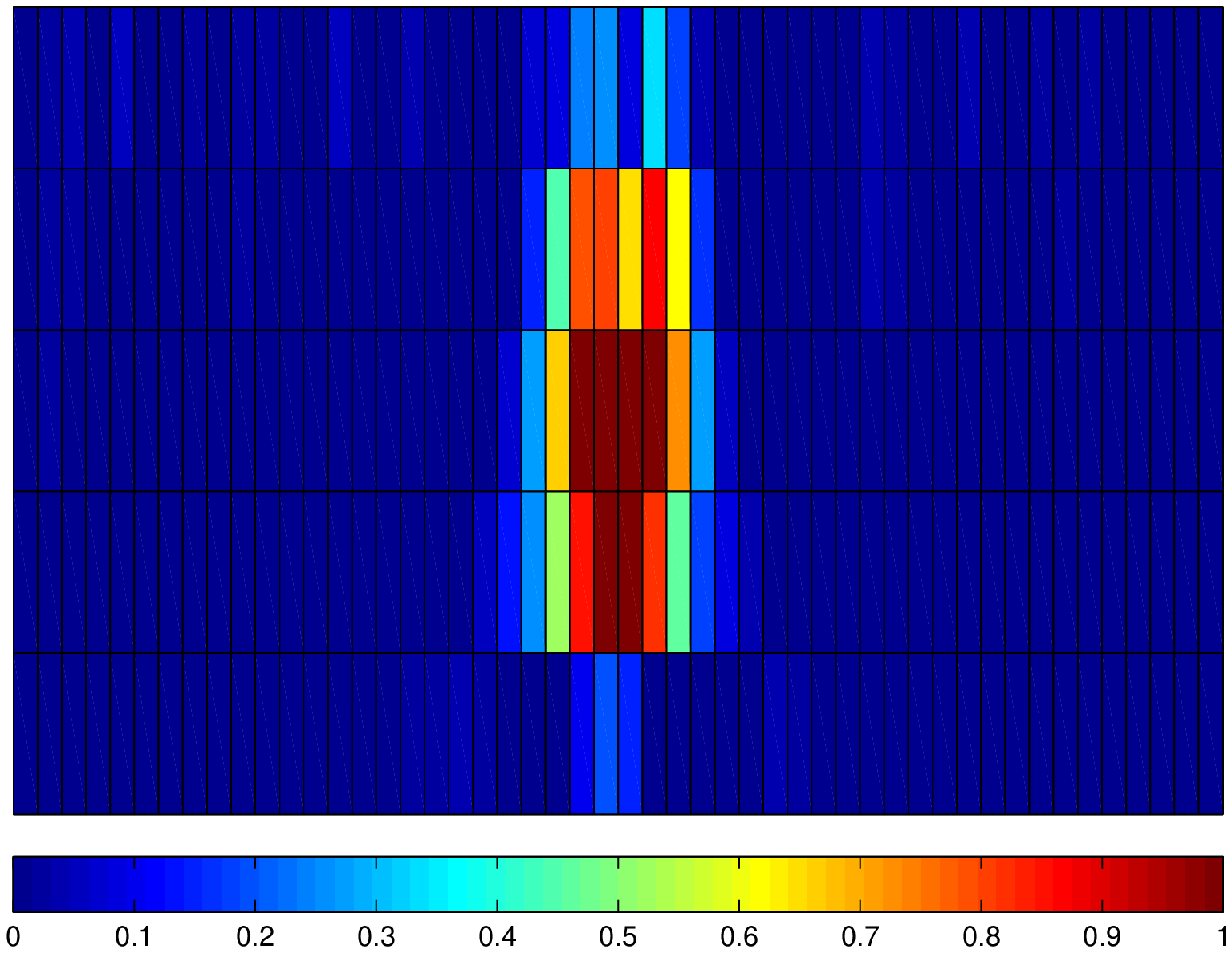}}
\subfigure[LC: error $.3299$]{\label{LCa}\includegraphics[height=.07\textheight,width=.45\textwidth]{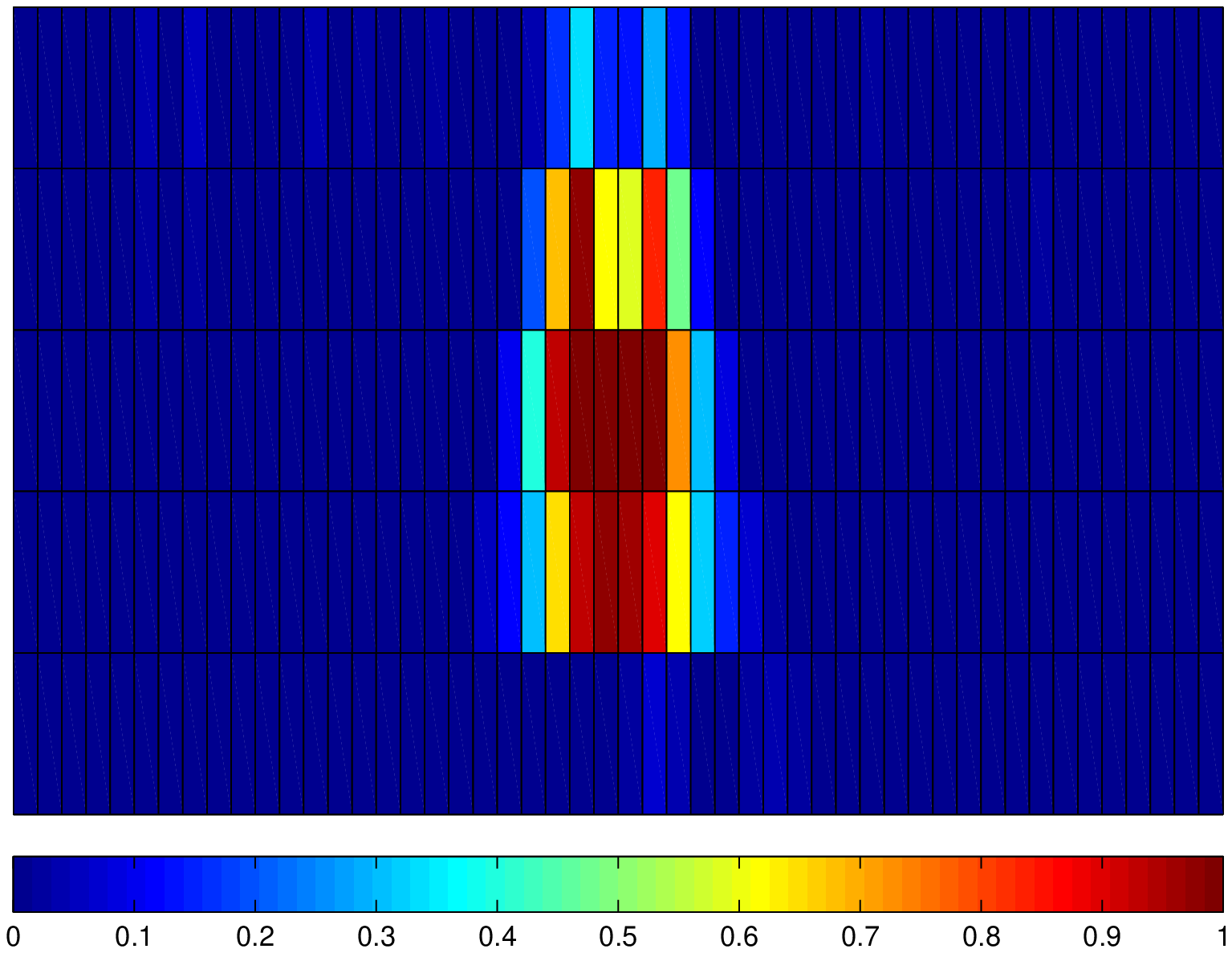}}
\subfigure[LC: error $.3970$]{\label{LCb}\includegraphics[height=.07\textheight,width=.45\textwidth]{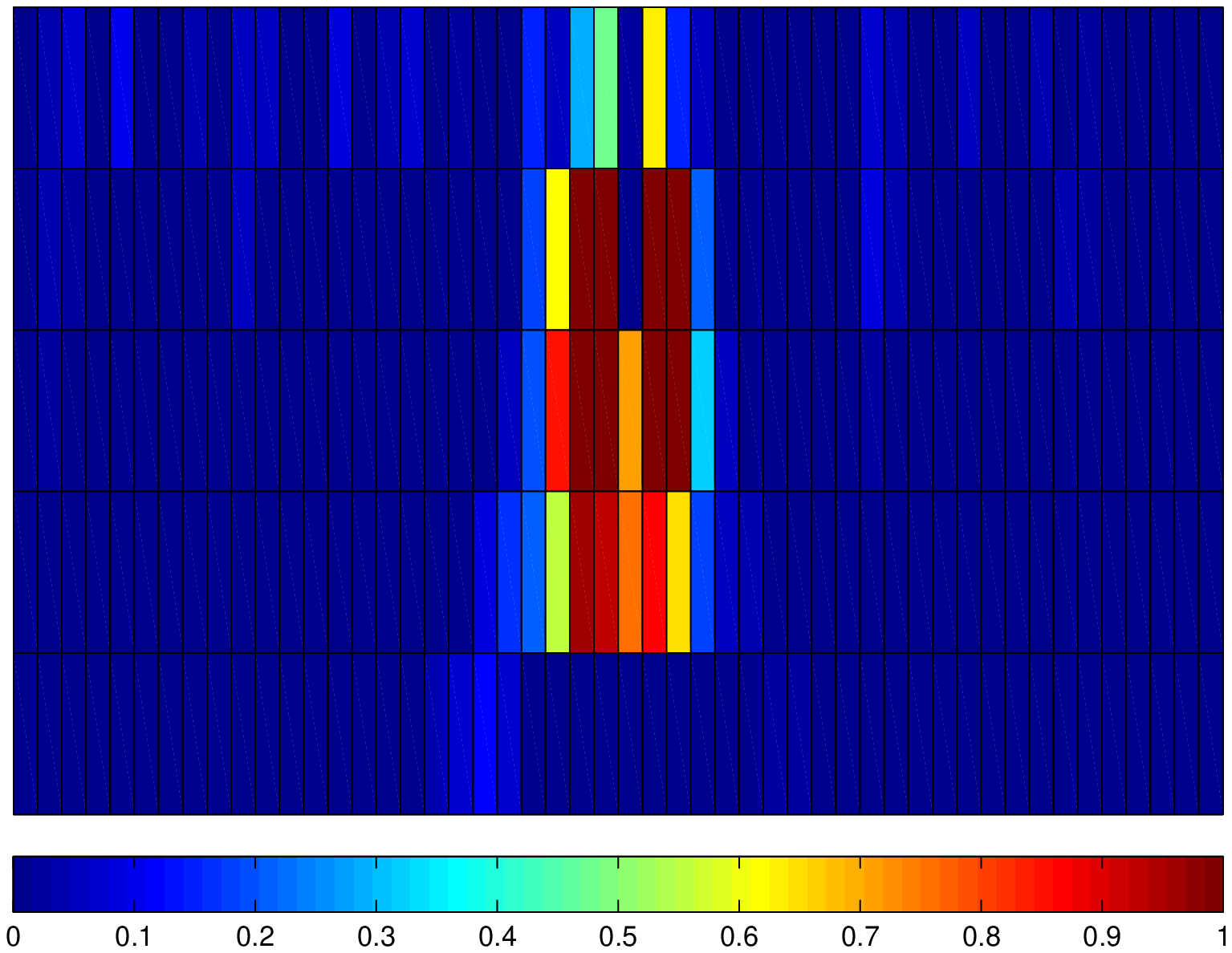}}
\caption{\label{fig4} Density model obtained from inverting the noise-contaminated data, as in Figure~\ref{fig2}  after just one step of the MS iteration. }
\end{center}
\end{figure} 

%
\section{Conclusions}\label{sec:conc}
The UPRE, GCV and $\chi^2$-principle algorithms  for estimating a regularization parameter in the context of underdetermined Tikhonov regularization have been developed and investigated,  extending the  $\chi^2$ method discussed  in \cite{mead:08,mead:13,MH13,mere:09,MR10}.  UPRE and $\chi^2$  techniques require that an estimate of the noise distribution in the data measurements is available, while ideally the $\chi^2$ also requires a prior estimate of the mean of the solution in order to apply the central version of the $\chi^2$ algorithm.  Results   demonstrate that   UPRE, GCV and $\chi^2$ techniques are useful for under sampled data sets, with UPRE and GCV yielding very consistent results. The $\chi^2$ is more useful in the context of the mapped problem where   prior information is not required. On the other hand,  we have shown that the use of the iterative MS stabilizer provides an effective alternative to the non-central algorithm  suggested in \cite{MR10} for the case without prior information. The UPRE, GCV and $\chi^2$ generally outperform L-curve and MDP methods to find the regularization parameter in the context of the iterative MS stabilizer for $2D$ gravity inversion. Moreover, with regard to efficiency the $\chi^2$ generally requires fewer iterations, and  is also cheaper to implement for each iteration because there is no  need to sweep through a large set of $\alpha$ values in order to find the optimal value.  These results are useful for the development of approaches for solving larger  $3D$ problems of gravity inversion, which will be investigated in future work. Then, the ideas have to be extended for iterative techniques replacing the SVD or GSVD for the solution.  

\ack  Rosemary Renaut acknowledges the support of AFOSR grant 025717: ``Development and Analysis of Non-Classical Numerical Approximation Methods", and 
NSF grant  DMS 1216559:   ``Novel Numerical Approximation Techniques for Non-Standard Sampling Regimes". She also notes conversations with Professor J. Mead  concerning the extension of the $\chi^2$-principle to the underdetermined situation presented here. 
\appendix
\section{Parameter Estimation Formulae}\label{App}
\setcounter{section}{1}
 We assume that the matrices and data are pre weighted by the covariance of the data, and thus use the GSVD of Lemma~\ref{gsvdlemma}  for the matrix pair $[\tilde{G};L]$. We also 
introduce inclusive notation for the limits of the summations, that are correct for all choices of $(m,n,p,r)$, where $r\le \min(m,n)$  determines filtering of the least $p-r-\tilde{q}$ singular values $\gamma_i$, $\tilde{q}=\max(n-m,0)$.
 Then  $\bfm(\sigmam) =\bfmo+\bfy(\sigmam)$ is obtained for  
 \begin{align}\label{solny}
\bfy(\sigmam) &= \sum_{i=\tilde{q}+1}^p \frac{\nu_i}{\nu_i^2 +\sigmam^{-2} \mu_i^2} s_i \bfzi +  \sum_{i=p+1}^{n} s_i \bfzi=  \sum_{i=q+1}^p f_i\frac{s_i}{\nu_i} \bfzi + \sum_{i=p+1}^n s_i \bfzi,   
\end{align}
where $Z:=(X^T)^{-1}=[\bfz_1, \dots, \bfz_n]$, ,$f_i=\left(\frac{\gamma_i^2}{\gamma_i^2 +\sigmam^{-2}}\right) $ are the filter factors and $s_i=\bfu_{i-\tilde{q}}^T\tilde{\bfr}$, $s_i=0$, $i<q$. Orthogonal matrix $V$ replaces $(X^T)^{-1}$ and $\sigma_i$ replaces $\gamma_i$, when applied for the singular value decomposition $\tilde{G}=U\Sigma V^T$ with $L=I$.

Let   $\tilde{s}_i(\sigmam)={s_i}/{(\gamma_i^2\sigmam^2+1)}$, 
 and note the filter factors with truncation are given by
\begin{align}
f_i=\left\{\begin{array}{ll}0 &\tilde{q}+1\le i\le p-r\\ \frac{\gamma_i^2}{\gamma_i^2+\sigmam^{-2}} & p-r+1\le i \le p \\ 1 & p+1\le n\end{array}\right. \quad (1-f_i)=\left\{\begin{array}{ll} 1& \tilde{q}+1\le i\le p-r\\ \frac{1}{\gamma_i^2 \sigmam^{2}+1} & p-r+1\le i \le p \\ 0 & p+1\le n\end{array}\right. . \end{align}
Then, with the assumption that if a lower limit is lower than a higher limit on a sum  the contribution is $0$, 
\begin{align}\nonumber
  \mathrm{trace}(I_m-G(\sigmam)) &=m-\sum_{i=\tilde{q}+1}^{\min(n,m)} f_i  =(m-(n-(p-r)))+\sum_{i=p-r+1}^{\min(n,m)}(1-f_i) \\&= (m+p-n-r)+\sum_{i=p-r+1}^p \frac{1}{\gamma_i^2 \sigmam^{2}+1}:=T(\sigmam) \label{trace}\\
 \| (I_m-G(\sigmam))\tilde{\bfr}\|_2^2  &=   \sum_{i=p-r+1}^p(1-f_i)^2 s_i^2+\sum_{i=n+1}^m s_i^2+\sum_{i=\tilde{q}+1}^{p-r}s_i^2\\&= 
 \sum_{i=p-r+1}^p  \tilde{s}_i^2(\sigmam)+\sum_{i=n+1}^m s_i^2+\sum_{i=\tilde{q}+1}^{p-r}s_i^2:=N(\sigmam). \label{norm}
\end{align}
Therefore we seek in each case $\sigmam$ as the root, minimum or corner of a given function.

\begin{description}
\item[UPRE:] Minimizing
$( \| \tilde{G}\bfy(\sigmam)-\tilde{\bfr}\|_2^2 +2\,\mathrm{trace}(G(\sigmam)) - m)$ we may shift by constant terms and minimize
\begin{align}U(\sigmam)&= \sum_{i=p-r+1}^p(1-f_i)^2s_i^2+2\sum_{i=p-r+1}^p(f_i-1) = \sum_{i=p-r+1}^p \tilde{s}_i^2  -2 \sum_{i=p-r+1}^p \frac{1}{\gamma_i^2\sigmam^2+1}.\label{upre}\end{align}
\item[GCV:] Minimize 
\begin{align}  GCV(\sigmam) &= \frac{\|\tilde{G}\bfy(\sigmam)-\tilde{\bfr}\|_2^2}{\mathrm{trace}(I_m-G(\sigmam))^2} 
 = 
 \frac{N(\sigmam)}{T^2(\sigmam)}\label{gcv}\end{align}
\item[$\chi^2$-principle] The iteration to find $\sigmam$ requires 
 \begin{align}
\|\bfk(\sigmam)\|_2^2&=\sum_{i=q+1}^{p} \frac{s_i^2}{\gamma_i^2\sigmam^{2}+1}, \, \frac{\partial{\|\bfk(\sigmam)\|_2^2}}{\partial \sigmam} = -2\sigmam\sum_{i=q+1}^{p} \frac{\gamma_i^2 s_i^2}{(\gamma_i^2\sigmam^{2}+1)^2}=-\frac{2}{\sigmam^3}\|L\bfy(\sigmam)\|_2^2, 
\end{align}
and  with a search parameter $\beta^{(j)}$  uses the Newton iteration
\begin{align}
\label{newt}
\sigma^{(j+1)} = \sigma^{(j)}\left( 1 + \beta^{(j)}\frac12 \left(\frac{\sigma^{(j)}}{\|L\bfy(\sigma^{(j)})\|_2}\right)^2(\|\bfk(\sigma^{(j)})\|_2^2-(m+p-n))\right).
\end{align}
This iteration holds for the filtered case by defining $\gamma_i=0$ for $q+1\le i \le p-r$, removing the constant terms in \eqref{kfilter} and using $r$ degrees of freedom, \cite{RHM10}.
\item[MDP]: For $0<\rho\le 1$ and $\delta=m$, solve  \begin{align}\label{MDPEq} \| (I_m-G(\sigmam))\tilde{\bfr}\|_2^2= N(\sigmam)=\rho \delta .\end{align}
\item[L-curve:] Determine the corner of the  \texttt{log-log} plot of $\|L\bfy\|_2$ against $\|\tilde{G}\bfy(\sigmam)-\tilde{\bfr}\|_2$,  namely the corner of the curve parameterized by 
$$\left( \sqrt{N(\sigmam)}
, \sigmam^2\sqrt{\sum_{i=p-r+1}^p\frac{\gamma_i^2s_i^2}{(\gamma_i^2\sigmam^2+1)^2}}\right).$$
\end{description}

\end{document}